\documentclass[11pt]{amsart}
\usepackage{tikz}

\usepackage{amssymb,amsfonts,amsthm,amsmath,enumerate,amscd}
\usepackage{graphicx}
\usepackage{color, colordvi}
\usepackage{times}
\usepackage[english]{babel}
\usepackage{comment}

\newtheorem*{prob}{Problem}

\newtheorem{Lem}{Lemma}[section]
\newtheorem{Cor}[Lem]{Corollary}
\newtheorem{Def}[Lem]{Definition}
\newtheorem{Teo}[Lem]{Theorem}
\newtheorem{Prop}[Lem]{Proposition}

\newtheorem{remark}[Lem]{Remark}

\theoremstyle{definition} 
\newtheorem{Exam}[Lem]{Example}

\newcommand{\ord}{{\rm ord}}

\newcommand{\Nod}{{\rm Nod}}
\newcommand{\Seg}{{\rm Seg}}
\newcommand{\Abn}{{\rm Abn}}
\newcommand{\tord}{{\rm tord}}
\newcommand{\itord}{{\rm itord}}
\newcommand{\Lsing}{{\rm Lsing}}
\newcommand{\Nor}{{\rm Nor}}
\newcommand{\Spec}{{\rm Spec}}
\newcommand{\Int}{{\rm Int}}

\newcommand{\field}[1]{\mathbb{#1}}
\newcommand{\C}{\field{C}}

\newcommand{\R}{\field{R}}

\newcommand{\F}{\field{F}}
\newcommand{\be}{{\bf e}}

\newcommand{\cN}{\mathcal{N}}

\title[Lipschitz Geometry and Combinatorics of Circular Snakes]{Lipschitz Geometry and Combinatorics of Circular Snakes}
\author[A. Costa]{Andr\'e Costa$\sharp$}
\address{$\sharp$Departamento de Matem\'atica, Universidade Estadual do Cear\'a
	(UECE), Campus Itaperi. Av. Dr. Silas Munguba, 1700, Cep. 60714-903. Fortaleza-CE, Brasil}
\email{and.costa@uece.br}
\author[D. L. Medeiros]{Davi Lopes Medeiros$\dagger$}
\thanks{$\dagger$Researcher supported by the Serrapilheira Institute (grant number Serra -- R-2110-39576) and by FAPESP (grant number FAPESP: 2024/13488-6).}
\address{$\dagger$Departamento de Matem\'atica, Universidade Federal do Cear\'a (UFC), Campus do Pici, Bloco 914, Cep.~60455-760, Fortaleza-Ce, Brasil}
\address{$\dagger$Departamento de Matem\'atica, Instituto de Ci\^encias Matem\'aticas e de Computa\c{c}\~ao (ICMC-USP). Avenida Trabalhador São-carlense, 400, Centro, 13566-590. São Carlos, SP, Brasil}
\email{profdavilopes@gmail.com}
\author[E. Souza]{Emanoel Souza$\star$}
\address{$\star$Departamento de Matem\'atica, Universidade Estadual do Cear\'a
	(UECE), Campus Itaperi. Av. Dr. Silas Munguba, 1700, Cep. 60714-903. Fortaleza-CE, Brasil}
\email{emanoel.souza@uece.br}

\subjclass{51F30, 14P10, 03C64, 57K10}

\keywords{Lipschitz Geometry, Circular Snakes, Abnormal surfaces, germ of singularities}

\begin{document}

\begin{abstract}
This paper explores the Lipschitz geometric and combinatorial properties of germs of real semialgebraic surfaces (or, more generally, definable in a polynomially bounded o-minimal structure) with circular link (homeomorphic to the circle $\mathbb{S}^1$). We define and investigate the outer Lipschitz geometry of the so-called circular snakes, showing what results in \cite{GabrielovSouza} valid to snakes still holds for the circular case. We prove the existence of a canonical decomposition for the Valette link of a circular snake into finitely many segments and
nodal zones and establish some necessary and sufficient criteria to determine
when it is possible to obtain a snake from a circular snake by “removing” either one of its segments or a
H\"older triangle whose Valette link is contained in one of its nodal zones. We construct a combinatorial object associated with a circular snake and prove a realization theorem for this combinatorial object. We also present a weakly outer Lipschitz classification for circular snakes. Finally, we show some results about the combinatorics of binary circular snakes, which is wildly distinct from the corresponding case shown in \cite{GabrielovSouza}.
\end{abstract}

\maketitle

\section{Introduction}


For the last twenty five years, Lipschitz geometry has attracted considerable attention as a natural approach to classifying singularities that are intermediate between their diffeomorphic (too fine) and topological (too coarse) equivalence.
In particular, the finiteness theorems in \cite{Mostowski85Dissertation} and \cite{Parusinski94} suggest the possibility of an effective bi-Lipschitz classification of definable real surface germs.

As evidenced in \cite{LBirbMosto2000NormalEmbedding}, any singular germ of a semialgebraic (or definable) set $X$ inherits two metrics from the ambient space: the inner metric where the distance between two points of $X$ is the length of the shortest path connecting them inside $X$, and the outer metric with the distance between two points of $X$ being just their distance in the ambient space. This defines two classification problems: equivalence up to bi-Lipschitz homeomorphisms for the inner and outer metrics, or simply, the ``inner classification problem'' and the ``outer classification problem''. Those classification problems and their variants are object of intense research, and some interesting examples of it are the Lipschitz classification of analytic complex plane curves in \cite{Pham-Teissier}, \cite{Fernandes2003}, \cite{neumann-pichon} and \cite{targino}; the thick-thin decomposition in \cite{thick-thin}; the deformation of weighted-homogeneous foliations in \cite{kerner-mendes}; the $\mathcal{K}$-equivalence of real function germs in \cite{K-equivalence}; and the ambient Lipschitz results in \cite{sampaio2016} regarding tangent cones, in \cite{Medeiros-amb} classifying isolated, Lipschitz normally embedded singularities in $\R^3$, and the metric knots in \cite{BBG} and \cite{microknot}.

When $X$ is a surface germ, the problems above have different outcomes. The inner classification problem was solved by Birbrair in \cite{birbrair1999local} and \cite{birbrairOminimal}, but the outer classification problem remains open. 
For the Lipschitz normally embedded singularities, the inner and outer metrics are equivalent, and thus, the two classifications are the same. It was proved in \cite{KurdykaOrro97} that any semialgebraic set can be decomposed into the union of finitely many Lipschitz normally embedded semialgebraic sets. Later, Birbrair and Mostowski called this decomposition a ``pancake decomposition'' in \cite{LBirbMosto2000NormalEmbedding} and they used Kurdyka's construction to prove that any semialgebraic set is inner Lipschitz equivalent to a Lipschitz normally embedded semialgebraic set.

The classification of surface germs with respect to the outer metric is much more complicated. A singular germ $X$ can be considered as the family ${X_t}$ of its links (intersections with the spheres of a small radius $t > 0$). Hence, the Lipschitz geometry of $X$ can be understood as the dynamics of $X_t$ as $t \rightarrow 0$. For this purpose one investigates the “Valette link” of $X$, the family of arcs in $X$ parameterized by the distance to the origin. The outer Lipschitz invariants of $X$ are described in terms of the tangency orders between those arcs. Other outer Lipschitz invariants were developed in order to distinguish two germs, such as volume growth introduced in \cite{metric-homology}, the vanishing homology in \cite{vanishing-homology}, and the moderately discontinuous homology, defined in the seminal paper \cite{moderately-homology} and studied for real surface germs in \cite{moderately-homology-surfaces}.

The first step towards the outer metric classification of surface germs was made
in \cite{PizzaPaper2017}, where the authors introduced the concept of the so-called pizza. Roughly, the pizzas captures a monotonic behavior of the change of outer Lipschitz contact in surfaces. Later on, the paper \cite{GabrielovSouza} was the next step towards the outer metric classification of surface germs. They identified basic ``abnormal'' parts of a surface germ, called snakes, and investigated their geometric and combinatorial properties. Indeed, they provided a canonical outer Lipschitz invariant decomposition of the Valette link of a snake into ``segments'' and ``nodal zones'' (which are Lipschitz normally embedded in the case where the surface has at least one nodal arc). Moreover, they solve a weak version of the outer bi-Lipschitz classification, called the weakly outer bi-Lipschitz classification (see Theorem 6.28 of \cite{GabrielovSouza}). One of the most recent contributions to the outer Lipschitz classification problem is the introduction of the $\sigma\tau$-pizza in \cite{PairsofHT} and its application to determine the outer Lipschitz geometry of pairs of Lipschitz normally embedded pairs of H\"older triangles.

Since most of the outer Lipschitz classification tools were made for surface germs whose link is homeomorphic to a segment, a natural approach is to develop such a geometric theory for surface germs whose link is homeomorphic to a circle. One must also construct a specific combinatorial theory to describe the Lipschitz geometry of circular snakes, and those are the main objectives of this paper.

The article is organized as follows. To make this paper self-contained, in Section 2 we present the basic definitions and some important results
that will be used later. Subsection 2.1 regards the basic definitions and results in Lipschitz geometry needed to understand the paper. In Subsection 2.2, we introduce the pancake and pizza decomposition and its main properties. Subsection 2.3 contains the first definitions in the ``zonology'' of surface germs and present the definition of snake. Finally, Subsection 2.4 gives a brief presentation of the main definitions and results regarding Lipschitz functions on H\"older triangles presented in \cite{GabrielovSouza}. 

In Section \ref{Section: Circular Snakes}, we present the definition of circular snake and prove our main results about the decomposition of its Valette's link into Lipschitz invariant zones. We also prove in this section that this decomposition induces a canonical pancake decomposition for a circular snake. In Section \ref{Section: CS vs Snakes}, we address the matter of determining criteria to obtain a snake from a circular snake. In particular, we prove that there circular snakes such that no snake can be obtained from it by applying the procedure established of ``removing'' segments or a ``piece of a nodal zone''. The main consequence of this fact is that if one intend to associate a Lipschitz invariant combinatorial object (such as a word) to a circular snake, a Realization Result as Theorem 6.23 of \cite{GabrielovSouza} will not easily follows from it.

The goal of Section 5 is to define circular snake names, which is the combinatorial object that describes circular snakes. In Section 6, we establish the relation between circular snakes and circular snake names. Section 7 is devoted to show that every circular snake name has a circular snake associated with it. In Section 8, the weakly bi-Lipschitz maps in surfaces with circular link are introduced. We show that the circular snake name is, in some sense, a complete weakly bi-Lipschitz invariant for circular snakes with at least one nodal zone. For circular snakes without nodal zones, we show that the multiplicity is the required invariant. Using this result, we give a quick proof of one of the main results in \cite{Fernandes2003}, showing how the Puiseux pairs determines the outer Lipschitz geometry of complex plane curves. Finally, in Section 9, we study binary circular snakes, showing how to reduce any circular snake name to a binary one. We also propose a question about the number of binary circular snake names, and we show that this is surprisingly not analogous with the respective problem for binary snakes.

We would like to thank Lev Birbrair and Andrei Gabrielov for productive discussions on Lipschitz geometry of surfaces, and for pointing out the main needs for the construction of the combinatorial theory of circular snakes, giving us the direction of this article.

\section{Preliminaries}\label{Section: Preliminaries}

All sets, functions and maps in this paper are assumed to be definable in a polynomially bounded o-minimal structure over $\mathbb{R}$ with the field of exponents $\mathbb{F}$, for example, real semialgebraic or subanalytic (see \cite{LvandenDRIES98} and \cite{DRIES-SPREISSEGGER}). Unless the contrary is explicitly stated, we consider germs at the origin of all sets and maps.

\subsection{Basic concepts in Lipschitz geometry}\label{Subsec: Holder triangles}

\begin{Def}\label{DEF: inner, outer and normally embedded}
	 Given a germ at the origin of a set  $X\subset \mathbb{R}^{n}$ we can define two metrics on $X$, the \em outer metric \em $d(x,y)=||x-y||$ and the \em inner metric \em $d_{i}(x,y)=\inf_{\alpha}\{l(\alpha)\}$, where $l(\alpha)$ is the length of a rectifiable path $\alpha$ from $x$ to $y$ in $X$. Note that such a path $\alpha$ always exists since $X$ is definable. A set $X \subset \R^{n}$ is \em Lipschitz normally embedded \em (LNE for short) if the outer and inner metrics are equivalent.
\end{Def}

\begin{remark}
	 The inner metric is not definable, but it is equivalent to a definable metric (see \cite{KurdykaOrro97}), for example, the \em pancake metric \em (see \cite{LBirbMosto2000NormalEmbedding}).
\end{remark}

\begin{Def}\label{Def: arc}
		An \em arc \em in $\mathbb{R}^{n}$ is a germ at the origin of a mapping $\gamma \colon [0,\epsilon) \longrightarrow \mathbb{R}^{n}$ such that $\gamma(0) = 0$. Unless otherwise specified, we suppose that arcs are parameterized by the distance to the origin, i.e., $||\gamma(t)||=t$. We usually identify an arc $\gamma$ with its image in $\mathbb{R}^{n}$. For a germ at the origin of a set $X$, the set of all arcs $\gamma \subset X$ is denoted by $V(X)$ (known as the \em Valette link \em of $X$, see \cite{valette2007link}).
\end{Def}

\begin{Def}\label{DEF: order of tangency}
	 The \em tangency order \em of two arcs $\gamma_{1}$ and $\gamma_{2}$ in $V(X)$ (notation $\tord(\gamma_{1},\gamma_{2})$) is the exponent $q$ where $||\gamma_{1}(t) - \gamma_{2}(t)|| = ct^{q} + o(t^{q})$ with $c\neq 0$. By convention, $\tord(\gamma,\gamma)=\infty$. For an arc $\gamma$ and a set of arcs $Z \subset V(X)$, the tangency order of $\gamma$ and $Z$ (notation $\tord(\gamma, Z)$), is the supremum of $\tord(\gamma, \lambda)$ over all arcs $\lambda \in Z$. The tangency order of two sets of arcs $Z$ and $Z'$ (notation $\tord(Z,Z')$) is the supremum of $\tord(\gamma, Z')$ over all arcs $\gamma \in Z$. Similarly, we define the tangency orders in the inner metric of $X$, denoted by $\itord(\gamma_{1},\gamma_{2}),\; \itord(\gamma, Z)$ and $\itord(Z,Z')$.
\end{Def}

\begin{remark}\label{Rem: non-archimedean property}
	 An interesting fact about the tangency order of arcs in $\R^n$ is the so called ``non-archimedean property'' (it first appeared in \cite{B.F.METRIC-THEORY} as ``Isosceles property''): given arcs $\gamma_1,\gamma_2,\gamma_3$ in $\R^n$, we have $$\tord(\gamma_2,\gamma_3) \ge \min(\tord(\gamma_1,\gamma_2),\tord(\gamma_1,\gamma_3)).$$ If $\tord(\gamma_1,\gamma_2)\ne \tord(\gamma_1,\gamma_3)$ then $\tord(\gamma_2,\gamma_3) = \min(\tord(\gamma_1,\gamma_2),\tord(\gamma_1,\gamma_3))$.
\end{remark}

\begin{Def}\label{DEF: standard Holder triangle}
	 For $\beta \in \mathbb{F}$, $\beta \ge 1$, the \em standard \em $\beta$-\em H\"older triangle \em $T_\beta\subset\R^2$ is the germ at the origin of the set
	\begin{equation*}
		T_\beta = \{(x,y)\in \R^2 \mid 0\le x\le 1, \; 0\le y \le x^\beta\}.
	\end{equation*}
	The curves $\{x\ge 0,\; y=0\}$ and $\{x\ge 0,\; y=x^\beta\}$ are the \em boundary arcs \em of $T_\beta$.

    For $\beta \in \mathbb{F}$, $\beta \ge 1$, the \em standard \em $\beta$-\em Horn \em $H_\beta\subset\R^2$ is the germ at the origin of the set
	\begin{equation*}\label{Formula:Standard Holder triangle}
		T_\beta = \{(x,y,z)\in \R^3 \mid z\ge 0, \; x^2+y^2=z^{2\beta}\}.
	\end{equation*}
\end{Def}

\begin{Def}\label{DEF: Holder triangle}
	 A germ at the origin of a set $T \subset \mathbb{R}^{n}$ that is bi-Lipschitz equivalent with respect to the inner metric to the standard $\beta$-H\"older triangle $T_\beta$ is called a $\beta$-\em H\"older triangle \em (see \cite{birbrair1999local}).
	The number $\beta \in \mathbb{F}$ is called the \em exponent \em of $T$ (notation $\beta=\mu(T)$). The arcs $\gamma_{1}$ and $\gamma_{2}$ of $T$ mapped to the boundary arcs of $T_\beta$ by the homeomorphism are the \em boundary arcs \em of $T$ (notation $T=T(\gamma_{1},\gamma_{2})$). All other arcs of $T$ are \em interior arcs \em. The set of interior arcs of $T$ is denoted by $I(T)$.

    A germ at the origin of a set $H \subset \mathbb{R}^{n}$ that is bi-Lipschitz equivalent with respect to the inner metric to the standard $\beta$-Horn $H_\beta$ is called a $\beta$-\em Horn \em (see \cite{birbrair1999local}).
	The number $\beta \in \mathbb{F}$ is called the \em exponent \em of $H$ (notation $\beta=\mu(H)$).
\end{Def}

\begin{remark}\label{Rem: NE HT condition}
	It was proved in \cite{birbrair1999local} that $\mu(T)$ and $\mu(H)$ are inner bi-Lipschitz invariants. Moreover, it was proved in \cite{birbrair2018arc}, using the Arc Selection Lemma (see Theorem 2.2), that a H\"older triangle $T$ is Lipschitz normally embedded if, and only if, $\tord(\gamma,\gamma')=\itord(\gamma,\gamma')$ for any two arcs $\gamma$ and $\gamma'$ of $T$.
\end{remark}

\begin{Def}\label{DEF: Lipschitz non-singular arc}
	 Let $X$ be a surface (a two-dimensional set). An arc $\gamma \subset X$ is \em Lipschitz non-singular \em if there exists a Lipschitz normally embedded H\"older triangle $T \subset X$ such that $\gamma$ is an interior arc of  $T$ and $\gamma \not\subset \overline{X\setminus T}$. Otherwise, $\gamma$ is \em Lipschitz singular \em. In particular, any interior arc of a Lipschitz normally embedded H\"older triangle is Lipschitz non-singular. The union of all Lipschitz singular arcs in $X$ is denoted by $\Lsing(X)$.
\end{Def}

\begin{remark}\label{Rem: boundary arcs are Lips sing arcs}
	 It follows from pancake decomposition (see Definition \ref{Def: pancake decomposition}) that a surface $X$ contains finitely many Lipschitz singular arcs. For an interesting example of a Lipschitz singular arc see Example 2.11 of \cite{GabrielovSouza}. Arcs which are boundary arcs of H\"older triangles or self-intersections of the surface are trivial examples of Lipschitz singular arcs.
\end{remark}

\begin{Def}\label{DEF: non-singular Holder triangle}
	 A H\"older triangle $T$ is \em non-singular \em if all interior arcs of $T$ are Lipschitz non-singular.
\end{Def}

\begin{Def}\label{Def: generic arc of a surface}
	 Let $X$ be a surface germ with connected link. The \em exponent \em $\mu(X)$ of $X$ is defined as $\mu(X)=\min\, \itord(\gamma,\gamma')$, where the minimum is taken over all arcs $\gamma,\,\gamma'$ of $X$.
	A surface $X$ with exponent $\beta$ is called a $\beta$-surface. An arc $\gamma \subset X\setminus \Lsing(X)$ is \em generic \em if $\itord(\gamma,\gamma') = \mu(X)$ for all arcs $\gamma'\subset \Lsing(X)$. The set of generic arcs of $X$ is denoted by $G(X)$. 
\end{Def}

\begin{remark}\label{Rem: generic arcs of a non-singular HT}
	 If $X=T(\gamma_{1},\gamma_{2})$ is a non-singular $\beta$-H\"older triangle then an arc $\gamma\subset X$ is \em generic \em if, and only if, $\itord(\gamma_{1},\gamma) = \itord(\gamma,\gamma_{2}) = \beta$.
\end{remark}

\subsection{Pancake and pizza decompositions}\label{subsection:pancake decomposition}

\begin{Def}\label{Def: pancake decomposition}
	 Let $X\subset \mathbb{R}^{n}$ be the germ at the origin of a closed set. A \em pancake decomposition \em of $X$ is a finite collection of closed LNE subsets $X_{k}$ of $X$ with connected links, called \em pancakes\em, such that $X=\bigcup X_{k}$ and $$\dim(X_{j}\cap X_{k}) < \min(\dim(X_{j}),\dim(X_{k}))\quad\text{for all}\; j,k.$$
\end{Def}

\begin{remark}\label{Rem: existence of pancake decomp}
	 The term ``pancake'' was introduced in \cite{LBirbMosto2000NormalEmbedding}, but this notion first appeared (with a different name) in \cite{kurdyka1992subanalytic} and \cite{KurdykaOrro97}, where the existence of such decomposition was established.
\end{remark}

\begin{remark}\label{Rem:pancake of holder triangle is holder triangle}
	 If $X$ is a H\"older triangle and $\{X_k\}_{i=1}^p$ is its pancake decomposition, then each pancake $X_k$ is also a H\"older triangle. Moreover, if $X$ is a non LNE surface germ and has circular link, then $p>1$, each pancake is a H\"older triangle, denoted by $X_k=T(\gamma_{k-1},\gamma_k)$, and $\gamma_0 = \gamma_{p}$.  
\end{remark}

\begin{Def}\label{Def: reduced pancak decomp}
	 A pancake decomposition $\{X_{k}\}$ of a set $X$ is \em reduced \em if the union of any two adjacent pancakes $X_{j}$ and $X_{k}$ (such that $X_{j}\cap X_{k}\ne \{0\}$) is not LNE.
\end{Def}

\begin{remark}
	 When the union of two adjacent pancakes is LNE, they can be replaced by their union, reducing the number of pancakes. Thus, a reduced pancake decomposition always exists. 
\end{remark}

In the following, we use the definitions and results of \cite{PizzaPaper2017}.

\begin{Def}\label{DEF: order of a function on an arc}
	 Let $f\not\equiv 0$ be a germ at the origin of a Lipschitz function defined on an arc $\gamma$. The \em order \em of $f$ on $\gamma$, denoted by $\ord_{\gamma}f$, is the value $q \in \mathbb{F}$ such that $f(\gamma(t)) = ct^{q} + o(t^{q})$ as $t \rightarrow 0$, where $c\neq 0$. If $f\equiv 0$ on $\gamma$, we set $\ord_{\gamma}f = \infty$.
\end{Def}

\begin{Def}\label{DEF: interval Q of a pizza slice}
	 Let $T\subset \mathbb{R}^{n}$ be a H\"older triangle, and let $f\colon (T,0) \rightarrow (\mathbb{R},0)$ be a Lipschitz function. We define $$Q_{f}(T)=\bigcup_{\gamma\in V(T)}\{ \ord_{\gamma}f\}.$$
\end{Def}

\begin{remark}\label{Rem:Q_{f}(T) is a segment}
	 It was shown in \cite{PizzaPaper2017} that $Q_{f}(T)$ is a closed segment in $\mathbb{F}\cup \{\infty\}$.
\end{remark}

\begin{Def}\label{DEF: Elementary Holder triangle}
	 A H\"older triangle $T$ is \em elementary \em with respect to a Lipschitz function $f$ if, for any two distinct arcs $\gamma$ and $\gamma'$ in $T$ such that $\ord_{\gamma}f=\ord_{\gamma'}f=q$, the order of $f$ is $q$ on any arc in the H\"older triangle $T(\gamma,\gamma')\subset T$.
\end{Def}

\begin{Def}\label{Def:width function}
	 Let $T\subset \mathbb{R}^{n}$ be a H\"older triangle and $f\colon (T,0) \rightarrow (\mathbb{R},0)$ a Lipschitz function. For each arc $\gamma \subset T$, the \em width \em $\mu_{T}(\gamma,f)$ of $\gamma$ with respect to $f$ is the infimum of the exponents of H\"older triangles $T'\subset T$ containing $\gamma$ such that $Q_{f}(T')$ is a point. For $q\in Q_{f}(T)$ let $\mu_{T,f}(q)$ be the set of exponents $\mu_{T}(\gamma,f)$, where $\gamma$ is any arc in $T$ such that $\ord_{\gamma}f=q$. It was shown in \cite{PizzaPaper2017} that the set $\mu_{T,f}(q)$ is finite. This defines a multivalued \em width function \em $\mu_{T,f}\colon Q_{f}(T)\rightarrow \mathbb{F}\cup \{\infty\}$. When $f$ is fixed, we write $\mu_{T}(\gamma)$ instead of $\mu_{T}(\gamma,f)$ and $\mu_{T}$ instead of $\mu_{T,f}$. If $T$ is an elementary H\"older triangle with respect to $f$ then the function $\mu_{T,f}$ is single valued.
\end{Def}

\begin{Def}\label{DEF: pizza slice}
	 Let $T$ be a H\"older triangle and $f\colon (T,0)\rightarrow (\mathbb{R},0)$ a Lipschitz function. We say that $T$ is a \em pizza slice \em associated with $f$ if it is elementary with respect to $f$ and $\mu_{T,f}(q)=aq+b$ is an affine function.
\end{Def}

\begin{Prop}\label{Prop:width function properties elementary triangle}
	\emph{(See \cite{PizzaPaper2017})} Let $T$ be a $\beta$-H\"older triangle, $f$ a Lipschitz function on $T$ and $Q=Q_{f}(T)$. If $T$ is a pizza slice associated with $f$ then:
 
	\begin{enumerate}
		\item $\mu_{T}$ is constant only when $Q$ is a point;
		
		\item $\beta\le\mu_{T}(q)\le q$ for all $q \in Q$;
		
		\item $\mu(\ord_{\gamma}f)=\beta$ for all $\gamma \in G(T)$;
		\item If $Q$ is not a point, let $\mu_0=\max_{q\in Q}\mu_T(q)$, and let $\gamma_0$ be the boundary arc of $T$ such that $\mu_T(\gamma_0)=\mu_0$. Then $\mu_T(\gamma)=\itord(\gamma_0,\gamma)$
		for all arcs $\gamma\subset T$ such that $\itord(\gamma_0,\gamma)\le\mu_0$.
	\end{enumerate}
\end{Prop}

\begin{Def}\label{Def:Pizza decomp}
	 A decomposition $\{T_i\}$ of a H\"older triangle $X$ into $\beta_{i}$-H\"older triangles $T_{i}=T(\lambda_{i-1},\lambda_{i})$ such that $T_{i-1}\cap T_{i}=\lambda_{i}$ is a \em pizza decomposition \em of $X$ (or just a \em pizza \em on $X$) associated with $f$ if each $T_{i}$ is a pizza slice associated with $f$. We write $Q_{i}=Q_{f}(T_{i})$, $\mu_{i}=\mu_{T_{i},f}$ and $q_i=\ord_{\lambda_i}f$. 
\end{Def}

\begin{remark}\label{REM: existence of pizza slice}
	 The existence of a pizza associated with a function $f$ was proved in \cite{PizzaPaper2017} for a function defined in $(\R^{2},0)$. The same arguments prove the existence of a pizza associated with a function defined on a H\"older triangle as in Definition \ref{Def:Pizza decomp}. The results mentioned in this subsection remain true when $f$ is a Lipschitz function on a H\"older triangle $T$ with respect to the inner metric, although in this paper we need them only for Lipschitz functions with respect to the outer metric.
\end{remark}

\begin{Def}\label{Def:minimal pizza}
	 A pizza $\{T_i\}_{i=1}^{p}$ associated with a function $f$ is \em minimal \em if, for any $i\in \{2,\ldots, p\}$, $T_{i-1}\cup T_i$ is not a pizza slice associated with $f$.
\end{Def}

\begin{Def}\label{DEF: multipizza}
	 Consider the set of germs of Lipschitz functions $f_{l}\colon (X,0) \rightarrow (\mathbb{R},0),\;l=1,\ldots , m$, defined on a H\"older triangle $X$. A \em multipizza \em on $X$ associated with $\{f_{1},\ldots, f_{m}\}$ is a decomposition $\{T_{i}\}$ of $X$ into $\beta_{i}$-H\"older triangles which is a pizza on $X$ associated with $f_{l}$ for each $l$.
\end{Def}

\begin{remark}\label{REM: existence of a multipizza}
	 The existence of a multipizza follows from the existence of a pizza associated with a single Lipschitz function $f$, since a refinement of a pizza associated with any function $f$ is also a pizza associated with $f$.
\end{remark}

\subsection{Zones and snakes}\label{Subsec: Zones}
In this subsection, $(X,0)\subset (\R^n,0)$ is a surface germ.

\begin{Def}\label{Def: zone}
	 A nonempty set of arcs $Z \subset V(X)$ is a \em zone \em if, for any two distinct arcs $\gamma_{1}$ and $\gamma_{2}$ in $Z$, there exists a non-singular H\"older triangle $T=T(\gamma_{1},\gamma_{2}) \subset X$ such that $V(T) \subset Z$. If $Z = \{\gamma\}$ then $Z$ is a \em singular zone \em.
\end{Def}

\begin{Def}\label{Def: maximal zone in}
	 Let $B \subset V(X)$ be a nonempty set. A zone $Z\subset B$ is \em maximal in \em $B$ if, for any H\"older triangle $T$ such that $V(T) \subset B$, one has either $Z\cap V(T)=\emptyset$ or $V(T) \subset Z$.
\end{Def}

\begin{remark}
	 A zone could be understood as an analog of a connected subset of $V(X)$, and a maximal zone in a set $B$ is an analog of a connected component of $B$.
\end{remark}

\begin{Def}\label{Def:order of zone}
	 The \em order \em $\mu(Z)$ of a zone $Z$ is the infimum of $\tord(\gamma,\gamma')$ over all arcs $\gamma$ and $\gamma'$ in $Z$. If $Z$ is a singular zone then $\mu(Z) = \infty$. A zone $Z$ of order $\beta$ is called a $\beta$-zone.
\end{Def}

\begin{remark}\label{Rem: replace tord by itord in def of order of a zone}
	 The tangency order can be replaced by the inner tangency order in Definition \ref{Def:order of zone}. Note that, for any arc $\gamma \in Z$, $\inf_{\gamma'\in Z}\tord(\gamma,\gamma')=\inf_{\gamma'\in Z}\itord(\gamma,\gamma')=\mu(Z)$.
\end{remark}

\begin{Def}\label{Def: NE zone}
	 A zone $Z$ is Lipschitz normally embedded if, for any two arcs $\gamma$ and $\gamma'$ in $Z$, there exists a LNE H\"older triangle $T=T(\gamma,\gamma')$ such that $V(T)\subset Z$.
\end{Def}

\begin{Def}\label{Def: open, closed and perfect zones}
	 A $\beta$-zone $Z$ is \em closed \em if there is a $\beta$-H\"older triangle $T$ such that $V(T)\subset Z$. Otherwise, $Z$ is \em open \em. A zone $Z$ is \em perfect \em if, for any two arcs $\gamma \ne \gamma'$ in $Z$, there is a H\"older triangle $T$ such that $V(T)\subset Z$ and both $\gamma$ and $\gamma'$ are generic arcs of $T$. By definition, any singular zone is perfect.
\end{Def}

\begin{Def}\label{Def:complete and open complete zone}
	 
	An open $\beta$-zone $Z \subset V(X)$ is $\beta$-\em complete \em if, for any $\gamma\in Z$, $$Z=\{\gamma'\in V(X)\mid \itord(\gamma,\gamma')> \beta\}.$$
\end{Def}

\begin{Exam}\label{Exam: closed, open and open complete zones}
	 Let $T$ be a H\"older triangle. The set $G(T)$ is a closed $\beta$-zone. If $Z = V(T)\setminus G(T)$ then $Z$ is the disjoint union of two open $\beta$-zones $H_1$ and $H_2$, which are not $\beta$-complete, since each of them contains a boundary arc. For an example of an open $\beta$-complete zone consider $H_{\gamma} = \{\theta \in V(X) \mid \itord(\theta,\gamma) > \beta\}$, where $\gamma \in G(T)$.  
\end{Exam}

\begin{remark}\label{Rem:open complete zones}
	 Let $Z$ and $Z'$ be open $\beta$-complete zones. Then, either $Z\cap Z'=\emptyset$ or $Z=Z'$. Moreover, $Z\cap Z'=\emptyset$ implies $\itord(Z,Z')\le\beta$. 
\end{remark}

\begin{Def}\label{Def:adjacent zones}
	 Two zones $Z$ and $Z'$ in $V(X)$ are \em adjacent \em if $Z\cap Z'=\emptyset$ and there exist arcs $\gamma \subset Z$ and $\gamma' \subset Z'$ such that $V(T(\gamma,\gamma')) \subset Z\cup Z'$.
\end{Def}

\begin{Exam}
	 Let $T$, $Z_1$ and $Z_2$ be as in Example \ref{Exam: closed, open and open complete zones}. The zones $Z_1$ and $Z_2$ are adjacent to $G(T)$.
\end{Exam}

The next three Lemmas are proved in \cite{GabrielovSouza} (Lemmas 2.45, 2.46 and 2.47) for $X$ being a H\"older triangle. When $X$ is a surface witg connected link one can consider a suitable H\"older triangle inside the surface and use the same proof.
 
\begin{Lem}\label{Lem: union of adjacent zones is a zone}
	Let $X$ be a surface, and let $Z$ and $Z'$ be two zones in $V(X)$ of orders $\beta$ and $\beta'$, respectively. If either $Z\cap Z'\ne\emptyset$ or $Z$ and $Z'$ are adjacent, then $Z\cup Z'$ is a zone of order $\min(\beta,\beta')$.
\end{Lem}

\begin{Lem}\label{Lem: beta zone must have beta intersection part in HT decomp}
	Let $\{X_i\}$ be a finite decomposition of a surface $X$ into $\beta_i$-H\"older triangles. If $Z\subset V(X)$ is a $\beta$-zone then $Z_i=Z\cap V(X_i)$ is a $\beta$-zone for some $i$.
\end{Lem}

\begin{Lem}\label{Lem: there are no adjacent perfect zones}
	Let $X$ be a surface. If $Z$ and $Z'$ are perfect $\beta$-zones in $V(X)$, then they are not adjacent.
\end{Lem}

Below it is finally presented the fundamental notion of ``abnormal arc". 

\begin{Def}\label{DEF: normal and abnormal arcs and zones}
	 A Lipschitz non-singular arc $\gamma$ of a surface germ $X$ is \em abnormal \em if there are two LNE non-singular H\"older triangles $T\subset X$ and $T'\subset X$ such that $T\cap T' = \gamma$ and $T\cup T'$ is not LNE.
	Otherwise $\gamma$ is \em normal \em. A zone is \em abnormal \em (resp., \em normal \em) if all of its arcs are abnormal (resp., normal). The sets of abnormal and normal arcs of $X$ are denoted $\Abn(X)$ and $\Nor(X)$, respectively.
\end{Def}
\begin{remark}
    The condition that the H\"older triangles $T$ and $T'$ are non-singular is missing in the definition of abnormal arc presented in \cite{GabrielovSouza}. In order to avoid overloading the proofs we will not be repeating this condition when using this definition along the text.
\end{remark}

\begin{Def}\label{Def: abnormal surface}
	 A surface germ $X$ is called \em abnormal \em if $\Abn(X)=G(X)$, the set of generic arcs of $X$.
\end{Def}

\begin{remark}\label{Rem:triangles of abnormal arc}
	 Given an abnormal arc $\gamma\subset X$, one can choose LNE H\"older triangles $T=T(\lambda,\gamma)\subset X$ and $T'=T(\gamma,\lambda')\subset X$ so that $T\cap T'=\gamma$ and $\tord(\lambda,\lambda')>\itord(\lambda,\lambda')$. It follows that $\tord(\lambda,\gamma)=\tord(\gamma,\lambda')=\itord(\lambda,\lambda')$ (see Lemma 2.14 in \cite{GabrielovSouza}).
\end{remark}

\begin{Exam}\label{Exam: examples of abnormal surfaces NE and not NE}
	 
	Given two arcs $\theta$ and $\tilde{\theta}$ in $\R^n$ let $T(\theta,\tilde{\theta})$ be the H\"older triangle defined by (the germ of) the union of the straight line segments, $[\theta(t),\tilde{\theta}(t)]$, connecting $\theta(t)$ and $\tilde{\theta}(t)$ for all $t\ge 0$. Consider the set $T = T_1\cup T_2$ with $T_1 = T(\gamma_{1},\lambda_1)$ and $T_2 = T(\lambda_1,\lambda_2)\cup T(\lambda_2,\gamma_{2})$ where $\gamma_i(t)=(t,(-1)^{i}t^{\frac{3}{2}},0)$ and $\lambda_i(t) = (t,(-1)^{i}t,t)$, for $i=1,2$, are arcs parameterized by the first coordinate, which is equivalent to the distance to the origin. The H\"older triangles $T_1$, $T_2$, $T(\lambda_1,\lambda_2)$ and $T(\lambda_2,\gamma_{2})$ are Lipschitz normally embedded. In particular, $T$ is non-singular. 
	
	Notice that $\tord(\gamma_1,\gamma_2) = \frac{3}{2}> 1$ and $\itord(\gamma_1,\gamma_2) = 1$, since $d_{i}(\gamma_1(t),\gamma_2(t)) \ge 2t$. Therefore, $T$ is a non-LNE $1$-H\"older triangle with $\Abn(T) = G(T)$ and $\Nor(T) = H_1\cup H_2$, where $H_i$ is defined as in Example \ref{Exam: closed, open and open complete zones} for $i=1,2$. Thus, $T$ is a non-LNE abnormal surface (indeed, $T$ is a bubble snake, see Definition 4.45 of \cite{GabrielovSouza}). 
    Another example of a non-Lipschitz normally embedded abnormal surface, in this case with circular link, is the complex cusp $\{(z,w)\in \C^2 \mid z^3 = w^2\}$ seen as a real surface in $\R^4$. Examples of Lipschitz normally embedded, abnormal surfaces with circular link are the standard $\beta$-horn, for each $\beta \in \mathbb{Q}_{\ge 1}$.
\end{Exam}

\begin{remark}
	There are abnormal surfaces with circular link containing Lipschitz singular arcs. Notice that by removing a H\"older triangle with exponent higher than $\beta$ containing the Lipschitz singular arc (represented by the point in the ``corner''), one obtains a snake (see Figure 1). 
\end{remark}

\begin{center}

\tikzset{every picture/.style={line width=0.75pt}} 

\begin{tikzpicture}[x=0.75pt,y=0.75pt,yscale=-1,xscale=1]

\draw    (429.52,114.16) .. controls (447.76,111.14) and (454.04,107.39) .. (474.88,84.07) .. controls (495.73,60.75) and (524.39,63.34) .. (539.37,78.89) .. controls (554.35,94.43) and (555.91,101.56) .. (556.3,112.57) .. controls (556.69,123.58) and (553.31,150.3) .. (512.27,150.3) .. controls (471.24,150.3) and (458.99,116.46) .. (423.17,118.4) .. controls (387.34,120.34) and (377.57,159.21) .. (389.95,183.17) .. controls (402.32,207.14) and (430.98,212.32) .. (449.87,195.48) .. controls (468.76,178.64) and (471.89,154.19) .. (514.23,155.48) ;
\draw  [draw opacity=0][fill={rgb, 255:red, 155; green, 155; blue, 155 }  ,fill opacity=0.47 ] (498.94,151.09) .. controls (498.94,142.3) and (506.11,135.16) .. (514.95,135.16) .. controls (523.8,135.16) and (530.97,142.3) .. (530.97,151.09) .. controls (530.97,159.89) and (523.8,167.02) .. (514.95,167.02) .. controls (506.11,167.02) and (498.94,159.89) .. (498.94,151.09) -- cycle ;
\draw  [draw opacity=0][fill={rgb, 255:red, 155; green, 155; blue, 155 }  ,fill opacity=0.47 ] (407.98,113.72) .. controls (407.98,104.92) and (415.15,97.79) .. (423.99,97.79) .. controls (432.84,97.79) and (440.01,104.92) .. (440.01,113.72) .. controls (440.01,122.51) and (432.84,129.64) .. (423.99,129.64) .. controls (415.15,129.64) and (407.98,122.51) .. (407.98,113.72) -- cycle ;
\draw    (389.66,150.9) -- (491.18,150.9) ;
\draw [shift={(493.18,150.9)}, rotate = 180] [color={rgb, 255:red, 0; green, 0; blue, 0 }  ][line width=0.75]    (10.93,-3.29) .. controls (6.95,-1.4) and (3.31,-0.3) .. (0,0) .. controls (3.31,0.3) and (6.95,1.4) .. (10.93,3.29)   ;
\draw [shift={(387.66,150.9)}, rotate = 0] [color={rgb, 255:red, 0; green, 0; blue, 0 }  ][line width=0.75]    (10.93,-3.29) .. controls (6.95,-1.4) and (3.31,-0.3) .. (0,0) .. controls (3.31,0.3) and (6.95,1.4) .. (10.93,3.29)   ;
\draw    (442.01,113.72) -- (550.39,113.86) ;
\draw [shift={(552.39,113.87)}, rotate = 180.08] [color={rgb, 255:red, 0; green, 0; blue, 0 }  ][line width=0.75]    (10.93,-3.29) .. controls (6.95,-1.4) and (3.31,-0.3) .. (0,0) .. controls (3.31,0.3) and (6.95,1.4) .. (10.93,3.29)   ;
\draw [shift={(440.01,113.72)}, rotate = 0.08] [color={rgb, 255:red, 0; green, 0; blue, 0 }  ][line width=0.75]    (10.93,-3.29) .. controls (6.95,-1.4) and (3.31,-0.3) .. (0,0) .. controls (3.31,0.3) and (6.95,1.4) .. (10.93,3.29)   ;
\draw    (514.23,155.48) .. controls (543.02,154.71) and (546.47,162.61) .. (559.82,169.09) .. controls (573.17,175.56) and (592.71,175.56) .. (607.69,166.82) .. controls (622.68,158.07) and (632.12,132.81) .. (631.47,113.7) .. controls (630.82,94.6) and (620.72,79.37) .. (605.09,66.74) .. controls (589.46,54.11) and (568.94,46.66) .. (548.75,46.34) .. controls (528.56,46.02) and (505.76,47.63) .. (488.17,53.46) .. controls (470.58,59.29) and (462.44,70.31) .. (454.3,83.26) .. controls (446.16,96.22) and (451.41,105.29) .. (428.75,109.85) ;
\draw    (561.39,113.72) -- (629.47,113.7) ;
\draw [shift={(631.47,113.7)}, rotate = 179.99] [color={rgb, 255:red, 0; green, 0; blue, 0 }  ][line width=0.75]    (10.93,-3.29) .. controls (6.95,-1.4) and (3.31,-0.3) .. (0,0) .. controls (3.31,0.3) and (6.95,1.4) .. (10.93,3.29)   ;
\draw [shift={(559.39,113.72)}, rotate = 359.99] [color={rgb, 255:red, 0; green, 0; blue, 0 }  ][line width=0.75]    (10.93,-3.29) .. controls (6.95,-1.4) and (3.31,-0.3) .. (0,0) .. controls (3.31,0.3) and (6.95,1.4) .. (10.93,3.29)   ;
\draw    (89.78,110.33) .. controls (108.02,107.31) and (115.57,106.39) .. (136.42,83.07) .. controls (157.26,59.75) and (185.92,62.34) .. (200.9,77.89) .. controls (215.88,93.43) and (217.44,100.56) .. (217.83,111.57) .. controls (218.23,122.58) and (214.84,149.3) .. (173.8,149.3) .. controls (132.77,149.3) and (120.52,115.46) .. (84.7,117.4) .. controls (48.87,119.34) and (39.1,158.21) .. (51.48,182.17) .. controls (63.86,206.14) and (92.51,211.32) .. (111.4,194.48) .. controls (130.29,177.64) and (133.42,153.19) .. (175.76,154.48) ;
\draw  [draw opacity=0][fill={rgb, 255:red, 155; green, 155; blue, 155 }  ,fill opacity=0.47 ] (160.47,150.09) .. controls (160.47,141.3) and (167.64,134.16) .. (176.48,134.16) .. controls (185.33,134.16) and (192.5,141.3) .. (192.5,150.09) .. controls (192.5,158.89) and (185.33,166.02) .. (176.48,166.02) .. controls (167.64,166.02) and (160.47,158.89) .. (160.47,150.09) -- cycle ;
\draw  [draw opacity=0][fill={rgb, 255:red, 155; green, 155; blue, 155 }  ,fill opacity=0.47 ] (69.51,112.72) .. controls (69.51,103.92) and (76.68,96.79) .. (85.53,96.79) .. controls (94.37,96.79) and (101.54,103.92) .. (101.54,112.72) .. controls (101.54,121.51) and (94.37,128.64) .. (85.53,128.64) .. controls (76.68,128.64) and (69.51,121.51) .. (69.51,112.72) -- cycle ;
\draw    (51.19,149.9) -- (152.71,149.9) ;
\draw [shift={(154.71,149.9)}, rotate = 180] [color={rgb, 255:red, 0; green, 0; blue, 0 }  ][line width=0.75]    (10.93,-3.29) .. controls (6.95,-1.4) and (3.31,-0.3) .. (0,0) .. controls (3.31,0.3) and (6.95,1.4) .. (10.93,3.29)   ;
\draw [shift={(49.19,149.9)}, rotate = 0] [color={rgb, 255:red, 0; green, 0; blue, 0 }  ][line width=0.75]    (10.93,-3.29) .. controls (6.95,-1.4) and (3.31,-0.3) .. (0,0) .. controls (3.31,0.3) and (6.95,1.4) .. (10.93,3.29)   ;
\draw    (103.54,112.72) -- (211.93,112.86) ;
\draw [shift={(213.93,112.87)}, rotate = 180.08] [color={rgb, 255:red, 0; green, 0; blue, 0 }  ][line width=0.75]    (10.93,-3.29) .. controls (6.95,-1.4) and (3.31,-0.3) .. (0,0) .. controls (3.31,0.3) and (6.95,1.4) .. (10.93,3.29)   ;
\draw [shift={(101.54,112.72)}, rotate = 0.08] [color={rgb, 255:red, 0; green, 0; blue, 0 }  ][line width=0.75]    (10.93,-3.29) .. controls (6.95,-1.4) and (3.31,-0.3) .. (0,0) .. controls (3.31,0.3) and (6.95,1.4) .. (10.93,3.29)   ;
\draw    (175.76,154.48) .. controls (204.55,153.71) and (208,161.61) .. (221.35,168.09) .. controls (234.7,174.56) and (254.24,174.56) .. (269.23,165.82) .. controls (284.21,157.07) and (293.65,131.81) .. (293,112.7) .. controls (292.35,93.6) and (282.25,78.37) .. (266.62,65.74) .. controls (250.99,53.11) and (230.47,45.66) .. (210.28,45.34) .. controls (190.09,45.02) and (167.29,46.63) .. (149.7,52.46) .. controls (132.12,58.29) and (123.97,69.31) .. (115.83,82.26) .. controls (107.69,95.22) and (112.45,105.76) .. (89.78,110.33) ;
\draw  [fill={rgb, 255:red, 0; green, 0; blue, 0 }  ,fill opacity=1 ] (86.88,110.33) .. controls (86.88,109.55) and (87.53,108.92) .. (88.33,108.92) .. controls (89.13,108.92) and (89.78,109.55) .. (89.78,110.33) .. controls (89.78,111.1) and (89.13,111.73) .. (88.33,111.73) .. controls (87.53,111.73) and (86.88,111.1) .. (86.88,110.33) -- cycle ;
\draw    (222.92,112.72) -- (291,112.7) ;
\draw [shift={(293,112.7)}, rotate = 179.99] [color={rgb, 255:red, 0; green, 0; blue, 0 }  ][line width=0.75]    (10.93,-3.29) .. controls (6.95,-1.4) and (3.31,-0.3) .. (0,0) .. controls (3.31,0.3) and (6.95,1.4) .. (10.93,3.29)   ;
\draw [shift={(220.92,112.72)}, rotate = 359.99] [color={rgb, 255:red, 0; green, 0; blue, 0 }  ][line width=0.75]    (10.93,-3.29) .. controls (6.95,-1.4) and (3.31,-0.3) .. (0,0) .. controls (3.31,0.3) and (6.95,1.4) .. (10.93,3.29)   ;
\draw  [fill={rgb, 255:red, 0; green, 0; blue, 0 }  ,fill opacity=1 ] (427.3,109.85) .. controls (427.3,109.08) and (427.95,108.45) .. (428.75,108.45) .. controls (429.54,108.45) and (430.19,109.08) .. (430.19,109.85) .. controls (430.19,110.63) and (429.54,111.26) .. (428.75,111.26) .. controls (427.95,111.26) and (427.3,110.63) .. (427.3,109.85) -- cycle ;
\draw  [fill={rgb, 255:red, 0; green, 0; blue, 0 }  ,fill opacity=1 ] (428.08,114.28) .. controls (428.02,113.5) and (428.61,112.82) .. (429.41,112.76) .. controls (430.2,112.7) and (430.9,113.27) .. (430.96,114.04) .. controls (431.02,114.82) and (430.43,115.5) .. (429.63,115.56) .. controls (428.84,115.62) and (428.14,115.05) .. (428.08,114.28) -- cycle ;

\draw (429.22,152.69) node [anchor=north west][inner sep=0.75pt]    {$\beta $};
\draw (498.95,93.27) node [anchor=north west][inner sep=0.75pt]    {$\beta $};
\draw (588.79,93.54) node [anchor=north west][inner sep=0.75pt]    {$\beta $};
\draw (3,230.83) node [anchor=north west][inner sep=0.75pt]   [align=left] {Figure 1: Example of an abnormal surface with circular link containing a Lipschitz singular arc in $\displaystyle a)$, and the\\snake obtained by removing a neighborhood of the Lipschitz singular arc in $\displaystyle b)$. Points inside the shaded disks\\represent arcs with tangency order higher than $\displaystyle \beta $.};
\draw (90.75,151.69) node [anchor=north west][inner sep=0.75pt]    {$\beta $};
\draw (160.48,92.27) node [anchor=north west][inner sep=0.75pt]    {$\beta $};
\draw (250.32,92.54) node [anchor=north west][inner sep=0.75pt]    {$\beta $};
\draw (41,51.4) node [anchor=north west][inner sep=0.75pt]    {$a)$};
\draw (373,52.4) node [anchor=north west][inner sep=0.75pt]    {$b)$};

\end{tikzpicture}
\end{center}

\begin{Def}\label{Def: maximal abnormal and normal zones}
	 Given an arc $\gamma \subset X$ the \em maximal abnormal zone \em (resp., \em maximal normal zone \em) in $V(X)$ containing $\gamma$ is the union of all abnormal (resp., normal) zones in $V(X)$ containing $\gamma$. Alternatively, the maximal abnormal (resp., normal) zone containing an arc $\gamma\subset X$ is a maximal zone in $\Abn(X)$ (resp., $\Nor(X)$) containing $\gamma$.
\end{Def}

\begin{Def}\label{Def:snake}
	 A non-singular $\beta$-H\"older triangle $T$ is called a $\beta$\em -snake \em if $T$ is an abnormal surface (see Definitions \ref{DEF: non-singular Holder triangle} and \ref{Def: abnormal surface}).
\end{Def}

\begin{remark}
	 The H\"older triangle $T$ defined in Example \ref{Exam: examples of abnormal surfaces NE and not NE} is a $1$-snake with link as shown in Figure 2a.
\end{remark}

\begin{remark}\label{Rem: abnormal arcs of a snake}
	 It follows from Definition \ref{Def:snake} and Remark \ref{Rem: generic arcs of a non-singular HT} that an arc in $T$ is normal if and only if it has inner tangency order higher than $\beta$ with one of its boundary arcs, and it is abnormal if and only if it has inner tangency order $\beta$ with both boundary arcs.
\end{remark}

\begin{center}

\tikzset{every picture/.style={line width=0.75pt}} 

\begin{tikzpicture}[x=0.75pt,y=0.75pt,yscale=-1,xscale=1]

\draw    (73.37,37.09) .. controls (39.17,38.06) and (21.58,51.67) .. (20.61,89.56) .. controls (19.63,127.45) and (70.64,129.21) .. (70.64,174.88) ;
\draw    (73.37,37.09) .. controls (107.56,38.06) and (125.15,51.67) .. (126.13,89.56) .. controls (127.1,127.45) and (76.64,130.41) .. (76.64,176.08) ;
\draw  [draw opacity=0][fill={rgb, 255:red, 155; green, 155; blue, 155 }  ,fill opacity=0.47 ] (57.76,176.71) .. controls (57.76,167.91) and (64.93,160.78) .. (73.77,160.78) .. controls (82.62,160.78) and (89.79,167.91) .. (89.79,176.71) .. controls (89.79,185.51) and (82.62,192.64) .. (73.77,192.64) .. controls (64.93,192.64) and (57.76,185.51) .. (57.76,176.71) -- cycle ;
\draw    (22.61,89.56) -- (124.13,89.56) ;
\draw [shift={(126.13,89.56)}, rotate = 180] [color={rgb, 255:red, 0; green, 0; blue, 0 }  ][line width=0.75]    (10.93,-3.29) .. controls (6.95,-1.4) and (3.31,-0.3) .. (0,0) .. controls (3.31,0.3) and (6.95,1.4) .. (10.93,3.29)   ;
\draw [shift={(20.61,89.56)}, rotate = 0] [color={rgb, 255:red, 0; green, 0; blue, 0 }  ][line width=0.75]    (10.93,-3.29) .. controls (6.95,-1.4) and (3.31,-0.3) .. (0,0) .. controls (3.31,0.3) and (6.95,1.4) .. (10.93,3.29)   ;
\draw    (215.8,90.57) .. controls (234.04,87.54) and (246.04,85.45) .. (266.88,62.13) .. controls (287.73,38.81) and (319.82,41.02) .. (334.8,56.57) .. controls (349.78,72.11) and (347.91,79.62) .. (348.3,90.63) .. controls (348.69,101.64) and (344.13,124.31) .. (303.1,124.31) .. controls (262.06,124.31) and (250.99,94.51) .. (215.17,96.46) .. controls (179.34,98.4) and (169.57,137.26) .. (181.95,161.23) .. controls (194.32,185.2) and (222.98,190.38) .. (241.87,173.54) .. controls (260.76,156.7) and (257.13,129.94) .. (299.47,131.23) ;
\draw  [draw opacity=0][fill={rgb, 255:red, 155; green, 155; blue, 155 }  ,fill opacity=0.47 ] (285.17,128.95) .. controls (285.17,120.16) and (292.35,113.03) .. (301.19,113.03) .. controls (310.04,113.03) and (317.21,120.16) .. (317.21,128.95) .. controls (317.21,137.75) and (310.04,144.88) .. (301.19,144.88) .. controls (292.35,144.88) and (285.17,137.75) .. (285.17,128.95) -- cycle ;
\draw  [draw opacity=0][fill={rgb, 255:red, 155; green, 155; blue, 155 }  ,fill opacity=0.47 ] (199.98,91.77) .. controls (199.98,82.98) and (207.15,75.85) .. (215.99,75.85) .. controls (224.84,75.85) and (232.01,82.98) .. (232.01,91.77) .. controls (232.01,100.57) and (224.84,107.7) .. (215.99,107.7) .. controls (207.15,107.7) and (199.98,100.57) .. (199.98,91.77) -- cycle ;
\draw    (181.66,128.95) -- (283.17,128.95) ;
\draw [shift={(285.17,128.95)}, rotate = 180] [color={rgb, 255:red, 0; green, 0; blue, 0 }  ][line width=0.75]    (10.93,-3.29) .. controls (6.95,-1.4) and (3.31,-0.3) .. (0,0) .. controls (3.31,0.3) and (6.95,1.4) .. (10.93,3.29)   ;
\draw [shift={(179.66,128.95)}, rotate = 0] [color={rgb, 255:red, 0; green, 0; blue, 0 }  ][line width=0.75]    (10.93,-3.29) .. controls (6.95,-1.4) and (3.31,-0.3) .. (0,0) .. controls (3.31,0.3) and (6.95,1.4) .. (10.93,3.29)   ;
\draw    (234.01,91.78) -- (342.39,91.92) ;
\draw [shift={(344.39,91.92)}, rotate = 180.08] [color={rgb, 255:red, 0; green, 0; blue, 0 }  ][line width=0.75]    (10.93,-3.29) .. controls (6.95,-1.4) and (3.31,-0.3) .. (0,0) .. controls (3.31,0.3) and (6.95,1.4) .. (10.93,3.29)   ;
\draw [shift={(232.01,91.77)}, rotate = 0.08] [color={rgb, 255:red, 0; green, 0; blue, 0 }  ][line width=0.75]    (10.93,-3.29) .. controls (6.95,-1.4) and (3.31,-0.3) .. (0,0) .. controls (3.31,0.3) and (6.95,1.4) .. (10.93,3.29)   ;
\draw  [fill={rgb, 255:red, 0; green, 0; blue, 0 }  ,fill opacity=1 ] (69.19,176.28) .. controls (69.19,175.51) and (69.84,174.88) .. (70.64,174.88) .. controls (71.44,174.88) and (72.09,175.51) .. (72.09,176.28) .. controls (72.09,177.06) and (71.44,177.69) .. (70.64,177.69) .. controls (69.84,177.69) and (69.19,177.06) .. (69.19,176.28) -- cycle ;
\draw  [fill={rgb, 255:red, 0; green, 0; blue, 0 }  ,fill opacity=1 ] (75.19,176.08) .. controls (75.19,175.3) and (75.84,174.68) .. (76.64,174.68) .. controls (77.44,174.68) and (78.09,175.3) .. (78.09,176.08) .. controls (78.09,176.86) and (77.44,177.48) .. (76.64,177.48) .. controls (75.84,177.48) and (75.19,176.86) .. (75.19,176.08) -- cycle ;
\draw  [fill={rgb, 255:red, 0; green, 0; blue, 0 }  ,fill opacity=1 ] (299.47,131.23) .. controls (299.47,130.46) and (300.11,129.83) .. (300.91,129.83) .. controls (301.71,129.83) and (302.36,130.46) .. (302.36,131.23) .. controls (302.36,132.01) and (301.71,132.64) .. (300.91,132.64) .. controls (300.11,132.64) and (299.47,132.01) .. (299.47,131.23) -- cycle ;
\draw  [fill={rgb, 255:red, 0; green, 0; blue, 0 }  ,fill opacity=1 ] (299.12,124.53) .. controls (299.12,123.75) and (299.77,123.12) .. (300.57,123.12) .. controls (301.36,123.12) and (302.01,123.75) .. (302.01,124.53) .. controls (302.01,125.3) and (301.36,125.93) .. (300.57,125.93) .. controls (299.77,125.93) and (299.12,125.3) .. (299.12,124.53) -- cycle ;
\draw  [fill={rgb, 255:red, 0; green, 0; blue, 0 }  ,fill opacity=1 ] (214.35,90.57) .. controls (214.35,89.79) and (215,89.16) .. (215.8,89.16) .. controls (216.6,89.16) and (217.25,89.79) .. (217.25,90.57) .. controls (217.25,91.34) and (216.6,91.97) .. (215.8,91.97) .. controls (215,91.97) and (214.35,91.34) .. (214.35,90.57) -- cycle ;
\draw  [fill={rgb, 255:red, 0; green, 0; blue, 0 }  ,fill opacity=1 ] (213.72,96.46) .. controls (213.72,95.68) and (214.37,95.05) .. (215.17,95.05) .. controls (215.96,95.05) and (216.61,95.68) .. (216.61,96.46) .. controls (216.61,97.23) and (215.96,97.86) .. (215.17,97.86) .. controls (214.37,97.86) and (213.72,97.23) .. (213.72,96.46) -- cycle ;
\draw    (425.75,92.73) .. controls (443.98,89.71) and (451.54,87.39) .. (472.38,64.07) .. controls (493.23,40.75) and (521.89,43.34) .. (536.87,58.89) .. controls (551.85,74.43) and (553.41,81.56) .. (553.8,92.57) .. controls (554.19,103.58) and (550.81,130.3) .. (509.77,130.3) .. controls (468.74,130.3) and (456.49,96.46) .. (420.67,98.4) .. controls (384.84,100.34) and (375.07,139.21) .. (387.45,163.17) .. controls (399.82,187.14) and (428.48,192.32) .. (447.37,175.48) .. controls (466.26,158.64) and (469.39,134.19) .. (511.73,135.48) ;
\draw  [draw opacity=0][fill={rgb, 255:red, 155; green, 155; blue, 155 }  ,fill opacity=0.47 ] (496.44,131.09) .. controls (496.44,122.3) and (503.61,115.16) .. (512.45,115.16) .. controls (521.3,115.16) and (528.47,122.3) .. (528.47,131.09) .. controls (528.47,139.89) and (521.3,147.02) .. (512.45,147.02) .. controls (503.61,147.02) and (496.44,139.89) .. (496.44,131.09) -- cycle ;
\draw  [draw opacity=0][fill={rgb, 255:red, 155; green, 155; blue, 155 }  ,fill opacity=0.47 ] (405.48,93.72) .. controls (405.48,84.92) and (412.65,77.79) .. (421.49,77.79) .. controls (430.34,77.79) and (437.51,84.92) .. (437.51,93.72) .. controls (437.51,102.51) and (430.34,109.64) .. (421.49,109.64) .. controls (412.65,109.64) and (405.48,102.51) .. (405.48,93.72) -- cycle ;
\draw    (387.16,130.9) -- (488.68,130.9) ;
\draw [shift={(490.68,130.9)}, rotate = 180] [color={rgb, 255:red, 0; green, 0; blue, 0 }  ][line width=0.75]    (10.93,-3.29) .. controls (6.95,-1.4) and (3.31,-0.3) .. (0,0) .. controls (3.31,0.3) and (6.95,1.4) .. (10.93,3.29)   ;
\draw [shift={(385.16,130.9)}, rotate = 0] [color={rgb, 255:red, 0; green, 0; blue, 0 }  ][line width=0.75]    (10.93,-3.29) .. controls (6.95,-1.4) and (3.31,-0.3) .. (0,0) .. controls (3.31,0.3) and (6.95,1.4) .. (10.93,3.29)   ;
\draw    (439.51,93.72) -- (547.89,93.86) ;
\draw [shift={(549.89,93.87)}, rotate = 180.08] [color={rgb, 255:red, 0; green, 0; blue, 0 }  ][line width=0.75]    (10.93,-3.29) .. controls (6.95,-1.4) and (3.31,-0.3) .. (0,0) .. controls (3.31,0.3) and (6.95,1.4) .. (10.93,3.29)   ;
\draw [shift={(437.51,93.72)}, rotate = 0.08] [color={rgb, 255:red, 0; green, 0; blue, 0 }  ][line width=0.75]    (10.93,-3.29) .. controls (6.95,-1.4) and (3.31,-0.3) .. (0,0) .. controls (3.31,0.3) and (6.95,1.4) .. (10.93,3.29)   ;
\draw  [fill={rgb, 255:red, 0; green, 0; blue, 0 }  ,fill opacity=1 ] (419.22,98.4) .. controls (419.22,97.62) and (419.87,97) .. (420.67,97) .. controls (421.47,97) and (422.11,97.62) .. (422.11,98.4) .. controls (422.11,99.18) and (421.47,99.8) .. (420.67,99.8) .. controls (419.87,99.8) and (419.22,99.18) .. (419.22,98.4) -- cycle ;
\draw    (511.73,135.48) .. controls (540.52,134.71) and (543.97,142.61) .. (557.32,149.09) .. controls (570.67,155.56) and (590.21,155.56) .. (605.19,146.82) .. controls (620.18,138.07) and (629.62,112.81) .. (628.97,93.7) .. controls (628.32,74.6) and (618.22,59.37) .. (602.59,46.74) .. controls (586.96,34.11) and (566.44,26.66) .. (546.25,26.34) .. controls (526.06,26.02) and (503.26,27.63) .. (485.67,33.46) .. controls (468.08,39.29) and (459.94,50.31) .. (451.8,63.26) .. controls (443.66,76.22) and (440.6,85.9) .. (417.93,90.47) ;
\draw  [fill={rgb, 255:red, 0; green, 0; blue, 0 }  ,fill opacity=1 ] (416.48,90.47) .. controls (416.48,89.69) and (417.13,89.06) .. (417.93,89.06) .. controls (418.73,89.06) and (419.38,89.69) .. (419.38,90.47) .. controls (419.38,91.24) and (418.73,91.87) .. (417.93,91.87) .. controls (417.13,91.87) and (416.48,91.24) .. (416.48,90.47) -- cycle ;
\draw  [fill={rgb, 255:red, 0; green, 0; blue, 0 }  ,fill opacity=1 ] (424.3,92.73) .. controls (424.3,91.96) and (424.95,91.33) .. (425.75,91.33) .. controls (426.55,91.33) and (427.19,91.96) .. (427.19,92.73) .. controls (427.19,93.51) and (426.55,94.14) .. (425.75,94.14) .. controls (424.95,94.14) and (424.3,93.51) .. (424.3,92.73) -- cycle ;
\draw  [fill={rgb, 255:red, 0; green, 0; blue, 0 }  ,fill opacity=1 ] (511.01,129.69) .. controls (511.01,128.91) and (511.65,128.28) .. (512.45,128.28) .. controls (513.25,128.28) and (513.9,128.91) .. (513.9,129.69) .. controls (513.9,130.46) and (513.25,131.09) .. (512.45,131.09) .. controls (511.65,131.09) and (511.01,130.46) .. (511.01,129.69) -- cycle ;
\draw  [fill={rgb, 255:red, 0; green, 0; blue, 0 }  ,fill opacity=1 ] (511.73,135.48) .. controls (511.73,134.71) and (512.37,134.08) .. (513.17,134.08) .. controls (513.97,134.08) and (514.62,134.71) .. (514.62,135.48) .. controls (514.62,136.26) and (513.97,136.89) .. (513.17,136.89) .. controls (512.37,136.89) and (511.73,136.26) .. (511.73,135.48) -- cycle ;
\draw    (558.89,93.72) -- (626.97,93.7) ;
\draw [shift={(628.97,93.7)}, rotate = 179.99] [color={rgb, 255:red, 0; green, 0; blue, 0 }  ][line width=0.75]    (10.93,-3.29) .. controls (6.95,-1.4) and (3.31,-0.3) .. (0,0) .. controls (3.31,0.3) and (6.95,1.4) .. (10.93,3.29)   ;
\draw [shift={(556.89,93.72)}, rotate = 359.99] [color={rgb, 255:red, 0; green, 0; blue, 0 }  ][line width=0.75]    (10.93,-3.29) .. controls (6.95,-1.4) and (3.31,-0.3) .. (0,0) .. controls (3.31,0.3) and (6.95,1.4) .. (10.93,3.29)   ;

\draw (67.67,67.98) node [anchor=north west][inner sep=0.75pt]    {$\beta $};
\draw (220.55,129.41) node [anchor=north west][inner sep=0.75pt]    {$\beta $};
\draw (291.36,71.59) node [anchor=north west][inner sep=0.75pt]    {$\beta $};
\draw (426.72,132.69) node [anchor=north west][inner sep=0.75pt]    {$\beta $};
\draw (496.45,73.27) node [anchor=north west][inner sep=0.75pt]    {$\beta $};
\draw (586.29,73.54) node [anchor=north west][inner sep=0.75pt]    {$\beta $};
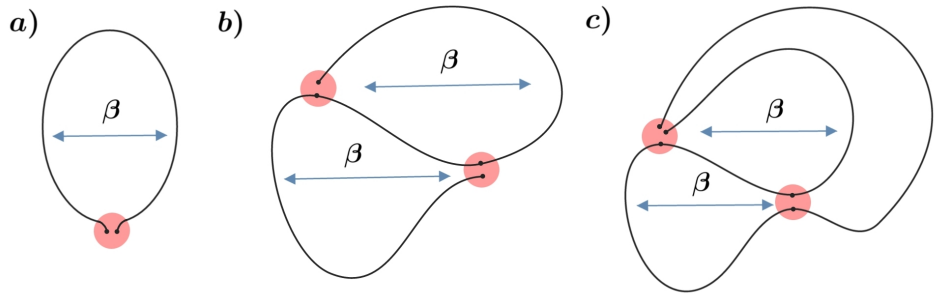
\draw (0.5,210.83) node [anchor=north west][inner sep=0.75pt]   [align=left] {Figure 2: links of three $\displaystyle \beta $-snakes. Points inside the shaded disks represent arcs with tangency order higher than $\displaystyle \beta $. };
\draw (14,25.4) node [anchor=north west][inner sep=0.75pt]    {$a)$};
\draw (174,27.4) node [anchor=north west][inner sep=0.75pt]    {$b)$};
\draw (374,26.4) node [anchor=north west][inner sep=0.75pt]    {$c)$};

\end{tikzpicture}
\end{center}


\subsection{Lipschitz functions on a Lipschitz normally embedded $\beta$-H\"older triangle}\label{Section: Lipschitz functions on a NE Holder triangle}

The results in this subsection were proved in Section 3 of \cite{GabrielovSouza} and can be used in this text without any adaptation.

\begin{Def}\label{Def:B_beta and H_beta}
	 Let $(T,0)\subset (\R^{n},0)$ be a LNE $\beta$-H\"older triangle, and $f\colon (T,0)\rightarrow (\R,0)$ a Lipschitz function such that $\ord_{\gamma}f\ge \beta$ for all $\gamma \in V(T)$. We define the following sets of arcs:
	$$B_{\beta}=B_{\beta}(f)=\{\gamma \in G(T)\mid \ord_{\gamma}f=\beta\}$$ and
	$$H_{\beta}=H_{\beta}(f)=\{\gamma \in G(T)\mid \ord_{\gamma}f>\beta\}.$$
\end{Def}

\begin{Lem}\label{Lem:itord>beta in H_{beta} and B_{beta} implies same ordf}
	Let $T$ and $f$ be as in Definition \ref{Def:B_beta and H_beta}, and let $\gamma \in B_{\beta}$ and $\gamma' \in H_\beta$. Then $\tord(\gamma,\gamma')=\beta$.
\end{Lem}

\begin{Prop}\label{Prop:arcs in B_{beta} are generic}
	Let $T$ and $f$ be as in Definition \ref{Def:B_beta and H_beta}, and let $\{T_{i}\}_{i=1}^{p}$ be a minimal pizza on $T$ associated with $f$. Let $B_0=G(T_1)$, $B_p=G(T_p)$ and, for $0<i<p$, $B_i=G(T_i\cup T_{i+1})$. Then
	\begin{enumerate}
		\item If $\ord_{\lambda_{i}}f=\beta$ then $B_i$ is a perfect $\beta$-zone maximal in $B_\beta$.
		\item If $p>1$ then the set $B_\beta$ is the disjoint union of all perfect $\beta$-zones $B_i$ such that  $\ord_{\lambda_{i}}f=\beta$.
	\end{enumerate}
\end{Prop}

\begin{Prop}\label{Prop:maximal zones in H_{beta}}
	Let $T$ and $f$ be as in Definition \ref{Def:B_beta and H_beta}, and let $\{T_{i}=T(\lambda_{i-1},\lambda_i)\}_{i=1}^p$ be a minimal pizza associated with $f$. Then
	\begin{enumerate}
		\item  For each $i\in \{1,\ldots,p-1\}$ such that $\lambda_i\in G(T)$ and $\ord_{\lambda_i} f>\beta$,
		
		$H_i=\{\gamma \in G(T)\mid \tord(\gamma,\lambda_{i})>\beta\}$ is an open $\beta$-complete zone in $H_\beta$.
		\item  For each $i\in \{1,\ldots,p\}$, if $\beta_{i}=\beta$ and $\ord_{\lambda_l}f>\beta$ for $l=i-1,i$ then $H'_i=G(T_i)$ is a perfect $\beta$-zone in $H_\beta$.
		\item Each maximal zone $Z\subset H_\beta$ is the union of some zones as in items $(1)$ and $(2)$.
		\item The set $H_{\beta}$ is a finite union of maximal $\beta$-zones.
	\end{enumerate}
\end{Prop}

\section{Circular snakes}\label{Section: Circular Snakes}
In this section are presented the notions of circular snakes and the Lipschitz invariant zones of their Valette's link. Unless otherwise specified, $X$ will denote a surface with connected link. 

\begin{Def}\label{Def: circular snake}
	 A $\beta$-surface $X$ with connected link is called a circular $\beta$-snake if its Valette's link coincides with the set of its abnormal arcs, i.e., $V(X) = \Abn(X)$.  
\end{Def}

\begin{center}

\tikzset{every picture/.style={line width=0.75pt}} 

\begin{tikzpicture}[x=0.75pt,y=0.75pt,yscale=-1,xscale=1]

\draw    (244.7,64.09) .. controls (216.77,65.08) and (192.92,78.67) .. (191.94,116.56) .. controls (190.96,154.45) and (201.17,171.08) .. (233.97,185.88) .. controls (266.77,200.68) and (288.22,198.33) .. (315.33,199.22) ;
\draw    (244.7,64.09) .. controls (278.9,65.06) and (313.56,80.69) .. (313.56,109.89) .. controls (313.56,139.09) and (266.37,141.88) .. (254.77,159.08) .. controls (243.17,176.28) and (274.77,192.28) .. (315.97,193.88) ;
\draw  [draw opacity=0][fill={rgb, 255:red, 155; green, 155; blue, 155 }  ,fill opacity=0.47 ] (299.67,196.45) .. controls (299.67,187.66) and (306.84,180.52) .. (315.69,180.52) .. controls (324.53,180.52) and (331.7,187.66) .. (331.7,196.45) .. controls (331.7,205.25) and (324.53,212.38) .. (315.69,212.38) .. controls (306.84,212.38) and (299.67,205.25) .. (299.67,196.45) -- cycle ;
\draw    (195.53,109.85) -- (311.16,109.66) ;
\draw [shift={(313.16,109.66)}, rotate = 179.91] [color={rgb, 255:red, 0; green, 0; blue, 0 }  ][line width=0.75]    (10.93,-3.29) .. controls (6.95,-1.4) and (3.31,-0.3) .. (0,0) .. controls (3.31,0.3) and (6.95,1.4) .. (10.93,3.29)   ;
\draw [shift={(193.53,109.85)}, rotate = 359.91] [color={rgb, 255:red, 0; green, 0; blue, 0 }  ][line width=0.75]    (10.93,-3.29) .. controls (6.95,-1.4) and (3.31,-0.3) .. (0,0) .. controls (3.31,0.3) and (6.95,1.4) .. (10.93,3.29)   ;
\draw    (388.75,64.09) .. controls (416.68,65.08) and (440.53,78.67) .. (441.51,116.56) .. controls (442.49,154.45) and (432.28,171.08) .. (399.48,185.88) .. controls (366.68,200.68) and (340.67,197.89) .. (315.33,199.22) ;
\draw    (388.75,64.09) .. controls (354.56,65.06) and (318.44,79.8) .. (318.44,109) .. controls (318.44,138.2) and (367.08,141.88) .. (378.68,159.08) .. controls (390.28,176.28) and (357.17,192.28) .. (315.97,193.88) ;
\draw  [draw opacity=0][fill={rgb, 255:red, 155; green, 155; blue, 155 }  ,fill opacity=0.47 ] (299.67,108.82) .. controls (299.67,100.02) and (306.84,92.89) .. (315.69,92.89) .. controls (324.53,92.89) and (331.7,100.02) .. (331.7,108.82) .. controls (331.7,117.61) and (324.53,124.74) .. (315.69,124.74) .. controls (306.84,124.74) and (299.67,117.61) .. (299.67,108.82) -- cycle ;
\draw    (324.03,109.35) -- (439.66,109.16) ;
\draw [shift={(441.66,109.16)}, rotate = 179.91] [color={rgb, 255:red, 0; green, 0; blue, 0 }  ][line width=0.75]    (10.93,-3.29) .. controls (6.95,-1.4) and (3.31,-0.3) .. (0,0) .. controls (3.31,0.3) and (6.95,1.4) .. (10.93,3.29)   ;
\draw [shift={(322.03,109.35)}, rotate = 359.91] [color={rgb, 255:red, 0; green, 0; blue, 0 }  ][line width=0.75]    (10.93,-3.29) .. controls (6.95,-1.4) and (3.31,-0.3) .. (0,0) .. controls (3.31,0.3) and (6.95,1.4) .. (10.93,3.29)   ;
\draw    (255.53,168.83) -- (377.53,167.87) ;
\draw [shift={(379.53,167.85)}, rotate = 179.55] [color={rgb, 255:red, 0; green, 0; blue, 0 }  ][line width=0.75]    (10.93,-3.29) .. controls (6.95,-1.4) and (3.31,-0.3) .. (0,0) .. controls (3.31,0.3) and (6.95,1.4) .. (10.93,3.29)   ;
\draw [shift={(253.53,168.85)}, rotate = 359.55] [color={rgb, 255:red, 0; green, 0; blue, 0 }  ][line width=0.75]    (10.93,-3.29) .. controls (6.95,-1.4) and (3.31,-0.3) .. (0,0) .. controls (3.31,0.3) and (6.95,1.4) .. (10.93,3.29)   ;

\draw (246.7,88.58) node [anchor=north west][inner sep=0.75pt]    {$\beta $};
\draw (375.2,88.08) node [anchor=north west][inner sep=0.75pt]    {$\beta $};
\draw (314.2,147.08) node [anchor=north west][inner sep=0.75pt]    {$\beta $};
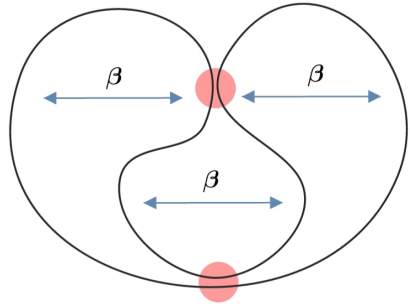
\draw (10,222.83) node [anchor=north west][inner sep=0.75pt]   [align=left] {Figure 3: An example of a non-Lipschitz normally embedded circular $\displaystyle \beta $-snake. Points inside the shaded disks\\represents arcs with the tangency order higher than $\displaystyle \beta $.};
\end{tikzpicture}
\end{center}

\begin{Exam}
	 Any $\beta$-horn $X$ is a Lipschitz normally embedded circular $\beta$-snake. Below it is shown the link of a not LNE circular $\beta$-snake. Moreover, notice that a surface with link as in Figure 1a is not a circular snake, since it contains a Lipschitz singular arc.
\end{Exam}

\begin{Prop}\label{Prop: link of a circular snake}
	If $X$ is a circular snake then the link of $X$ is homeomorphic to $S^1$.
\end{Prop}
\begin{proof}
	Otherwise $X$ would contain Lipschitz singular arcs (see Remark \ref{Rem: boundary arcs are Lips sing arcs}).
\end{proof}

\begin{Prop}\label{Prop: LNE circular snakes}
	If $X$ is a Lipschitz normally embedded circular snake then $X$ is outer bi-Lipschitz homeomorphic to a $\beta$-horn with $\beta\ge 1$.
\end{Prop}
\begin{proof}
	This follows from Theorem 8.3 in \cite{birbrairOminimal}.
\end{proof}

\begin{remark}\label{Remark: assumption CS are not NE}
    Proposition \ref{Prop: LNE circular snakes} gives a complete classification of LNE circular snakes up to outer bi-Lipschitz homeomorphism. Then, from now on, we shall focus on the Lipschitz geometry of circular snakes which are not LNE. Therefore, except otherwise specified, all circular snakes are assumed to be non LNE.
\end{remark}

\subsection{Circular snakes and their pancake decompositions}

Let $X$ be a circular snake, and let $\{X_{k}\}_{k=1}^{p}$ be a reduced pancake decomposition of $X$. If $X$ is not LNE, then $p\ge 2$ and each $X_{k}$ is a H\"older triangle (see Remark \ref{Rem:pancake of holder triangle is holder triangle}). Moreover, Proposition \ref{Prop: link of a circular snake} implies that $\gamma_{0}=\gamma_p$. 
Thus, from now on we will assume that $\gamma_{j-1}=\gamma_{p-1}$ if $j=0$ and $\gamma_{j+1} = \gamma_1$ if $j=p$. In particular, $X_{j-1} = X_p$ if $j=1$ and $X_{j+1} = X_1$ if $j=p$.

The following Lemma is proved in \cite{GabrielovSouza} as Lemma 4.7. 

\begin{Lem}\label{Lem: maximal abn zones and triangle exponents}
	Let $X$ be a surface germ and $A\subset V(X)$ a maximal abnormal $\beta$-zone. Let $\gamma \in A$, and let $T=T(\lambda,\gamma)\subset X$ and $T'=T(\gamma,\lambda')\subset X$ be Lipschitz normally embedded $\alpha$-H\"older triangles such that $T\cap T'=\gamma$ and $\tord(\lambda,\lambda')>\itord(\lambda,\lambda')$. Then $\alpha\le \beta$.
\end{Lem}

\begin{Lem}\label{Lem:minimal pancake decomp of a snake}
	Let $X$ be a circular $\beta$-snake, and let $\{X_{k}\}_{k=1}^{p}$ be a reduced pancake decomposition of $X$. If $X$ is not LNE, then each $X_{k}$ is a $\beta$-H\"older triangle.
\end{Lem}
\begin{proof}
	Let $\{X_{k}=T(\gamma_{k-1},\gamma_{k})\}_{k=1}^{p}$ be a reduced pancake decomposition of $X$. By Remark \ref{Rem:pancake of holder triangle is holder triangle} it only remains to prove that the exponent of each $X_k$ is $\beta$. 
	
	Consider $\gamma_j$ for some $j=0,\ldots,k$. As $X$ is a circular $\beta$-snake, $\gamma_j$ is abnormal and if $\lambda, \lambda' \in V(X)$ are arcs such that $T=T(\lambda,\gamma_j)$ and $T'=T(\gamma_j,\lambda')$ are LNE $\alpha$-H\"older triangles with $\tord(\lambda, \lambda') > \itord(\lambda, \lambda')$, then $\alpha \le \beta =\mu(X)$, by Lemma \ref{Lem: maximal abn zones and triangle exponents}, because $\Abn(X) = X$. Indeed, we have $\alpha = \beta$, since $T(\lambda,\gamma_j)\subset X$ implies that $\alpha \ge \beta$. 
	
	As $\{X_{k}\}$ is a reduced pancake decomposition, we may assume that $\lambda \in X_{j-1}$ and $\lambda'\in X_{j+1}$. Consequently, $\mu(X_{j+1})=\beta$. As $j$ could be taken arbitrarily, we have $\mu(X_{j+1})=\beta$ for any $j=0,\ldots, k-1$.
\end{proof}

\begin{remark}\label{Rem: adjacent pancakes}
Let $\{X_k\}$ be a reduced pancake decomposition of a circular snake $X$. If $\gamma, \gamma' \in V(X)$ are such that $\tord(\gamma, \gamma') > \itord(\gamma, \gamma')$ then we can assume that $\gamma\subset X_l$ and $\gamma'\subset X_{l+1}$ for some $l$. By Lemma \ref{Lem:minimal pancake decomp of a snake}, one has $\mu(X_l)=\mu(X_{l+1})=\beta$. Thus, enlarging the pancakes $X_l$ and $X_{l+1}$ attaching $\beta$-H\"older triangles to one of their boundary arcs, if necessary, we can further assume that $\gamma \in G(X_l)$ and $\gamma'\in G(X_{l+1})$.
\end{remark}

\begin{Prop}\label{Prop:CS are weakly NE}
	Let $X$ be a circular $\beta$-snake. For any two arcs $\gamma$ and $\gamma'$ in $V(X)$ such that $\tord(\gamma, \gamma') > \itord(\gamma, \gamma')$, we have $\itord(\gamma, \gamma')=\beta$.
\end{Prop}
\begin{proof}
	This is an immediate consequence of Lemma \ref{Lem: maximal abn zones and triangle exponents}. 
\end{proof}

\subsection{Segments and nodal zones}

If $X$ is Lipschitz normally embedded then all the arcs in $X$ are segment arcs (see Definition \ref{Def:segment-nodal arcs-zones}) and the results in this subsection are trivially true. Thus, we may assume that $X$ is not Lipschitz normally embedded along this subsection (see Remark \ref{Remark: assumption CS are not NE}). Note that the definitions do not require this assumption.

\begin{Def}\label{Def:horn-neighbourhood}
	 Let $X$ be a surface and $\gamma \subset X$ an arc. For $a>0$ and $1\leq\alpha\in\F$, the $(a,\alpha)$-\em horn neighborhood \em of $\gamma$ in $X$ is defined as follows: 
 $$
 HX_{a,\alpha}(\gamma) = \bigcup_{0\le t \ll1} X\cap S(0,t)\cap \overline{B}(\gamma(t),at^{\alpha}),
 $$
	where $S(0,t)=\{x \in \mathbb{R}^{n}\mid ||x||=t\}$ and $\overline{B}(y,R) = \{x \in \mathbb{R}^{n}\mid ||x - y||\le R\}$.
\end{Def}

\begin{remark}
	 When there is no confusion about the surface $X$ being considered, one writes $H_{a,\alpha}(\gamma)$ instead of $HX_{a,\alpha}(\gamma)$.
\end{remark}

\begin{Def}\label{Def of multiplicity}
	 If $X$ is a circular $\beta$-snake and $\gamma$ an arc in $X$, the \em multiplicity \em of $\gamma$, denoted by $m_{X}(\gamma)$ (or just $m(\gamma)$), is defined as the number of connected components of $HX_{a,\beta}(\gamma)\setminus \{0\}$ for $a>0$ small enough.
\end{Def}

\begin{Def}\label{DEF: constant zone}
	 Let $X$ be a $\beta$-surface and $Z \subset V(X)$ a zone. We say that $Z$ is a \em constant zone \em of multiplicity $q$ (notation $m(Z)=q$) if all arcs in $Z$ have the same multiplicity $q$.
\end{Def}

\begin{Def}\label{Def:segment-nodal arcs-zones}
	 Let $X$ be a circular $\beta$-snake and $\gamma \subset X$ an arc. We say that $\gamma$ is a \em segment arc \em if there exists a $\beta$-H\"older triangle $T \subset X$ such that $\gamma$ is a generic arc of $T$ and $V(T)$ is a constant zone. Otherwise $\gamma$ is a \em nodal arc \em. We denote the set of segment arcs and the set of nodal arcs in $X$ by $\mathbf{S}(X)$ and $\mathbf{N}(X)$, respectively. A \em segment \em of $X$ is a maximal zone in $\mathbf{S}(X)$. A \em nodal zone \em of $X$ is a maximal zone in $\mathbf{N}(X)$. We write $\Seg_{\gamma}$ or $\Nod_{\gamma}$ for a segment or a nodal zone containing an arc $\gamma$.
\end{Def}

\begin{Prop}\label{Prop:segment of a CS is a perfect zone}
	If $X$ is a circular $\beta$-snake then each segment of $X$ is a closed perfect $\beta$-zone.
\end{Prop}
\begin{proof}
	Given arcs $\gamma$ and $\gamma'$ in a segment $S$ of $X$, by Definition \ref{Def:segment-nodal arcs-zones}, there exist $\beta$-H\"older triangles $T=T(\gamma_1,\gamma_2)$ and $T'=T(\gamma'_{1},\gamma'_{2})$ such that $\gamma$ and $\gamma'$ are generic arcs of $T$ and $T'$, respectively, and $V(T)\subset S$ and $V(T')\subset S$ are constant zones. This immediately implies that $S$ is a closed $\beta$-zone.
	
	 In order to prove that $S$ is a perfect zone, assume that $\gamma_2$ and $\gamma'_1$ are in $T(\gamma,\gamma')$ and $V(T(\gamma,\gamma')) \subset S$. Thus, $T(\gamma_1,\gamma'_2)$ is a $\beta$-H\"older triangle with Valette link contained in $S$ and containing both $\gamma$ and $\gamma'$ as generic arcs.
\end{proof}

\begin{remark}
	 The proof of Proposition \ref{Prop:segment of a CS is a perfect zone} also works to prove Proposition 4.20 of \cite{GabrielovSouza}. There it is given a proof based on an old version (not equivalent to the current one) of the definition of perfect zone, which depended on the order of the zone (a $\beta$-zone $Z$ is perfect if for any arc $\gamma \in Z$ there is a $\beta$-H\"older triangle $T$ with $V(T)\subset Z$ such that $\gamma$ is a generic arc of $T$). 
\end{remark}

\begin{Lem}\label{Lem:multipizza and generic arcs of multipizza triangle for CS}
	Let $X$ be a circular $\beta$-snake and $\{X_{k}\}_{k=1}^{p}$ a pancake decomposition of $X$. Let $T=X_{j}$ be one of the pancakes and consider the set of germs of Lipschitz functions $f_{k}\colon (T,0) \rightarrow (\mathbb{R},0)$ given by $f_{k}(x)=d(x,X_{k})$. If the set $\{T_{i}\}$ of $\beta_i$-H\"older triangles is a multipizza on $T$ associated with $\{f_{1},\ldots, f_{p}\}$ then, for each $i$, the following holds:
	\begin{enumerate}
		\item $\mu_{ik}(\ord_{\gamma}f_{k})=\beta_{i}$ for all $k \in \{1, \ldots,p\}$ and all $\gamma \in G(T_{i})$, thus $G(T_{i})$ is a constant zone.
		\item $V(T_{i})$ intersects at most one segment of $X$.
		\item If $V(T_i)$ is contained in a segment then it is a constant zone.
	\end{enumerate}	
\end{Lem}
\begin{proof}
	$(1)$. This is an immediate consequence of Definition \ref{Def:Pizza decomp} and Proposition \ref{Prop:width function properties elementary triangle}.
	
	$(2)$. If $\beta_{i}>\beta$ and $V(T_{i})$ intersects a segment $S$, then $V(T_{i})\subset S$, since $S$ is a perfect $\beta$-zone by Proposition \ref{Prop:segment of a CS is a perfect zone}.
	
	Let $\beta_{i}=\beta$. Suppose that $V(T_{i})$ intersects distinct segments $S$ and $S'$. As each segment is a perfect $\beta$-zone, we can choose arcs $\lambda \in S$ and $\lambda' \in S'$ so that $\lambda, \lambda' \in G(T_{i})$. Let $T'=T(\lambda,\lambda')$ be a H\"older triangle contained in $T_i$. By item $(1)$ of this Lemma, all arcs in $G(T')$ have the same multiplicity. It follows from Definition \ref{Def:segment-nodal arcs-zones} that each arc in $T'$ is a segment arc. Thus, $\lambda$ and $\lambda'$ belong to the same segment, a contradiction.
	
	$(3)$ This a consequence of Definition \ref{Def:segment-nodal arcs-zones} and item $(1)$ of this Lemma. 
\end{proof}

\begin{Prop}\label{Prop:segment has finitely many zones}
	Let $X$ be a circular $\beta$-snake. Then
	\begin{enumerate}
		\item There are no adjacent segments in $X$.
		\item $X$ has finitely many segments.
	\end{enumerate}
	
\end{Prop}
\begin{proof}
	$(1)$ This is an immediate consequence of Proposition \ref{Prop:segment of a CS is a perfect zone} and Lemma \ref{Lem: there are no adjacent perfect zones}.
	
	$(2)$ Let $\{X_{k}\}_{k=1}^{p}$ be a pancake decomposition of $X$. It is enough to show that, for each pancake $X_{j}$, $V(X_{j})$ intersects with finitely many segments. But this follows from Lemma \ref{Lem:multipizza and generic arcs of multipizza triangle for CS}, since there are finitely many H\"older triangles in a multipizza.
\end{proof}

\begin{Lem}\label{Lem:itord > beta imples same multiplicity}
	Let $X$ be a circular $\beta$-snake. Then, any two arcs in $V(X)$ with inner tangency order higher than $\beta$ have the same multiplicity.
\end{Lem}
\begin{proof}
	Let $\{X_{k}\}_{k=1}^{p}$ be a pancake decomposition of $X$, $T=X_{j}$ one of the pancakes and $\{T_{i}\}$ a multipizza associated with $\{f_1,\ldots, f_p\}$ as in Lemma \ref{Lem:multipizza and generic arcs of multipizza triangle for CS}. Consider arcs $\gamma$ and $\gamma'$ in $V(X)$ such that $\itord(\gamma,\gamma')>\beta$ and $\gamma \in V(T)$. We can suppose that $\gamma,\gamma' \in V(T)$, otherwise we can just replace $\gamma'$ by the boundary arc of $T$ in $T(\gamma,\gamma')$.
	
	It is enough to show that for each $l$ we have $\ord_{\gamma}f_{l}>\beta$ if and only if $\ord_{\gamma'}f_{l}>\beta$. This follows from Lemma \ref{Lem:itord>beta in H_{beta} and B_{beta} implies same ordf}.
\end{proof}

\begin{Cor}\label{Cor:segments and nodal zones are constant zones}
	Let $X$ be a circular $\beta$-snake. Then, all segments and all nodal zones of $X$ are constant zones.
\end{Cor}
\begin{proof}
	Let $\{X_{k}=T(\gamma_{k-1},\gamma_{k})\}_{k=1}^{p}$ be a reduced pancake decomposition of $X$, and $\{T_i\}$ a multipizza on $T=X_j$ associated with $\{f_{1},\ldots, f_{p}\}$ as in Lemma \ref{Lem:multipizza and generic arcs of multipizza triangle for CS}.
	
	Let $S$ be a segment of $X$. One wants to prove that any two arcs of $S$ have the same multiplicity. Consider two arcs $\lambda,\lambda'\in S$. Replacing, if necessary, one of the arcs $\lambda$, $\lambda'$ by one of the boundary arcs of $T$,  one can assume that $\lambda,\lambda'\in V(T)$. Moreover, by Lemma \ref{Lem:itord > beta imples same multiplicity}, it is enough to consider the case where $\itord(\lambda,\lambda') = \beta$. Thus, if $\lambda \in T_i$ and $\lambda'\in T_{i+l}$, for some $l\ge 0$, it follows from Lemma \ref{Lem:multipizza and generic arcs of multipizza triangle for CS} that $m(V(T_{i+1}))=\cdots=m(V(T_{i+l-1}))$, since $V(T_{i+1}),\ldots ,V(T_{i+l-1})$ are subsets of $S$. Finally, as $m(\gamma_i)=m(V(T_{i+1}))=m(\gamma_{i+l-2})$ and $\gamma_i, \gamma_{i+l-2} \in S$, it follows that $m(G(T_i)) = m(G(T_{i+l}))$. Consequently, since $\lambda$ and $\lambda'$ are segment arcs, $m(\lambda)=m(G(T_i))=m(V(T_{i+1}))=m(V(T_{i+l-1}))= m(G(T_{i+l}))=m(\lambda')$.

	Let now $N$ be a nodal zone of $X$. Consider two arcs $\lambda,\lambda'\in N$ and a H\"older triangle $T'=T(\lambda,\lambda')$ such that $V(T') \subset N$. Assume, without loss of generality, that $\lambda,\lambda'\in V(T)$ (otherwise one can replace one of them by its ``closest'' boundary arc of $T$). If $\itord(\lambda,\lambda')=\beta$ then $G(T')\cap G(T_i)\ne \emptyset$ for some $i$ such that $\beta_i=\beta$. As $G(T')$ and $G(T_i)$ are perfect $\beta$-zones, $G(T')\cap G(T_i)$ is also a perfect $\beta$-zone. Item $(1)$ of Lemma \ref{Lem:multipizza and generic arcs of multipizza triangle for CS} implies that $G(T')\cap G(T_i)$ contains a segment arc, since $G(T_i)$ is a constant zone, a contradiction with $V(T')\subset N$. Thus, $\itord(\lambda,\lambda')>\beta$ and $m(\lambda)=m(\lambda')$ by Lemma \ref{Lem:itord > beta imples same multiplicity}.
\end{proof}

\begin{remark}\label{Rem:non-closed zones are constant zones}
	 If $X$ is a circular $\beta$-snake then any open zone $Z$ in $V(X)$, and any zone $Z'$ of order $\beta'>\beta$, is a constant zone.
\end{remark}

\begin{Prop}\label{Prop:nodal zones are finite}
	Let $X$ be a circular $\beta$-snake. Then
	\begin{enumerate}
		\item  For any nodal arc $\gamma$ we have $\Nod_{\gamma}=\{\gamma'\in V(X)\mid \itord(\gamma,\gamma')>\beta\}$. In particular, a nodal zone is an open perfect $\beta$-complete zone (see Definition \ref{Def:complete and open complete zone}).
		\item There are no adjacent nodal zones.
		\item There are finitely many nodal zones in $V(X)$.
	\end{enumerate}
\end{Prop}
\begin{proof}
	$(1)$ Let $\gamma \in V(X)$ be a nodal arc. Given $\gamma'\in V(X)$, if $\itord(\gamma,\gamma')=\beta$ then $\gamma'\notin \Nod_{\gamma}$. Indeed, if $\gamma'\in \Nod_{\gamma}$ and $\itord(\gamma,\gamma')=\beta$, let $T(\gamma,\gamma')$ be a H\"older triangle such that $V(T(\gamma,\gamma'))\subset \Nod_{\gamma}$. As $\Nod_{\gamma}$ is a constant zone, by Corollary \ref{Cor:segments and nodal zones are constant zones}, every arc in $G(T(\gamma,\gamma'))$ is a segment arc, a contradiction with $\Nod_{\gamma}$ being a zone. Thus, a nodal zone is completely determined by any one of its arcs, i.e., $\Nod_{\gamma}=\{\gamma'\in V(X)\mid \itord(\gamma,\gamma')>\beta\}$. Therefore, any nodal zone is an open $\beta$-complete zone.
	
	It only remains to be proved that a nodal zone $N$ is perfect. Given two distinct arcs $\gamma, \gamma' \in N$, let $\beta'=\itord(\gamma,\gamma')$. Since $\Nod_{\gamma}=\{\theta\in V(X)\mid \itord(\gamma,\theta)>\beta\}$, one must have $\beta'>\beta$. Considering $\beta''\in \F$ such that $\beta <\beta''<\beta'$ there are arcs $\theta,\theta'\in V(X)$ such that $\itord(\theta,\gamma) = \itord(\gamma,\theta')=\beta''$ (see Definition \ref{Def:order of zone} and Remark \ref{Rem: replace tord by itord in def of order of a zone}) and $\gamma\subset T=T(\theta,\theta')$. As $\beta''>\beta$ one has $V(T)\subset N$. Finally, since $\beta''<\beta'$ it follows that $\gamma,\gamma'\in G(T)$.
	
	$(2)$ This is an immediate consequence of $(1)$ and Remark \ref{Rem:open complete zones}.

	$(3)$ It is a consequence of Proposition \ref{Prop:segment has finitely many zones} and item $(2)$ of this Proposition.
\end{proof}

\begin{Teo}\label{Teo:segment is the generic arcs from HT of nodal adjacent zones}
	Let $X$ be a circular $\beta$-snake. Then
	\begin{enumerate}
		\item The Vallete link $V(X)$ is a disjoint union of finitely many segments and nodal zones. Moreover, each segment is a closed perfect $\beta$-zone and each nodal zone is an open perfect $\beta$-complete zone.
		\item Each nodal zone in $X$ has exactly two adjacent segments, and each segment in $X$ is adjacent to exactly two nodal zones.
		\item If $N$ and $N'$ are the nodal zones adjacent to a segment $S$, then for any arcs $\gamma \subset N$ and $\gamma' \subset N'$, we have $S = G(T(\gamma,\gamma'))$, where $T(\gamma,\gamma')$ is a H\"older triangle such that $V(T(\gamma,\gamma'))\cap S\ne \emptyset$.
	\end{enumerate}
\end{Teo}
\begin{proof}
	$(1)$ and $(2)$ Propositions \ref{Prop:segment of a CS is a perfect zone}, \ref{Prop:segment has finitely many zones} and \ref{Prop:nodal zones are finite} imply that each segment is a closed perfect $\beta$-zone and each nodal zone is an open perfect $\beta$-zone and $V(X)$ is the union of finitely many segments and nodal zones. Moreover, each nodal zone in $V(X)$ could only be adjacent to a segment $S$, and vice versa. 
	
	$(3)$ Let $N$ and $N'$ be the two nodal zones adjacent to $S$ and let $\gamma \in N$ and $\gamma'\in N'$. Let $T(\gamma,\gamma')$ be a H\"older triangle such that $V(T(\gamma,\gamma'))\cap S\ne \emptyset$. Since each arc in $T(\gamma,\gamma')$ which has tangency order higher than $\beta$ with one of the boundary arcs is a nodal arc, by Proposition \ref{Prop:nodal zones are finite}, each segment arc in $T(\gamma,\gamma')$ must be in $G(T(\gamma,\gamma'))$ and vice versa.
\end{proof}

\begin{Cor}\label{multiplicidade-sem-nodos}
    If $X$ is a circular snake without nodal zones, then $m_X(\gamma)=m_X(\gamma')$, for each $\gamma, \gamma' \in V(X)$. Therefore, one can define the multiplicity $m(X)$ of $X$ as the value $m_X(\gamma)$, for each such $\gamma$.
\end{Cor}

\subsection{Segments and nodal zones with respect to a pancake}\label{Subsection: seg and nodal zones with respect}

\begin{Def}\label{Def: distance functions to pancakes}
	 Let $X$ be a circular $\beta$-snake, and $\{X_{k}\}_{k=1}^{p}$ a pancake decomposition of $X$. If $\mu(X_j)=\beta$ we define the functions $f_1,\ldots, f_p$, where $f_{k}\colon (X_{j},0)\rightarrow (\R,0)$ is given by $f_{k}(x)=d(x,X_{k})$. For each $k$ we define $m_{k}\colon V(X_{j})\rightarrow \{0,1\}$ as follows: $m_{k}(\gamma)=1$ if $\ord_{\gamma}f_{k}>\beta$ and $m_{k}(\gamma)=0$ otherwise. In particular, $m_{j}\equiv 1$.
\end{Def}

\begin{remark}\label{multiplicity formula}
	 Consider $m_1,\ldots,m_p$ as in Definition \ref{Def: distance functions to pancakes}. For each $\gamma\in G(X_{j})$ we have $m(\gamma)=\sum_{k=1}^{p}m_{k}(\gamma)$.
\end{remark}

\begin{Def}\label{DEF: constant zones with resp to a pancake}
	 Consider $m_1,\ldots,m_p$ as in Definition \ref{Def: distance functions to pancakes}. A zone $Z \subset V(X_{j})$ is \em constant with respect to \em $X_{l}$ if $m_{l}|_{Z}$ is constant.
\end{Def}

\begin{Def}\label{DEF: segments and nodal zones with resp to a pancake}
	 Let $m_1,\ldots,m_p$ be as in Definition \ref{Def: distance functions to pancakes}. Consider an arc $\gamma\in G(X_{j})$. For each $l$ we say that $\gamma$ is a \em segment arc with respect to \em $X_{l}$ if there exists a $\beta$-H\"older triangle $T$, with $T\subset X_j$, such that $\gamma$ is a generic arc of $T$ and $V(T)$ is constant with respect to $X_{l}$. Otherwise $\gamma$ is a \em nodal arc with respect to \em $X_{l}$. The set of segment arcs in $G(X_j)$ with respect to $X_l$ and the set of nodal arcs in $G(X_{j})$ with respect to $X_l$ are denoted by $\mathbf{S}_l(X_{j})$ and $\mathbf{N}_l(X_{j})$, respectively. Furthermore, a \em segment with respect to \em $X_{l}$ is a zone $S_{l,j}$ maximal in $\mathbf{S}_l(X_{j})$, and a \em nodal zone with respect to \em $X_{l}$ is a zone $N_{l,j}$ maximal in $\mathbf{N}_l(X_{j})$. We write $\Seg_{\gamma}^{l,j}$ or $\Nod_{\gamma}^{l,j}$ for a segment or a nodal zone with respect to $X_{l}$ in $G(X_j)$ containing an arc $\gamma$.
\end{Def}

\begin{remark}\label{Rem: relative segments and nodal zones and perfect and open complete zones}
	 Let $f_1,\ldots,f_p$ be as in Definition \ref{Def: distance functions to pancakes}. Propositions \ref{Prop:segment of a CS is a perfect zone}, \ref{Prop:segment has finitely many zones} and \ref{Prop:nodal zones are finite} remain valid for segments and nodal zones in $G(X_j)$ with respect to $X_{l}$. In particular, taking $f=f_{l}$ and $T=X_{j}$, segments in $G(X_j)$ with respect to $X_{l}$ are in a one-to-one correspondence with the maximal perfect zones in $B_{\beta}(f_l)$ and $H_{\beta}(f_l)$ (see Definition \ref{Def:B_beta and H_beta}). Similarly, the nodal zones with respect to $X_{l}$ are in a one-to-one correspondence with the open $\beta$-complete zones in  $H_{\beta}(f_l)$ (see Definition \ref{Def:complete and open complete zone} and Propositions \ref{Prop:arcs in B_{beta} are generic} and \ref{Prop:maximal zones in H_{beta}}).
\end{remark}

\subsection{Nodes}

\begin{Def}\label{Def: node}
Let $X$ be a circular $\beta$-snake. A \em node \em $\mathcal{N}$ in $X$ is a union of nodal zones $N_1, \ldots, N_m$ in $X$ such that for any nodal zones $N_i,N_j$ with $N_i\subset \mathcal{N}$, one has $N_j\subset \mathcal{N}$ if and only if $\tord(N_i,N_j)>\beta$. The set $\Spec(\mathcal{N})=\{q_{ij}=\tord(N_i,N_j)\mid i\ne j\}$ is called the \em spectrum \em of $\mathcal{N}$.

\end{Def}

\begin{center}

\tikzset{every picture/.style={line width=0.75pt}} 

\begin{tikzpicture}[x=0.75pt,y=0.75pt,yscale=-1,xscale=1]

\draw    (167.61,132.56) .. controls (126.26,132.56) and (121,123.02) .. (121,108.9) .. controls (121,94.79) and (166.71,90.48) .. (166.48,71.51) .. controls (166.26,52.55) and (140.77,47.36) .. (122.72,50.42) .. controls (104.68,53.47) and (87.24,65.94) .. (84.38,80.94) .. controls (81.52,95.93) and (117.39,91.19) .. (117.92,109.23) .. controls (118.45,127.28) and (82.27,120.73) .. (75.36,139.69) .. controls (68.44,158.65) and (88.74,170.1) .. (108.44,165.64) .. controls (128.14,161.17) and (125.73,137.67) .. (167.99,137.52) ;
\draw    (167.61,132.56) .. controls (208.96,132.56) and (214.98,123.02) .. (214.98,108.9) .. controls (214.98,94.79) and (169.27,90.48) .. (169.49,71.51) .. controls (169.72,52.55) and (195.21,47.36) .. (213.25,50.42) .. controls (231.29,53.47) and (248.74,65.94) .. (251.6,80.94) .. controls (254.45,95.93) and (218.72,91.6) .. (218.19,109.65) .. controls (217.67,127.7) and (253.7,120.73) .. (260.62,139.69) .. controls (267.54,158.65) and (247.23,170.1) .. (227.54,165.64) .. controls (207.84,161.17) and (210.24,137.67) .. (167.99,137.52) ;
\draw    (339.97,103.56) .. controls (340.45,119.13) and (343.83,129.18) .. (353.97,134.14) .. controls (364.1,139.09) and (382.45,139.09) .. (392.1,131.66) .. controls (401.76,124.23) and (388.24,104.77) .. (426.38,104.41) ;
\draw    (339.97,103.56) .. controls (340.45,87.99) and (343.35,77.08) .. (353.48,72.13) .. controls (363.62,67.17) and (381.97,67.17) .. (391.62,74.6) .. controls (401.28,82.03) and (387.76,101.5) .. (425.9,101.85) ;

\draw    (429.41,33.6) .. controls (408.17,33.95) and (394.46,36.43) .. (387.7,43.86) .. controls (380.94,51.29) and (380.94,64.74) .. (391.08,71.82) .. controls (401.22,78.89) and (427.77,68.98) .. (428.25,96.94) ;
\draw    (429.41,33.6) .. controls (450.65,33.95) and (465.54,36.08) .. (472.3,43.51) .. controls (479.06,50.94) and (479.06,64.38) .. (468.92,71.46) .. controls (458.78,78.54) and (432.23,68.63) .. (431.75,96.58) ;

\draw    (521,101.99) .. controls (520.52,86.42) and (517.14,76.37) .. (507,71.42) .. controls (496.86,66.46) and (478.52,66.46) .. (468.86,73.9) .. controls (459.21,81.33) and (472.72,100.79) .. (434.59,101.14) ;
\draw    (521,101.99) .. controls (520.52,117.56) and (517.62,128.47) .. (507.48,133.43) .. controls (497.34,138.38) and (479,138.38) .. (469.35,130.95) .. controls (459.69,123.52) and (473.21,104.06) .. (435.07,103.7) ;

\draw    (431.56,171.6) .. controls (452.8,171.25) and (466.51,168.77) .. (473.27,161.34) .. controls (480.02,153.91) and (480.02,140.46) .. (469.89,133.38) .. controls (459.75,126.31) and (433.2,136.22) .. (432.71,108.26) ;
\draw    (431.56,171.6) .. controls (410.31,171.25) and (395.42,169.12) .. (388.67,161.69) .. controls (381.91,154.26) and (381.91,140.82) .. (392.05,133.74) .. controls (402.18,126.66) and (428.74,136.57) .. (429.22,108.62) ;

\draw    (425.9,101.85) .. controls (429.23,102.24) and (428.26,100.47) .. (428.25,96.94) ;
\draw    (434.59,101.09) .. controls (430.61,101.45) and (431.76,99.82) .. (431.78,96.58) ;
\draw    (426.38,104.46) .. controls (430.36,104.13) and (429.2,105.63) .. (429.19,108.62) ;
\draw    (435.07,103.7) .. controls (431.74,103.32) and (432.7,105.09) .. (432.71,108.62) ;
\draw  [draw opacity=0][fill={rgb, 255:red, 155; green, 155; blue, 155 }  ,fill opacity=0.72 ] (420.04,103.25) .. controls (420.04,98.04) and (424.39,93.82) .. (429.74,93.82) .. controls (435.1,93.82) and (439.45,98.04) .. (439.45,103.25) .. controls (439.45,108.46) and (435.1,112.68) .. (429.74,112.68) .. controls (424.39,112.68) and (420.04,108.46) .. (420.04,103.25) -- cycle ;
\draw  [draw opacity=0][fill={rgb, 255:red, 155; green, 155; blue, 155 }  ,fill opacity=0.72 ] (382.34,133.74) .. controls (382.34,128.53) and (386.69,124.31) .. (392.05,124.31) .. controls (397.4,124.31) and (401.75,128.53) .. (401.75,133.74) .. controls (401.75,138.95) and (397.4,143.17) .. (392.05,143.17) .. controls (386.69,143.17) and (382.34,138.95) .. (382.34,133.74) -- cycle ;
\draw  [draw opacity=0][fill={rgb, 255:red, 155; green, 155; blue, 155 }  ,fill opacity=0.72 ] (381.38,71.82) .. controls (381.38,66.61) and (385.72,62.38) .. (391.08,62.38) .. controls (396.44,62.38) and (400.78,66.61) .. (400.78,71.82) .. controls (400.78,77.02) and (396.44,81.25) .. (391.08,81.25) .. controls (385.72,81.25) and (381.38,77.02) .. (381.38,71.82) -- cycle ;
\draw  [draw opacity=0][fill={rgb, 255:red, 155; green, 155; blue, 155 }  ,fill opacity=0.72 ] (459.22,71.46) .. controls (459.22,66.25) and (463.56,62.03) .. (468.92,62.03) .. controls (474.28,62.03) and (478.62,66.25) .. (478.62,71.46) .. controls (478.62,76.67) and (474.28,80.89) .. (468.92,80.89) .. controls (463.56,80.89) and (459.22,76.67) .. (459.22,71.46) -- cycle ;
\draw  [draw opacity=0][fill={rgb, 255:red, 155; green, 155; blue, 155 }  ,fill opacity=0.72 ] (460.18,133.38) .. controls (460.18,128.18) and (464.53,123.95) .. (469.89,123.95) .. controls (475.25,123.95) and (479.59,128.18) .. (479.59,133.38) .. controls (479.59,138.59) and (475.25,142.82) .. (469.89,142.82) .. controls (464.53,142.82) and (460.18,138.59) .. (460.18,133.38) -- cycle ;
\draw  [draw opacity=0][fill={rgb, 255:red, 155; green, 155; blue, 155 }  ,fill opacity=0.72 ] (206.18,111.72) .. controls (206.18,106.52) and (210.63,102.31) .. (216.12,102.31) .. controls (221.62,102.31) and (226.07,106.52) .. (226.07,111.72) .. controls (226.07,116.91) and (221.62,121.12) .. (216.12,121.12) .. controls (210.63,121.12) and (206.18,116.91) .. (206.18,111.72) -- cycle ;
\draw  [draw opacity=0][fill={rgb, 255:red, 155; green, 155; blue, 155 }  ,fill opacity=0.72 ] (157.71,69.79) .. controls (157.71,64.6) and (162.16,60.38) .. (167.65,60.38) .. controls (173.14,60.38) and (177.59,64.6) .. (177.59,69.79) .. controls (177.59,74.98) and (173.14,79.19) .. (167.65,79.19) .. controls (162.16,79.19) and (157.71,74.98) .. (157.71,69.79) -- cycle ;
\draw  [draw opacity=0][fill={rgb, 255:red, 155; green, 155; blue, 155 }  ,fill opacity=0.72 ] (107.98,109.23) .. controls (107.98,104.04) and (112.43,99.83) .. (117.92,99.83) .. controls (123.41,99.83) and (127.86,104.04) .. (127.86,109.23) .. controls (127.86,114.43) and (123.41,118.64) .. (117.92,118.64) .. controls (112.43,118.64) and (107.98,114.43) .. (107.98,109.23) -- cycle ;
\draw  [draw opacity=0][fill={rgb, 255:red, 155; green, 155; blue, 155 }  ,fill opacity=0.72 ] (158.13,134.92) .. controls (158.13,129.73) and (162.58,125.52) .. (168.07,125.52) .. controls (173.56,125.52) and (178.01,129.73) .. (178.01,134.92) .. controls (178.01,140.12) and (173.56,144.33) .. (168.07,144.33) .. controls (162.58,144.33) and (158.13,140.12) .. (158.13,134.92) -- cycle ;

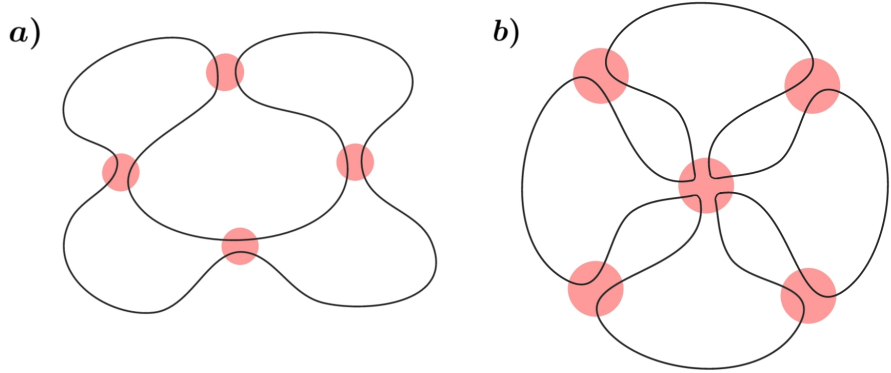
\draw (8,191.33) node [anchor=north west][inner sep=0.75pt]   [align=left] {Figure 2: Examples of circular snakes with many nodes. Shaded disks represent arcs with the tangency order\\higher than the respective surface exponent. };
\draw (53,37.4) node [anchor=north west][inner sep=0.75pt]    {$a)$};
\draw (315,37.4) node [anchor=north west][inner sep=0.75pt]    {$b)$};

\end{tikzpicture}
\end{center}

\begin{Exam}
	 One can construct a circular snake in $\mathbb{R}^3$ possessing any number of nodes (see Figure 4). The nodes are represented by the intersection of the shaded disks with the surface in Figure 4. Also, there are circular snakes in $\mathbb{R}^3$ with a node containing any number of nodal zones (see Figure 4b).  
\end{Exam}

\begin{Prop}\label{Prop: nodes of a circular snakes}
	Every node of a circular snake contains at least two nodal zones. 
\end{Prop}
\begin{proof}
	Let $X$ be a circular $\beta$-snake. Let $\gamma, \gamma' \in V(X)$ be arcs such that $\tord(\gamma,\gamma')>\itord(\gamma,\gamma')$. It is enough to prove that $\gamma$ is a nodal arc if and only if $\gamma'$ is a nodal arc (see Definition \ref{Def: node}). Suppose, by contradiction, $\gamma$ is a nodal arc and $\gamma'$ is a segment arc. 
	
	Let $\{X_k\}$ be a reduced pancake decomposition of $X$. By Remark \ref{Rem: adjacent pancakes} we can assume that $\gamma \in G(X_l)$ and $\gamma'\in G(X_{l+1})$ for some $l$. As $\gamma'$ is a segment arc, $\gamma'$ is a segment arc with respect to $X_l$ by Remark \ref{Rem: relative segments and nodal zones and perfect and open complete zones}. Then, there is $\beta$-H\"older triangle $T'\subset X_{l+1}$ such that $V(T')\subset G(X_{l+1})\cap H_{\beta}(f_l)$ is a constant zone with respect to $X_l$. But this implies that there exists a $\beta$-H\"older triangle $T\subset X_{l}$ such that $V(T)\subset G(X_{l})\cap H_{\beta}(f_{l+1})$ is a constant zone with respect to $X_{l+1}$, a contradiction with $\gamma$ being a nodal arc (see Remark \ref{Rem: relative segments and nodal zones and perfect and open complete zones}).
\end{proof}

The following Lemma is proved in \cite{GabrielovSouza} as Corollary 4.40 for $X$ being a $\beta$-snake. The proof for $X$ being a circular $\beta$-snake is analog and will be omitted.

\begin{Lem}\label{Lem:relative multiplicity pancake and segment}
	Let $X$ be a circular $\beta$-snake, $\{X_{k}\}_{k=1}^{p}$ a pancake decomposition of $X$, and $S\subset V(X)$ a segment.
	If $\gamma,\lambda \in S\cap G(X_{j})$ then $m_{l}(\gamma)=m_{l}(\lambda)$ for all $l$.
\end{Lem}

\begin{Lem}\label{Lem: nodal arcs for boundaries of HT of Def of Abn arc}
	Let $X$ be a circular $\beta$-snake and $\gamma \in V(X)$. There are nodal arcs $\lambda,\lambda' \in V(X)$ such that $T(\lambda,\gamma)$ and $T(\gamma,\lambda')$ are LNE $\beta$-H\"older triangles with $\tord(\lambda,\lambda')>\itord(\lambda,\lambda')$.
\end{Lem}
\begin{proof}
	Since $X$ is a circular $\beta$-snake, $\gamma$ is abnormal and, consequently, there are arcs $\lambda,\lambda' \in V(X)$ such that $T(\lambda,\gamma)$ and $T(\gamma,\lambda')$ are LNE $\beta$-H\"older triangles with $\tord(\lambda,\lambda')>\itord(\lambda,\lambda')$. 
	
	Let $\{X_k\}$ be a reduced pancake decomposition of $X$. By Remark \ref{Rem: adjacent pancakes} we can assume that $\lambda \in G(X_l)$ and $\lambda'\in G(X_{l+1})$ for some $l$.
	
	Suppose that $\lambda$ is a segment arc. Therefore, Lemma \ref{Lem:relative multiplicity pancake and segment} and Remark \ref{Rem: relative segments and nodal zones and perfect and open complete zones} imply that $\lambda$ is a segment arc if and only if $\lambda'$ is a segment arc. 
	
	If both $\lambda$ and $\lambda'$ are segment arcs, there are nodal zones $N$ and $N'$, adjacent to the segments $\Seg_{\lambda}$ and $\Seg_{\lambda'}$, respectively, such that $N, N'\subset V(T(\lambda,\lambda'))$ and $\tord(N,N')>\beta$. Hence, we can replace $\lambda$ and $\lambda'$ by arcs in $N$ and $N'$.	
\end{proof}

\begin{Cor}\label{Cor: nodal arc implies two nodes}
	If $X$ is a circular snake then the following conditions are equivalent:
	\begin{enumerate}
		\item $X$ contains a nodal arc;
		\item $X$ contains a node;
		\item $X$ contains two nodes.
	\end{enumerate}
\end{Cor}

\begin{proof}
     Notice that $(3)$ implies $(2)$ and $(2)$ implies $(1)$ are immediate. Assume that $X$ contains a nodal arc $\gamma$. Let $\mathcal{N}$ be the node containing $\Nod_\gamma$. Since $X$ is a circular snake, $\gamma$ must be abnormal. Hence, there are LNE non-singular H\"older triangles $T$ and $T'$, contained in $X$, such that $T\cap T' = \gamma$ and $T\cup T'$ is not LNE. By Lemma \ref{Lem: nodal arcs for boundaries of HT of Def of Abn arc} we could assume that the boundary arcs of $T$ and $T'$, other than $\gamma$, are both nodal arcs. As $T$ and $T'$ are LNE, those arcs must be in a node different from $\mathcal{N}$. Therefore, $(1)$ implies $(3)$.

\end{proof}

\subsection{Canonical pancake decomposition}

\begin{Teo}\label{Teo: segments are LNE}
	Let $X$ be a circular $\beta$-snake with at least one nodal arc. Then, segments and nodal zones of $X$ are Lipschitz normally embedded. Moreover, given a segment $S$ of $X$, if $N$ and $N'$ are the nodal zones adjacent to $S$, then $N\cup S \cup N'$ is a normally embedded zone (see Definition \ref{Def: NE zone}).
\end{Teo}
\begin{proof}
	Note that item $(1)$ of Proposition \ref{Prop:nodal zones are finite} and Proposition \ref{Prop:CS are weakly NE} imply that nodal zones are LNE. To prove that segments are also LNE, it is enough to prove that there are no arcs $\gamma \in S$ and $\gamma' \in S\cup N$ such that $\tord(\gamma,\gamma')>\itord(\gamma,\gamma')$. 
	
	Suppose, by contradiction, that such arcs exist. Let $\{X_k\}$ be a reduced pancake decomposition of $X$. By Remark \ref{Rem: adjacent pancakes} we can assume that $\gamma \in G(X_l)$ and $\gamma'\in G(X_{l+1})$ for some $l$.
	
	Since $\gamma$ is a segment arc and $\tord(\gamma,\gamma')>\itord(\gamma,\gamma')$, Lemma \ref{Lem:relative multiplicity pancake and segment} implies that $\gamma'$ is a segment arc with respect to $X_k$. Remark \ref{Rem: relative segments and nodal zones and perfect and open complete zones} implies that $\gamma'\in S$.
	
	As $X$ is circular snake, $\gamma$ is abnormal. Then, there are arcs $\lambda,\lambda'\in V(X)$ such that $T_1=T(\lambda,\gamma)$ and $T'_1=T(\gamma,\lambda')$ are LNE $\beta$-H\"older triangles such that $\tord(\lambda,\lambda')>\itord(\lambda,\lambda')$.
	
	As the Vallete link $V(T) \subset S$ of the H\"older triangle $T=T(\gamma,\gamma')$ is not LNE, one of arcs $\lambda$, $\lambda'$ must belong to $G(T)$, say $\lambda \in G(T)$. Since $\{X_k\}$ is a reduced pancake decomposition, we can assume that $\lambda$ and $\lambda'$ are in adjacent pancakes.

	If $\lambda\subset X_{l+1}$ then enlarging once more the pancake $X_{l+1}$ we can assume that $\lambda \in G(X_{l+1})$. Since $\gamma',\lambda \in S\cap G(X_{l+1})$ and $\tord(\gamma,\gamma')>\itord(\gamma,\gamma')$, Lemma \ref{Lem:relative multiplicity pancake and segment} implies that there is an arc $\theta \in G(X_l)$ such that $\tord(\theta, \lambda)>\beta$, a contradiction with $T_1$ being LNE.

	If $\lambda\subset X_{l}$ then $\lambda'\subset X_{l-1}$. We can assume, without loss of generality, that $\lambda\in G(X_l)$ and $\lambda'\in G(X_{l-1})$. Then, as $\gamma,\lambda \in S\cap G(X_l)$ and $\tord(\lambda,\lambda')>\itord(\lambda,\lambda')$, Lemma \ref{Lem:relative multiplicity pancake and segment} implies that there is an arc $\theta' \in G(X_{l-1})$ such that $\tord(\theta', \lambda)>\beta$, a contradiction with $T'_1$ being LNE. 
	
	Finally, let us prove that $N\cup S\cup N'$ is a Lipschitz normally embedded zone. It only remains to be proved that there no arcs $\gamma \in N$ and $\gamma'\in N'$ such that $\tord(\gamma,\gamma')>\itord(\gamma,\gamma')$. 
	
	Consider the H\"older triangle $T = T(\gamma,\gamma')$ such that $S=G(T)$ (see Theorem \ref{Teo:segment is the generic arcs from HT of nodal adjacent zones}). As $X$ is circular snake, $\gamma$ is abnormal. Then, there are arcs $\lambda,\lambda'\in V(X)$ such that $T_1=T(\lambda,\gamma)$ and $T'_1=T(\gamma,\lambda')$ are LNE $\beta$-H\"older triangles such that $\tord(\lambda,\lambda')>\itord(\lambda,\lambda')$. Since $\{X_k\}$ is a reduced pancake decomposition, we can assume that $\lambda \in G(X_l)$ and $\lambda'\in G(X_{l+1})$. 
	
	As $T$ is not LNE, one of arcs $\lambda$, $\lambda'$ must belong to $G(T)$, say $\lambda \in G(T) = S$. Thus, since $\tord(\gamma,\gamma')>\itord(\gamma,\gamma')$, we can assume that $\lambda' \in G(X_{l-1})$ and $\lambda \in G(X_l)$. Then, using an arc $\theta \in  G(T_1)\cap G(X_l)$ instead of $\gamma$ (resp., $\theta' \in G(T'_1)\cap G(X_{l+1})$ instead of $\gamma'$) one obtains the analog contradictions with $T_1$ (resp., $T'_1$) being LNE as proved above, by Lemma \ref{Lem:relative multiplicity pancake and segment}.

\end{proof}

\begin{Cor}\label{Cor: intrinsic pancake decomp of a CS}
	The following decomposition of a circular snake $X$, with at least two nodes, into H\"older triangles determines a pancake decomposition of $X$: the boundary arcs of the H\"older triangles in the decomposition are chosen so that each nodal zone contains exactly one of them. The segments of $X$ are in one-to-one correspondence with the sets of generic arcs of its pancakes.
\end{Cor}
\begin{proof}
	Theorem \ref{Teo: segments are LNE} implies that the chosen H\"older triangles in this decomposition are Lipschitz normally embedded. So they are pancakes.
\end{proof}

\begin{Def}\label{Def: intrinsic decomp}
	 A pancake decomposition of a circular snake $X$ as in Corollary \ref{Cor: intrinsic pancake decomp of a CS} is called an \em intrinsic pancake decomposition of \em $X$. 
\end{Def}

\section{Circular snakes versus snakes}\label{Section: CS vs Snakes}
In this section we study conditions to obtain a snake from a circular snake regarding the canonical decomposition of the Valette link of a circular snake into segments and nodal zones. This matter will be clearer after the next definition.

\begin{Def}\label{Def: HT assoc with an intrinsic decomp}
	 Let $X$ be a circular snake with at least two nodes (which is equivalent to contain a nodal arc, see Corollary \ref{Cor: nodal arc implies two nodes}) and let $\{X_k = T(\gamma_{k-1},\gamma_k)\}_{k=1}^p$ be an intrinsic pancake decomposition of $X$. We denote the segments and nodal zones of $X$ as $S_1, \ldots, S_p$ and $N_0,\ldots, N_p$, respectively, where $S_i=G(X_i)$ and $N_j = \Nod_{\gamma_j}$ (notice that $N_0 = N_p$). For each $k$ we define two H\"older triangles $X(k)$ and $X_{\alpha,k}$ associated with $(X, \{X_k\})$ as follows: 
        $$
        X(k) = X_1\cup \cdots \cup\hat{X}_k \cup\cdots \cup X_p = \bigcup_{j=1, j\ne k}^pX_j,
        $$
	
	and $X_{\alpha, k} = X\setminus \Int(T)$, where $T$ is an $\alpha$-H\"older triangle, with $\alpha > \beta$, such that $\gamma_k\in G(T)$. Roughly speaking, $X(k)$ is obtaining from $X$ by ``removing the segment $S_k$'', while $X_{\alpha,k}$ is obtained from $X$ by removing the interior of a H\"older triangle which Vallete link is contained in $N_k$.  
\end{Def}

\begin{remark}\label{Rem: convertion for consecutive pancakes}
	Let $X$ and $\{X_k\}$ be as in Definition \ref{Def: HT assoc with an intrinsic decomp}. We assume that $S_{j-1}=S_p$ if $j=1$ and $S_{j+1}=S_1$ if $j=p$. We shall also use the same convention for the nodal zones of $X$. 
\end{remark}

A natural question one can formulate is: what conditions can we impose over a circular $\beta$-snake $X$ to guarantee that $X(k)$ (respectively, $X_{\alpha,k}$) is a $\beta$-snake (see Definition \ref{Def:snake})? Both these questions will be addressed below.

\begin{center}

\tikzset{every picture/.style={line width=0.75pt}} 

\begin{tikzpicture}[x=0.75pt,y=0.75pt,yscale=-1,xscale=1]

\draw    (151.24,26.06) .. controls (129.05,26.82) and (110.09,37.25) .. (109.31,66.35) .. controls (108.53,95.45) and (137.52,106.29) .. (163.59,117.66) .. controls (189.66,129.02) and (163.59,156.98) .. (175.67,163.73) .. controls (187.75,170.49) and (199.83,174.49) .. (232.57,168.03) ;
\draw    (151.24,26.06) .. controls (178.41,26.81) and (205.96,38.81) .. (205.96,61.23) .. controls (205.96,83.66) and (168.46,85.8) .. (159.24,99.01) .. controls (150.03,112.22) and (175.14,124.51) .. (207.88,125.74) ;
\draw  [draw opacity=0][fill={rgb, 255:red, 155; green, 155; blue, 155 }  ,fill opacity=0.47 ] (153.6,116.96) .. controls (153.6,110.2) and (159.3,104.73) .. (166.33,104.73) .. controls (173.36,104.73) and (179.06,110.2) .. (179.06,116.96) .. controls (179.06,123.71) and (173.36,129.19) .. (166.33,129.19) .. controls (159.3,129.19) and (153.6,123.71) .. (153.6,116.96) -- cycle ;
\draw    (265.72,26.06) .. controls (287.91,26.82) and (306.87,37.25) .. (307.64,66.35) .. controls (308.42,95.45) and (283.43,104.14) .. (257.36,115.51) .. controls (231.3,126.87) and (253.87,162.51) .. (232.57,168.03) ;
\draw    (265.72,26.06) .. controls (238.54,26.81) and (209.84,38.12) .. (209.84,60.55) .. controls (209.84,82.97) and (248.49,85.8) .. (257.71,99.01) .. controls (266.93,112.22) and (240.62,124.51) .. (207.88,125.74) ;
\draw  [draw opacity=0][fill={rgb, 255:red, 155; green, 155; blue, 155 }  ,fill opacity=0.47 ] (194.93,60.41) .. controls (194.93,53.65) and (200.62,48.18) .. (207.65,48.18) .. controls (214.68,48.18) and (220.38,53.65) .. (220.38,60.41) .. controls (220.38,67.16) and (214.68,72.64) .. (207.65,72.64) .. controls (200.62,72.64) and (194.93,67.16) .. (194.93,60.41) -- cycle ;
\draw  [draw opacity=0][fill={rgb, 255:red, 155; green, 155; blue, 155 }  ,fill opacity=0.47 ] (241.02,116.65) .. controls (241.02,109.9) and (246.72,104.42) .. (253.75,104.42) .. controls (260.78,104.42) and (266.47,109.9) .. (266.47,116.65) .. controls (266.47,123.41) and (260.78,128.88) .. (253.75,128.88) .. controls (246.72,128.88) and (241.02,123.41) .. (241.02,116.65) -- cycle ;
\draw    (436.02,158.79) .. controls (466.64,159.13) and (441.42,134.06) .. (437.46,123.9) .. controls (433.5,113.74) and (443.58,94.44) .. (451.15,94.1) .. controls (458.71,93.76) and (468.44,113.74) .. (465.92,123.57) .. controls (463.39,133.39) and (437.46,159.13) .. (468.8,158.79) ;
\draw    (396.84,45.07) .. controls (380.71,69.54) and (416.56,62.19) .. (427.86,64.24) .. controls (439.16,66.3) and (451.53,84.39) .. (447.93,90.66) .. controls (444.33,96.92) and (421.11,94.43) .. (413.47,87.33) .. controls (405.83,80.23) and (395.78,46.05) .. (379.9,71.46) ;
\draw    (519.98,71.9) .. controls (505.49,46.54) and (494.35,79.41) .. (486.81,87.58) .. controls (479.27,95.76) and (456.42,96.79) .. (452.45,90.73) .. controls (448.47,84.67) and (462.37,67) .. (472.73,64.33) .. controls (483.1,61.65) and (519.6,70.55) .. (504.13,44.92) ;
\draw    (396.84,45.07) .. controls (409.36,29.42) and (492.93,28.74) .. (504.13,44.92) ;
\draw    (436.02,158.79) .. controls (415.37,156.22) and (370.41,88.58) .. (379.9,71.46) ;
\draw    (519.98,71.9) .. controls (528.09,89.94) and (489.31,157.8) .. (468.8,158.79) ;
\draw  [draw opacity=0][fill={rgb, 255:red, 155; green, 155; blue, 155 }  ,fill opacity=0.47 ] (380.37,62.51) .. controls (380.37,55.07) and (386.83,49.03) .. (394.8,49.03) .. controls (402.76,49.03) and (409.22,55.07) .. (409.22,62.51) .. controls (409.22,69.96) and (402.76,76) .. (394.8,76) .. controls (386.83,76) and (380.37,69.96) .. (380.37,62.51) -- cycle ;
\draw  [draw opacity=0][fill={rgb, 255:red, 155; green, 155; blue, 155 }  ,fill opacity=0.47 ] (435.74,91.95) .. controls (435.74,84.5) and (442.2,78.46) .. (450.16,78.46) .. controls (458.13,78.46) and (464.59,84.5) .. (464.59,91.95) .. controls (464.59,99.39) and (458.13,105.43) .. (450.16,105.43) .. controls (442.2,105.43) and (435.74,99.39) .. (435.74,91.95) -- cycle ;
\draw  [draw opacity=0][fill={rgb, 255:red, 155; green, 155; blue, 155 }  ,fill opacity=0.47 ] (491.22,61.66) .. controls (491.22,54.21) and (497.67,48.17) .. (505.64,48.17) .. controls (513.6,48.17) and (520.06,54.21) .. (520.06,61.66) .. controls (520.06,69.1) and (513.6,75.14) .. (505.64,75.14) .. controls (497.67,75.14) and (491.22,69.1) .. (491.22,61.66) -- cycle ;
\draw  [draw opacity=0][fill={rgb, 255:red, 155; green, 155; blue, 155 }  ,fill opacity=0.47 ] (437.51,152.26) .. controls (437.51,144.81) and (443.97,138.78) .. (451.93,138.78) .. controls (459.9,138.78) and (466.36,144.81) .. (466.36,152.26) .. controls (466.36,159.71) and (459.9,165.75) .. (451.93,165.75) .. controls (443.97,165.75) and (437.51,159.71) .. (437.51,152.26) -- cycle ;
\draw  [fill={rgb, 255:red, 0; green, 0; blue, 0 }  ,fill opacity=1 ] (108.16,65.28) .. controls (108.16,64.68) and (108.68,64.2) .. (109.31,64.2) .. controls (109.95,64.2) and (110.46,64.68) .. (110.46,65.28) .. controls (110.46,65.87) and (109.95,66.35) .. (109.31,66.35) .. controls (108.68,66.35) and (108.16,65.87) .. (108.16,65.28) -- cycle ;
\draw  [fill={rgb, 255:red, 0; green, 0; blue, 0 }  ,fill opacity=1 ] (164.5,118.78) .. controls (164.5,118.19) and (165.01,117.71) .. (165.65,117.71) .. controls (166.28,117.71) and (166.8,118.19) .. (166.8,118.78) .. controls (166.8,119.38) and (166.28,119.86) .. (165.65,119.86) .. controls (165.01,119.86) and (164.5,119.38) .. (164.5,118.78) -- cycle ;
\draw  [fill={rgb, 255:red, 0; green, 0; blue, 0 }  ,fill opacity=1 ] (208.69,60.55) .. controls (208.69,59.95) and (209.21,59.47) .. (209.84,59.47) .. controls (210.48,59.47) and (210.99,59.95) .. (210.99,60.55) .. controls (210.99,61.15) and (210.48,61.63) .. (209.84,61.63) .. controls (209.21,61.63) and (208.69,61.15) .. (208.69,60.55) -- cycle ;
\draw  [fill={rgb, 255:red, 0; green, 0; blue, 0 }  ,fill opacity=1 ] (392.38,57.5) .. controls (392.38,56.84) and (392.96,56.31) .. (393.68,56.31) .. controls (394.4,56.31) and (394.98,56.84) .. (394.98,57.5) .. controls (394.98,58.16) and (394.4,58.69) .. (393.68,58.69) .. controls (392.96,58.69) and (392.38,58.16) .. (392.38,57.5) -- cycle ;
\draw  [fill={rgb, 255:red, 0; green, 0; blue, 0 }  ,fill opacity=1 ] (417.01,90.69) .. controls (417.01,90.03) and (417.59,89.5) .. (418.31,89.5) .. controls (419.03,89.5) and (419.61,90.03) .. (419.61,90.69) .. controls (419.61,91.35) and (419.03,91.88) .. (418.31,91.88) .. controls (417.59,91.88) and (417.01,91.35) .. (417.01,90.69) -- cycle ;
\draw  [fill={rgb, 255:red, 0; green, 0; blue, 0 }  ,fill opacity=1 ] (446.63,89.47) .. controls (446.63,88.81) and (447.21,88.28) .. (447.93,88.28) .. controls (448.65,88.28) and (449.24,88.81) .. (449.24,89.47) .. controls (449.24,90.13) and (448.65,90.66) .. (447.93,90.66) .. controls (447.21,90.66) and (446.63,90.13) .. (446.63,89.47) -- cycle ;
\draw  [fill={rgb, 255:red, 0; green, 0; blue, 0 }  ,fill opacity=1 ] (451.14,89.54) .. controls (451.14,88.89) and (451.72,88.35) .. (452.44,88.35) .. controls (453.16,88.35) and (453.75,88.89) .. (453.75,89.54) .. controls (453.75,90.2) and (453.16,90.73) .. (452.44,90.73) .. controls (451.72,90.73) and (451.14,90.2) .. (451.14,89.54) -- cycle ;
\draw  [fill={rgb, 255:red, 0; green, 0; blue, 0 }  ,fill opacity=1 ] (506.12,56.31) .. controls (506.12,55.66) and (506.7,55.13) .. (507.42,55.13) .. controls (508.14,55.13) and (508.72,55.66) .. (508.72,56.31) .. controls (508.72,56.97) and (508.14,57.5) .. (507.42,57.5) .. controls (506.7,57.5) and (506.12,56.97) .. (506.12,56.31) -- cycle ;
\draw  [fill={rgb, 255:red, 0; green, 0; blue, 0 }  ,fill opacity=1 ] (485.51,87.58) .. controls (485.51,86.93) and (486.09,86.39) .. (486.81,86.39) .. controls (487.53,86.39) and (488.11,86.93) .. (488.11,87.58) .. controls (488.11,88.24) and (487.53,88.77) .. (486.81,88.77) .. controls (486.09,88.77) and (485.51,88.24) .. (485.51,87.58) -- cycle ;

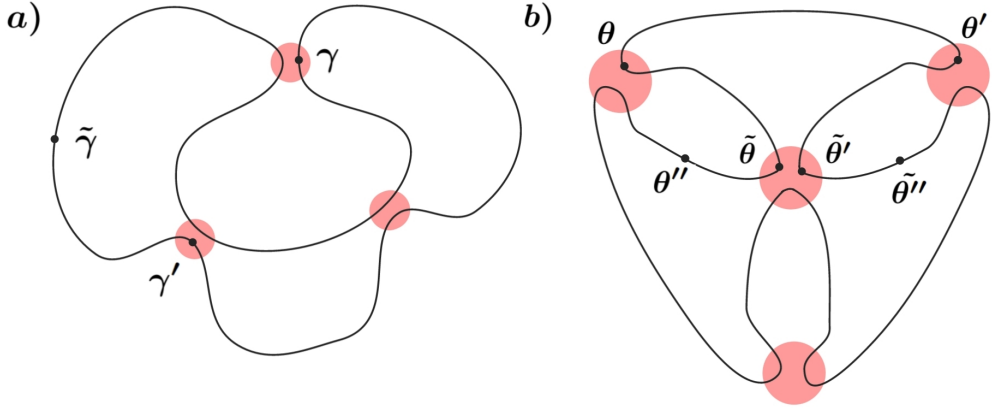
\draw (6,179.83) node [anchor=north west][inner sep=0.75pt]   [align=left] {Figure 5: Examples of segments and nodal zones of circular snakes. Shaded disks represent arcs with the tangency\\order higher than each surface exponent.};
\draw (214.93,49.42) node [anchor=north west][inner sep=0.75pt]    {$\gamma $};
\draw (149.99,120.07) node [anchor=north west][inner sep=0.75pt]    {$\gamma '$};
\draw (95.09,51.1) node [anchor=north west][inner sep=0.75pt]    {$\tilde{\gamma }$};
\draw (508.14,41.43) node [anchor=north west][inner sep=0.75pt]    {$\theta '$};
\draw (381.5,39.59) node [anchor=north west][inner sep=0.75pt]    {$\theta $};
\draw (429.82,68.88) node [anchor=north west][inner sep=0.75pt]    {$\tilde{\theta }$};
\draw (456.65,69.8) node [anchor=north west][inner sep=0.75pt]    {$\tilde{\theta } '$};
\draw (404.1,91.19) node [anchor=north west][inner sep=0.75pt]    {$\theta ''$};
\draw (487.39,83.32) node [anchor=north west][inner sep=0.75pt]    {$\tilde{\theta } ''$};
\draw (87.16,20.69) node [anchor=north west][inner sep=0.75pt]    {$a)$};
\draw (354.97,19.93) node [anchor=north west][inner sep=0.75pt]    {$b)$};

\end{tikzpicture}
\end{center}

Notice that in Figure 5b the segment $S=G(T(\theta,\theta'))$ can be removed as in Definition \ref{Def: HT assoc with an intrinsic decomp}, since $X\setminus \Int(T(\theta,\theta'))$ is a snake, but $\tilde{S}=G(T(\theta,\tilde{\theta}))$ and $\tilde{S}'=G(T(\theta',\tilde{\theta}'))$ cannot be removed, since $X\setminus \Int(T(\theta,\tilde{\theta}))$ and $X\setminus \Int(T(\theta',\tilde{\theta}'))$ are not snakes (the generic arcs $\theta''$ and $\tilde{\theta}''$ of $X\setminus \Int(T(\theta,\tilde{\theta}))$ and $X\setminus \Int(T(\theta',\tilde{\theta}'))$, respectively,  would not be abnormal).

Also notice that if $X$ is a circular $\beta$-snake with link as in Figure 5a then any $\alpha$-H\"older triangle $T$, with $\alpha>\beta$, such that $\gamma \in G(T)$, satisfies that $X\setminus \Int(T)$ is a $\beta$-snake. However, if $T$ is an $\alpha$-H\"older triangle, with $\alpha>\beta$, such that $\gamma' \in G(T)$, then $X\setminus \Int(T)$ is not a $\beta$-snake, since $\tilde{\gamma}$ is not going to be an abnormal arc anymore.

\subsection{Removing segments}

\begin{Teo}\label{Prop: removing segments}
	Let $X$ and $\{X_k\}$ be as in Definition \ref{Def: HT assoc with an intrinsic decomp}. Then $X(k)$ is a snake if, and only if, $N_{k-2}\cap\Nor(X(k))=\emptyset=N_{k+1}\cap\Nor(X(k))$. In words, $X(k)$ is a snake if and only if the nodal zones $N_{k-2}$ and $N_{k+1}$ of $X$ do not contain normal arcs of $X(k)$.
\end{Teo}

\begin{proof}
	Since $N_{k-2}\cap\Nor(X(k))=\emptyset$, given an arc $\gamma \in N_{k-2}$ there are arcs $\gamma_{-} \subset T(\gamma_k,\gamma)\subset X(k)$ and $\gamma_{+}\subset T(\gamma,\gamma_{k-1})\subset X(k)$ such that the H\"older triangles $T(\gamma_{-},\gamma)$ and $T(\gamma,\gamma_{+})$, contained in $X(k)$, are LNE and $\tord(\gamma_{-},\gamma_{+})>\itord(\gamma_{-},\gamma_{+})$. By Lemma \ref{Lem: nodal arcs for boundaries of HT of Def of Abn arc}, we can assume that $\gamma_{-}$ and $\gamma_{+}$ are nodal arcs. In particular, as $T(\gamma,\gamma_{+})$ is LNE, Theorem \ref{Teo: segments are LNE} implies that $\gamma_{+}\in N_{k-1}\cap V(X(k))$.
	
	Similarly, as $N_{k+1}\cap\Nor(X(k))=\emptyset$, given an arc $\lambda\in N_{k+1}$ there are nodal arcs $\lambda_{-}\subset T(\gamma_k,\lambda)\subset X(k)$ and $\lambda_{+}\subset T(\lambda,\gamma_{k-1})\subset X(k)$ such that the H\"older triangles $T(\lambda_{-},\lambda)$ and $T(\lambda,\lambda_{+})$, contained in $X(k)$, are LNE and $\tord(\lambda_{-},\lambda_{+})>\itord(\lambda_{-},\lambda_{+})$. Moreover, since $T(\lambda_{-},\lambda)$ is LNE, Theorem \ref{Teo: segments are LNE} implies that $\lambda_{-} \in N_{k}\cap V(X(k))$.
	
	We claim that any generic arc of $X(k)$ is abnormal. Consider $\theta \in G(X(k))$. As $\theta \subset X$ and $X$ is a circular snake, $\theta$ is abnormal in $X$. Thus, by Lemma \ref{Lem: nodal arcs for boundaries of HT of Def of Abn arc}, there are arcs $\theta_{-},\theta_{+}\in V(X)$ such that $T(\theta_{-},\theta)$ and $T(\theta,\theta_{+})$ are LNE and $\tord(\theta_{-},\theta_{+})>\itord(\theta_{-},\theta_{+})$.
	
	If $\theta \in G(T(\gamma_{-},\gamma_{+}))\cup G(T(\lambda_{-},\lambda_{+}))$, we have $\theta$ trivially abnormal. If $\tord(\theta,\gamma_{-})>\mu(X)$ (the case $\tord(\theta,\lambda_{+})>\mu(X)$ is analog) then, since $\gamma_{-}$ is an abnormal arc of $X$, there are LNE $\mu(X)$-H\"older triangles $T_{-}$ and $T_{+}$, where $T_{-}\cap T_{+} = \gamma_{-}$, such that $T_{-}\cup T^{+}$ is not LNE. Then, as $\tord(\gamma_{-},\gamma_{+})>\itord(\gamma_{-},\gamma_{+})$ and $\tord(\lambda_{-},\lambda_{+})>\itord(\lambda_{-},\lambda_{+})$, we have $T_{-}\cup T^{+} \subset X(k)$. Hence, $\gamma_{-}$ is an abnormal arc of $X(k)$ and, since $\tord(\theta,\gamma_{-})>\mu(X)$, it follows that $\theta$ is also an abnormal arc of $X(k)$. 
	
	Finally, suppose that $\theta \in G(T(\gamma_{-},\lambda_{+}))$, where $T(\gamma_{-},\lambda_{+}) \subset X(k)$. Since $X$ is a circular snake, $\theta$ is an abnormal arc of $X$. So, there are LNE $\mu(X)$-H\"older triangles $T_{-}$ and $T_{+}$, where $T_{-}\cap T_{+} = \theta$, such that $T_{-}\cup T^{+}$ is not LNE. As $\tord(\gamma_{-},\gamma_{+})>\itord(\gamma_{-},\gamma_{+})$ and $\tord(\lambda_{-},\lambda_{+})>\itord(\lambda_{-},\lambda_{+})$, we have $T_{-}\cup T^{+} \subset X(k)$. This finishes the proof.
\end{proof}

\begin{remark}\label{Rem: removing segments}
	 Besides the criterion given in Theorem \ref{Prop: removing segments}, there are circular snakes such that $X(k)$ is not a snake for every $k$. For example, let $X$ be a circular snake with link as in Figure 6a and let $\{X_k\}$ be an intrinsic pancake decomposition of $X$ such that $S_k = G(X_k) = \Seg_{\theta_k}$. Then, $X(k)$ is not a snake for any $k\in \{1,2,3,4\}$, since for $j\in \{2,3\}\setminus \{k\}$, $\theta_j$ is a generic normal arc of $X(k)$. Furthermore, for any $k\in \{5,6,7,8\}$, $X(k)$ is not a snake as well, since for $j\in \{6,7\}\setminus \{k\}$, $\theta_j$ is a generic normal arc of $X(k)$.
\end{remark}

\subsection{Removing H\"older triangles inside nodal zones}

\begin{Prop}\label{Prop: removing HT inside nodal zones}
Let $X$ and $\{X_k\}$ be as in Definition \ref{Def: HT assoc with an intrinsic decomp}. Then, $X_{\alpha, k}$ is a $\beta$-snake if, and only if, both $S_k$ and $S_{k+1}$ are subsets of $\Abn(X_{\alpha, k})$.

\end{Prop}
\begin{proof}
	If $X_{\alpha, k}$ is a snake then $\Abn(X_{\alpha, k})=G(X_{\alpha, k})$ and $S_k, S_{k+1} \subset G(X_{\alpha, k})$. Therefore, it follows that $S_k$ and $S_{k+1}$ are subsets of $\Abn(X_{\alpha, k})$.
	
	Suppose that $S_k$ and $S_{k+1}$ are subsets of $\Abn(X_{\alpha, k})$. We proceed as in the proof of Theorem \ref{Prop: removing segments} to conclude that any generic arc of $X_{\alpha, k}$ is abnormal.
\end{proof}

\begin{remark}\label{Rem: removing piece of a nodal zone}
	 There are circular snakes $X$ such that $X_{\alpha, k}$ as in Proposition \ref{Prop: removing HT inside nodal zones} is never a snake. For example, let $X$ be a circular snake with link as in Figure 6b with $\{X_k = T(\lambda_{k-1},\lambda_k)\}$ being an intrinsic pancake decomposition of $X$. Note that $\lambda_2$ is not an abnormal arc of $X_{\alpha,1}$ and $\lambda_1$ is not an abnormal arc of $X_{\alpha,2}$. Moreover, $\lambda_3$, $\lambda_8$, $\lambda_4$ and $\lambda_7$ are not abnormal arcs of $X_{\alpha,4}$, $X_{\alpha,7}$, $X_{\alpha,3}$ and $X_{\alpha,8}$, respectively. Finally, any arc in $G(X_6 = T(\lambda_5,\lambda_6))$ is not an abnormal arc of $X_{\alpha,5}$ or $X_{\alpha,5}$.
\end{remark}

\begin{center}

\tikzset{every picture/.style={line width=0.75pt}} 

\begin{tikzpicture}[x=0.75pt,y=0.75pt,yscale=-1,xscale=1]

\draw    (180.32,179.2) .. controls (140.32,178.2) and (117.92,202.4) .. (90.32,189.2) .. controls (62.72,176) and (66.72,153.6) .. (84.32,142.8) .. controls (101.92,132) and (120.72,141.6) .. (135.52,127.2) .. controls (150.32,112.8) and (96.72,88) .. (111.52,68.8) .. controls (126.32,49.6) and (178.29,69.42) .. (178.84,81.05) .. controls (179.38,92.69) and (172,92.3) .. (147,118.3) .. controls (122,144.3) and (143.12,175.2) .. (181.12,175.2) ;
\draw    (180.32,179.2) .. controls (220.32,178.2) and (244.32,202.4) .. (271.92,189.2) .. controls (299.52,176) and (295.52,153.6) .. (277.92,142.8) .. controls (260.32,132) and (241.52,141.6) .. (226.72,127.2) .. controls (211.92,112.8) and (265.52,88) .. (250.72,68.8) .. controls (235.92,49.6) and (181.93,69.05) .. (182.11,81.24) .. controls (182.29,93.42) and (190.24,92.3) .. (215.24,118.3) .. controls (240.24,144.3) and (219.12,175.2) .. (181.12,175.2) ;
\draw  [draw opacity=0][fill={rgb, 255:red, 155; green, 155; blue, 155 }  ,fill opacity=0.47 ] (124.83,123.98) .. controls (124.83,115.71) and (131.57,109) .. (139.89,109) .. controls (148.21,109) and (154.96,115.71) .. (154.96,123.98) .. controls (154.96,132.25) and (148.21,138.96) .. (139.89,138.96) .. controls (131.57,138.96) and (124.83,132.25) .. (124.83,123.98) -- cycle ;
\draw  [draw opacity=0][fill={rgb, 255:red, 155; green, 155; blue, 155 }  ,fill opacity=0.47 ] (166.54,177.12) .. controls (166.54,168.85) and (173.29,162.14) .. (181.61,162.14) .. controls (189.93,162.14) and (196.67,168.85) .. (196.67,177.12) .. controls (196.67,185.4) and (189.93,192.1) .. (181.61,192.1) .. controls (173.29,192.1) and (166.54,185.4) .. (166.54,177.12) -- cycle ;
\draw  [draw opacity=0][fill={rgb, 255:red, 155; green, 155; blue, 155 }  ,fill opacity=0.47 ] (207.4,125.69) .. controls (207.4,117.42) and (214.14,110.71) .. (222.47,110.71) .. controls (230.79,110.71) and (237.53,117.42) .. (237.53,125.69) .. controls (237.53,133.97) and (230.79,140.67) .. (222.47,140.67) .. controls (214.14,140.67) and (207.4,133.97) .. (207.4,125.69) -- cycle ;
\draw  [draw opacity=0][fill={rgb, 255:red, 155; green, 155; blue, 155 }  ,fill opacity=0.47 ] (165.4,83.69) .. controls (165.4,75.42) and (172.14,68.71) .. (180.47,68.71) .. controls (188.79,68.71) and (195.53,75.42) .. (195.53,83.69) .. controls (195.53,91.97) and (188.79,98.67) .. (180.47,98.67) .. controls (172.14,98.67) and (165.4,91.97) .. (165.4,83.69) -- cycle ;
\draw  [fill={rgb, 255:red, 0; green, 0; blue, 0 }  ,fill opacity=1 ] (159.39,104.91) .. controls (159.39,104.06) and (160.12,103.36) .. (161.03,103.36) .. controls (161.93,103.36) and (162.67,104.06) .. (162.67,104.91) .. controls (162.67,105.77) and (161.93,106.46) .. (161.03,106.46) .. controls (160.12,106.46) and (159.39,105.77) .. (159.39,104.91) -- cycle ;
\draw  [fill={rgb, 255:red, 0; green, 0; blue, 0 }  ,fill opacity=1 ] (216.72,160.58) .. controls (216.72,159.72) and (217.45,159.03) .. (218.36,159.03) .. controls (219.27,159.03) and (220,159.72) .. (220,160.58) .. controls (220,161.43) and (219.27,162.13) .. (218.36,162.13) .. controls (217.45,162.13) and (216.72,161.43) .. (216.72,160.58) -- cycle ;
\draw  [fill={rgb, 255:red, 0; green, 0; blue, 0 }  ,fill opacity=1 ] (141.39,160.25) .. controls (141.39,159.39) and (142.12,158.7) .. (143.03,158.7) .. controls (143.93,158.7) and (144.67,159.39) .. (144.67,160.25) .. controls (144.67,161.1) and (143.93,161.79) .. (143.03,161.79) .. controls (142.12,161.79) and (141.39,161.1) .. (141.39,160.25) -- cycle ;
\draw  [fill={rgb, 255:red, 0; green, 0; blue, 0 }  ,fill opacity=1 ] (200.39,106.58) .. controls (200.39,105.72) and (201.12,105.03) .. (202.03,105.03) .. controls (202.93,105.03) and (203.67,105.72) .. (203.67,106.58) .. controls (203.67,107.43) and (202.93,108.13) .. (202.03,108.13) .. controls (201.12,108.13) and (200.39,107.43) .. (200.39,106.58) -- cycle ;
\draw  [fill={rgb, 255:red, 0; green, 0; blue, 0 }  ,fill opacity=1 ] (245.05,64.91) .. controls (245.05,64.06) and (245.79,63.36) .. (246.69,63.36) .. controls (247.6,63.36) and (248.33,64.06) .. (248.33,64.91) .. controls (248.33,65.77) and (247.6,66.46) .. (246.69,66.46) .. controls (245.79,66.46) and (245.05,65.77) .. (245.05,64.91) -- cycle ;
\draw  [fill={rgb, 255:red, 0; green, 0; blue, 0 }  ,fill opacity=1 ] (285.72,176.58) .. controls (285.72,175.72) and (286.45,175.03) .. (287.36,175.03) .. controls (288.27,175.03) and (289,175.72) .. (289,176.58) .. controls (289,177.43) and (288.27,178.13) .. (287.36,178.13) .. controls (286.45,178.13) and (285.72,177.43) .. (285.72,176.58) -- cycle ;
\draw  [fill={rgb, 255:red, 0; green, 0; blue, 0 }  ,fill opacity=1 ] (75.39,179.58) .. controls (75.39,178.72) and (76.12,178.03) .. (77.03,178.03) .. controls (77.93,178.03) and (78.67,178.72) .. (78.67,179.58) .. controls (78.67,180.43) and (77.93,181.13) .. (77.03,181.13) .. controls (76.12,181.13) and (75.39,180.43) .. (75.39,179.58) -- cycle ;
\draw  [fill={rgb, 255:red, 0; green, 0; blue, 0 }  ,fill opacity=1 ] (113.19,65.47) .. controls (113.19,64.61) and (113.92,63.92) .. (114.83,63.92) .. controls (115.73,63.92) and (116.47,64.61) .. (116.47,65.47) .. controls (116.47,66.32) and (115.73,67.01) .. (114.83,67.01) .. controls (113.92,67.01) and (113.19,66.32) .. (113.19,65.47) -- cycle ;
\draw    (514.4,45.07) .. controls (474.4,45.07) and (449.73,55.07) .. (463.07,69.73) .. controls (476.4,84.4) and (512.4,73.73) .. (512.4,103.73) ;
\draw    (514.4,45.07) .. controls (554.4,45.07) and (579.07,55.07) .. (565.73,69.73) .. controls (552.4,84.4) and (516.88,73.28) .. (516.88,103.28) ;
\draw    (512.4,103.73) -- (512.88,178.08) ;
\draw    (516.88,103.28) -- (517.36,177.63) ;
\draw    (408.48,75.68) .. controls (420.88,60.88) and (446.88,60.48) .. (466.48,75.68) .. controls (486.08,90.88) and (444.88,150.48) .. (453.68,168.48) .. controls (462.48,186.48) and (512.8,194.06) .. (512.88,178.08) ;
\draw    (621.76,75.23) .. controls (609.36,60.43) and (583.36,60.03) .. (563.76,75.23) .. controls (544.16,90.43) and (585.36,150.03) .. (576.56,168.03) .. controls (567.76,186.03) and (517.44,193.61) .. (517.36,177.63) ;
\draw    (408.48,75.68) .. controls (397.68,90.08) and (390.08,164.08) .. (408.08,187.28) .. controls (426.08,210.48) and (482.48,211.68) .. (517.68,211.28) .. controls (552.88,210.88) and (604.48,208.88) .. (618.88,193.68) .. controls (633.28,178.48) and (639.36,100.03) .. (621.76,75.23) ;
\draw  [draw opacity=0][fill={rgb, 255:red, 155; green, 155; blue, 155 }  ,fill opacity=0.47 ] (448.88,73.12) .. controls (448.88,64.85) and (455.63,58.14) .. (463.95,58.14) .. controls (472.27,58.14) and (479.01,64.85) .. (479.01,73.12) .. controls (479.01,81.4) and (472.27,88.1) .. (463.95,88.1) .. controls (455.63,88.1) and (448.88,81.4) .. (448.88,73.12) -- cycle ;
\draw  [draw opacity=0][fill={rgb, 255:red, 155; green, 155; blue, 155 }  ,fill opacity=0.47 ] (499.78,183.02) .. controls (499.78,174.75) and (506.53,168.04) .. (514.85,168.04) .. controls (523.17,168.04) and (529.91,174.75) .. (529.91,183.02) .. controls (529.91,191.3) and (523.17,198) .. (514.85,198) .. controls (506.53,198) and (499.78,191.3) .. (499.78,183.02) -- cycle ;
\draw  [draw opacity=0][fill={rgb, 255:red, 155; green, 155; blue, 155 }  ,fill opacity=0.47 ] (498.48,90.72) .. controls (498.48,82.45) and (505.23,75.74) .. (513.55,75.74) .. controls (521.87,75.74) and (528.61,82.45) .. (528.61,90.72) .. controls (528.61,99) and (521.87,105.7) .. (513.55,105.7) .. controls (505.23,105.7) and (498.48,99) .. (498.48,90.72) -- cycle ;
\draw  [draw opacity=0][fill={rgb, 255:red, 155; green, 155; blue, 155 }  ,fill opacity=0.47 ] (550.48,73.12) .. controls (550.48,64.85) and (557.23,58.14) .. (565.55,58.14) .. controls (573.87,58.14) and (580.61,64.85) .. (580.61,73.12) .. controls (580.61,81.4) and (573.87,88.1) .. (565.55,88.1) .. controls (557.23,88.1) and (550.48,81.4) .. (550.48,73.12) -- cycle ;
\draw  [fill={rgb, 255:red, 0; green, 0; blue, 0 }  ,fill opacity=1 ] (457.92,71.62) .. controls (457.92,70.77) and (458.65,70.07) .. (459.56,70.07) .. controls (460.46,70.07) and (461.2,70.77) .. (461.2,71.62) .. controls (461.2,72.48) and (460.46,73.17) .. (459.56,73.17) .. controls (458.65,73.17) and (457.92,72.48) .. (457.92,71.62) -- cycle ;
\draw  [fill={rgb, 255:red, 0; green, 0; blue, 0 }  ,fill opacity=1 ] (464.2,71.3) .. controls (464.2,70.45) and (464.93,69.75) .. (465.84,69.75) .. controls (466.75,69.75) and (467.48,70.45) .. (467.48,71.3) .. controls (467.48,72.16) and (466.75,72.85) .. (465.84,72.85) .. controls (464.93,72.85) and (464.2,72.16) .. (464.2,71.3) -- cycle ;
\draw  [fill={rgb, 255:red, 0; green, 0; blue, 0 }  ,fill opacity=1 ] (507.95,92.05) .. controls (507.95,91.2) and (508.68,90.5) .. (509.59,90.5) .. controls (510.5,90.5) and (511.23,91.2) .. (511.23,92.05) .. controls (511.23,92.91) and (510.5,93.6) .. (509.59,93.6) .. controls (508.68,93.6) and (507.95,92.91) .. (507.95,92.05) -- cycle ;
\draw  [fill={rgb, 255:red, 0; green, 0; blue, 0 }  ,fill opacity=1 ] (517.7,92.05) .. controls (517.7,91.2) and (518.43,90.5) .. (519.34,90.5) .. controls (520.25,90.5) and (520.98,91.2) .. (520.98,92.05) .. controls (520.98,92.91) and (520.25,93.6) .. (519.34,93.6) .. controls (518.43,93.6) and (517.7,92.91) .. (517.7,92.05) -- cycle ;
\draw  [fill={rgb, 255:red, 0; green, 0; blue, 0 }  ,fill opacity=1 ] (568.2,72.05) .. controls (568.2,71.2) and (568.93,70.5) .. (569.84,70.5) .. controls (570.75,70.5) and (571.48,71.2) .. (571.48,72.05) .. controls (571.48,72.91) and (570.75,73.6) .. (569.84,73.6) .. controls (568.93,73.6) and (568.2,72.91) .. (568.2,72.05) -- cycle ;
\draw  [fill={rgb, 255:red, 0; green, 0; blue, 0 }  ,fill opacity=1 ] (560.45,72.3) .. controls (560.45,71.45) and (561.18,70.75) .. (562.09,70.75) .. controls (563,70.75) and (563.73,71.45) .. (563.73,72.3) .. controls (563.73,73.16) and (563,73.85) .. (562.09,73.85) .. controls (561.18,73.85) and (560.45,73.16) .. (560.45,72.3) -- cycle ;
\draw  [fill={rgb, 255:red, 0; green, 0; blue, 0 }  ,fill opacity=1 ] (508.2,183.08) .. controls (508.2,182.22) and (508.93,181.53) .. (509.84,181.53) .. controls (510.75,181.53) and (511.48,182.22) .. (511.48,183.08) .. controls (511.48,183.93) and (510.75,184.63) .. (509.84,184.63) .. controls (508.93,184.63) and (508.2,183.93) .. (508.2,183.08) -- cycle ;
\draw  [fill={rgb, 255:red, 0; green, 0; blue, 0 }  ,fill opacity=1 ] (518.7,182.83) .. controls (518.7,181.97) and (519.43,181.28) .. (520.34,181.28) .. controls (521.25,181.28) and (521.98,181.97) .. (521.98,182.83) .. controls (521.98,183.68) and (521.25,184.38) .. (520.34,184.38) .. controls (519.43,184.38) and (518.7,183.68) .. (518.7,182.83) -- cycle ;

\draw (6,240.83) node [anchor=north west][inner sep=0.75pt]   [align=left] {Figure 6: Points inside the shaded disks represent arcs with tangency order higher than each surface exponent.};
\draw (27,30.4) node [anchor=north west][inner sep=0.75pt]    {$a)$};
\draw (364,29.4) node [anchor=north west][inner sep=0.75pt]    {$b)$};
\draw (156.33,105.6) node [anchor=north west][inner sep=0.75pt]    {$\theta _{1}$};
\draw (144.33,145.6) node [anchor=north west][inner sep=0.75pt]    {$\theta _{2}$};
\draw (197,148.6) node [anchor=north west][inner sep=0.75pt]    {$\theta _{3}$};
\draw (185.67,107.93) node [anchor=north west][inner sep=0.75pt]    {$\theta _{4}$};
\draw (248.67,52.27) node [anchor=north west][inner sep=0.75pt]    {$\theta _{5}$};
\draw (288.67,172.27) node [anchor=north west][inner sep=0.75pt]    {$\theta _{6}$};
\draw (60,174.27) node [anchor=north west][inner sep=0.75pt]    {$\theta _{7}$};
\draw (94.33,52.27) node [anchor=north west][inner sep=0.75pt]    {$\theta _{8}$};
\draw (469.47,56.27) node [anchor=north west][inner sep=0.75pt]    {$\lambda _{1}$};
\draw (539.47,56.88) node [anchor=north west][inner sep=0.75pt]    {$\lambda _{2}$};
\draw (521.07,85.47) node [anchor=north west][inner sep=0.75pt]    {$\lambda _{3}$};
\draw (522.27,165.87) node [anchor=north west][inner sep=0.75pt]    {$\lambda _{4}$};
\draw (571.47,71.87) node [anchor=north west][inner sep=0.75pt]    {$\lambda _{5}$};
\draw (439.87,67.47) node [anchor=north west][inner sep=0.75pt]    {$\lambda _{6}$};
\draw (488.27,165.07) node [anchor=north west][inner sep=0.75pt]    {$\lambda _{7}$};
\draw (488.27,83.87) node [anchor=north west][inner sep=0.75pt]    {$\lambda _{8}$};

\end{tikzpicture}
\end{center}

\begin{remark}
    Remarks \ref{Rem: removing segments} and \ref{Rem: removing piece of a nodal zone} explicit that if one intends to associated a word to a circular snake $X$, as it was done for snakes in Definition 6.10 of \cite{GabrielovSouza}, by removing a H\"older triangle $T(\gamma_1,\gamma_2) \subset X$ such that either $V(T(\gamma_1,\gamma_2))$ contains exactly one segment of $X$ or $\gamma_1, \gamma_2$ are in a same nodal zone of $X$, the obtained H\"older triangle $\overline{X \setminus T(\gamma_1,\gamma_2)}$ may not be a snake. Therefore, the existence of a circular snake associated with a given word does not naturally arises from Theorem 6.23 of \cite{GabrielovSouza} as the authors hopefully expected, i.e., an analog of this theorem for circular snakes will not easily follow from it. Consequently, a weakly outer classification for circular snakes turns to be harder (given the difficulty of choosing a word associated with it), once the word associated with a snake was fundamental to obtain Theorem 6.28 of \cite{GabrielovSouza}. 
\end{remark}

\section{Foundations of Circular Snake Names}

This section is devoted to establishing the basic concepts to understand the meaning of circular snake names. Every definition here is purely abstract in a first read, but their geometric and combinatorial motivations will become very clear in further sections.

\begin{Def}\label{DEF: word}
    Consider an alphabet $A=\{x_1,\ldots,x_n\}$. A \textbf{word} $W$ of length $m=|W|$ over $A$ is a finite sequence of $m$ letters in $A$, i.e., $W=[w_1 \cdots w_m]$ with $w_i\in A$ for $1\le i \le m$. One also considers the \textbf{empty word} $\varepsilon=[\;]$ of length $0$. Given a word $W=[w_1\cdots w_m]$, the letter $w_i$ is called the $i$-\textbf{th entry} of $W$.
	If $w_i=x$ for some $x\in A$, it is called a \textbf{node entry} of $W$ if it is the first occurrence of $x$ in $W$.
	Alternatively, $w_i$ is a node entry of $W$ if $w_j\ne w_i$ for all $j<i$
\end{Def}

\begin{Def}\label{DEF: open, semi-open subwords etc}
    Given a word $W=[w_1\cdots w_m]$, a \textbf{subword} of $W$ is either an empty word or a word $[w_j\cdots w_k]$ formed by consecutive entries of $W$ in positions $j, \ldots, k$, for some $1\le j\le k\le m$. We also consider \textbf{open} subwords $(w_j\cdots w_k)$ formed by the entries of $W$ in positions $j+1, \ldots, k-1$, for some $1\le j<k\le m$, and \textbf{semi-open} subwords $(w_j\cdots w_k]$ and $[w_j\cdots w_k)$ formed by the entries of $W$ in positions $j+1, \ldots, k$ and $j, \ldots, k-1$, respectively.
\end{Def}

\begin{Def}\label{Def: partition assoc with a snake name}
    Let $W=[w_1\cdots w_m]$ be a word of length $m$ containing $n$ distinct letters $x_1,\ldots,x_n$.
	We associate with $W$ a partition $P(W)=\{I_1,\ldots,I_n\}$ of the set $\{1,\ldots,m\}$ where $i\in I_j$ if and only if $w_i=x_j$. Two words $W$ and $W'$ are \textbf{equivalent} if $P(W)=P(W')$. In particular, equivalent words have the same length and the same number of distinct letters.
\end{Def}

\begin{remark}\label{REM: convention for word and equivalent words}
    Note that $P(W)$ does not depend on the alphabet, only on positions where the same letters appear. For convenience we often assign $a$ (or $x_1$) to the first letter of the word $W$, $b$ (or $x_2$) to the second letter of $W$ other than $a$, and so on. The words $X=abcdacbd$, $Y=bcdabdca$ and $Z=xyzwxzyw$ are equivalent, since
	$P(X)=P(Y)=P(Z)=\{ \{1,5\},\{2,7\},\{3,6\},\{4,8\}\}$.
\end{remark}

\begin{Def}\label{Def: circular word and power of a circ word}
    A word $W=[w_1\cdots w_m]$ is called a \textbf{circular word} if it contains at least two different letters and $w_1=w_m$. For $1\le j \le m-1$, we also define the \textbf{$j$-th power} of a circular word $W=[w_1\cdots w_m]$ as the word $W^j= [w_1^j\cdots w_m^j]$ given as follows: we set $w_1^j=w_m^j = w_j$ and $$w_i^j = \left\{
	\begin{array}{ccl}
		w_{j+i-1}, & \emph{if} & 2\le i \le m-j \\
		w_{i-m+j}, & \emph{if} & m-j < i \le m-1
	\end{array}\right..$$
In other words, $W^j=[w_j \cdots w_{m-1} w_1 \cdots w_j]$.
Note that $W^j$ is a circular word and $W^1=W$.
 Let $W$ and $W'$ be circular words. We say that $W'$ is a \textbf{circular subword} of $W$ if $W'$ is a subword of $W^j$ for some $1\le j \le m-1$. 
\end{Def}

\begin{remark}
    Let $W$ be a circular word of length $m$. In general, we do not have $P(W^i) = P(W)$ for $1\le i < m$. 
\end{remark}

\begin{Def}\label{Def: circular equivalence}
    We say that two circular words $W$ and $W'$ are \textbf{circular equivalent} if they 
 have the same lenght $m$ and $P(W') = P(W^i)$ for some $1\le i < m$. This relation defines an equivalence relation in the set of circular words with length $m$. We denote the class of $W$ under this relation by $[[W]]$. 
\end{Def}

\begin{Def}\label{Def: primitive and binary words}
    A word $W$ is \textbf{primitive} if it contains no repeated letters, i.e., if each part of $P(W)$ contains a single entry. We say that $W=[w_1\cdots w_m]$ is \textbf{semi-primitive} if $w_1=w_m$ and the subword $[w_1\cdots w_m)$ of $W$ is primitive, i.e., if each part of $P(W)$ except $I_1 = \{1, m\}$ contains a single entry. A word $W$ is \textbf{binary} if each of its letters appears in $W$ exactly twice, i.e., if each part of $P(W)$ contains exactly two entries.
\end{Def}

\begin{Def}\label{Def: rules for snakes names}
    Given a circular word $W=[w_1\cdots w_m]$, we say that $W$ is a \textbf{circular snake name} if the following conditions hold:	
	\begin{enumerate}
		\item[$(i)$] Each of the letters of $W$ appears in $W$ at least twice;
		\item[$(ii)$] For any $k\in\{1,\ldots,m-1\}$, there is a circular, semi-primitive subword of $W$ containing $w_k$ exactly once.
	\end{enumerate}
\end{Def}

\begin{remark}
    It follows from Definition \ref{Def: rules for snakes names} that a circular snake name has at least five letters.
\end{remark}

\section{Relating circular snakes with their names}\label{Subsec: snake names}

Let $X$ be a circular $\beta$-snake with $n$ nodes $\mathcal{N}_1,\ldots, \mathcal{N}_n$ and $m$ nodal zones $N_1,\ldots N_m$. Let $N$ be a fixed nodal zone in a node $\mathcal{N}$. Consider distinct arcs $\gamma_1, \gamma_2 \in N$ and the H\"older triangle $T=T(\gamma_1,\gamma_2)$ such that $V(T)\not\subset N$ (the H\"older triangle $T$ is the triangle in $X$ with $\gamma_1$, $\gamma_2$ as boundary arcs and lowest exponent). From now on we assume that the link of $T$ is oriented from $\gamma_1$ to $\gamma_2$, and the nodal zones of $X$ are enumerated in the order in which they appear when we move along the link of $T$ (which is a subset of the link of $X$) from $\gamma_{1}$ to $\gamma_{2}$. We enumerate the nodes of $X$ similarly, starting with the node $\mathcal{N}_1$ containing $\gamma_1$, skipping the nodes already assigned.
In particular, $\gamma_1, \gamma_2 \in N_1 \subset \mathcal{N}_1$. Consider an alphabet $A=\{x_1,\ldots,x_n\}$ where each letter $x_j$ is assigned to the node $\mathcal{N}_j$ of $X$.	

\begin{Def}\label{Def:agreed orientations}
Let $X$ be a circular snake and let $N$ and $N'$ be two distinct nodal zones of $X$ with orientations given by the H\"older triangles $T=T(\gamma_1,\gamma_2)$ and $T'=T(\gamma'_1,\gamma'_2)$, respectively. We say that the orientations of $N$ and $N'$ \textbf{agree} if the Valette link of the H\"older triangle $T(\gamma_1,\gamma'_1) \subset T$ contains $\gamma'_2$.  An \textbf{orientation of $X$} is the equivalence class of the orientation of nodal zones, where two orientations of $N$ and $N'$ are equivalent if, and only if, the orientations of $N$ and $N'$ agree.
\end{Def}
 
 \begin{Def}\label{Def: word associated with a snake}
 A word $W=[w_1\cdots w_m]$ over $A$ is \textbf{associated with the pair} $(X,N)$ if, while moving along the link of $T$ from $\gamma_1$ to $\gamma_2$, the $i$-th entry $w_i$ of $W$ is the letter $x_j$ assigned to the node $\mathcal{N}_j$ to which the nodal zone $N_i$ belongs.
\end{Def}

\begin{remark}\label{Rem: word associated with a snake}
By construction, there is a unique word over $A$ associated with $(X,N)$ related to a given orientation of $T$. So, we denote it by $W_N(X)$. When $X$ is understood, we may write just $W_N$.
\end{remark}

\begin{Exam}\label{Example:circular snake name}
In each of the three following figures, we have the link of a circular snake, with the shaded disks representing nodes. Each node is assigned to a letter $a$, $b$, $c$.

\tikzset{every picture/.style={line width=0.75pt}} 



Let $X$ be the circular snake whose link is represented in Figure 7a, $N$ the nodal zone containing the arcs $\gamma_1, \gamma_2$ and $\tilde N$ the nodal zone containing the arcs $\tilde{\gamma}_1, \tilde{\gamma}_2$. Taking into account the orientation of $T=T(\gamma_1,\gamma_2)$ and $\tilde T=T(\tilde{\gamma}_1,\tilde{\gamma}_2)$ (which agrees), the word over $A=\{a,b\}$ associated with $(X,N)$ is $W_N(X)=[ababa]$ and the word over $A$ associated with $(X,\tilde N)$ is $W_{\tilde N}(X)=[babab]$.

On the other hand, let $X$ be the circular snake whose link is represented in Figures 7b and 7c, and $N$ be the nodal zone containing the arcs $\gamma_1, \gamma_2$. Considering the orientation of $T=T(\gamma_1,\gamma_2)$ (Figure 2), we have $W_N=[abacbacba]$, and considering the orientation of $T'=T(\gamma_2,\gamma_1)$, we have $W_N=[abcabcaba]$. 
\end{Exam}

\begin{Prop}\label{Prop: word assoc with a snake is a snake name}
	Let $X$ be a circular snake and $W_N=W_N(X)$ be the word associated with $(X,N)$. Then $W_N$ is a circular snake name. 
\end{Prop}
\begin{proof}
    Let $W_N=[w_1 \cdots w_m]$. Since each $w_j$ is a letter assigned to a node of $V(X)$ and each node contains at least two nodal zones (Proposition \ref{Prop: nodes of a circular snakes}), condition $(i)$ of Definition \ref{Def: rules for snakes names} is satisfied.
 
    Let $z$ be any letter in $W_N$ and let $\gamma$ be an arc in the nodal zone $N_z$, endowed with the same orientation as $N$ (see Definition \ref{Def:agreed orientations}), which $z$ was assigned. Since $\gamma$ is abnormal, there are LNE H\"older triangles $T:=T(\theta,\gamma)$ and $T':=T(\gamma,\theta')$ in $X$ such that $T \cap T'$ is not LNE and $T \cap T'=\{\gamma\}$. Without loss of generality we may assume that $\theta$ and $\theta'$ are in nodal zones on the same node and that the orientation of $T(\theta,\theta')$ agrees with that on $N_z$. Thus, there exists a subword $\widetilde{W} = [\tilde{w}_1 \cdots \tilde{w}_n]$ of some power of $W_N$ containing $z$ and such that the letter $\tilde{w}_1$ is assigned to the nodal zone containing $\theta$ and $\tilde{w}_n$ is assigned that containing $\theta'$. 
 
    Up to replace the arcs $\theta$ and $\theta'$ and consider ``smaller'' H\"older triangles $T$ and $T'$, we may also assume that for any $\lambda,\lambda'\in V(T\cup T')$ with $\tord(\lambda,\lambda')>\itord(\lambda,\lambda')$, we have $\lambda,\lambda'\notin G(T\cup T')$. In particular, there is no pair of equal letters $\tilde{w}_i$ and $\tilde{w}_j$, $i<j$, in $\widetilde{W}$, other than $\tilde{w}_1$ and $\tilde{w}_n$. Consequently, $G(T\cup T')$ is a LNE zone, what implies that $\widetilde{W}$ is a circular semi-primitive subword of $W_N$ containing $z$ exactly once. Thus, condition $(ii)$ of Definition \ref{Def: rules for snakes names} is satisfied. 
\end{proof}

The next proposition is intuitive, as we construct circular snake names by starting on nodal zones with orientations agreeing, and then moving along the link. Since we can start from any nodal zone, it is expected to obtain snake names that are equivalent in a cyclic order.

\begin{Prop}\label{Prop:Circular Class}
	Let $X$ be a circular snake and let $N$ and $N'$ be two distinct nodal zones of $X$, with orientations agreeing. Then, $[[W_N(X)]] = [[W_{N'}(X)]]$ (see Definition \ref{Def: circular equivalence}).
\end{Prop}
\begin{proof}
Let $W_N = [w_1 \cdots w_m]$ and let $N_1, \cdots N_{m-1}$ be the nodal zones of $X$ such that $w_1$ and $w_m$ are assigned to $N_1$ and $w_i$ is assigned to $N_i$, for $i=2,\dots,m-1$. Then $N'=N_j$ for some $j \in \{ 2, \cdots, m-1 \}$, since $N\neq N'$. We claim that $W_{N'}=(W_N)^j = [w_j \cdots w_{m-1} w_1 \cdots w_j]$.

In each $N_i$ we consider a pair of arcs $\gamma_1^i,\gamma_2^i$ in a such way that the Valette link of the  H\"older triangle $T_i:=T(\gamma_1^i,\gamma_2^i)$ is not contained in $N_i$ and its orientation agrees with that of $N$ (consequently with that of $N'$). Let $\Gamma = \{\gamma_1^i,\gamma_2^i : i=1,\cdots,m-1\}$, and for $i\ge m$, define $\gamma_1^i=\gamma_1^{i-m+1}$, $\gamma_2^i=\gamma_2^{i-m+1}$. By construction of $W_N$ and orientation of the $T_i$, we deduce that each pair $\gamma_1^i,\gamma_2^{i+1}$ defines a H\"older triangle $T(\gamma_1^i,\gamma_2^{i+1})$, which contains no elements of $\Gamma$ and another one which contains all elements of $\Gamma$. Write $W_{N'} = [w_1' \cdots w_m']$ and $(W_N)^j = [w_1^j\cdots w_m^j]$. By construction, we have $w_m' = w_m^j = w_1^j = w_1'$. 

Assume, for sake of contradiction, that $w_1' = w_1^j, \cdots, w_r' = w_r^j$, but $w_{r+1}' \not = w_{r+1}^j$, for some $2\le r \le m-1$. Consider first the case $2 \leq r+1 \leq m-j$, so that $w_{r+1}^j = w_{r+j}$ (see Definition \ref{Def: circular word and power of a circ word}). Since $w_{r+1}' \not = w_{r+j}$ and the orientation of $T(\gamma_1^{r+j-1},\gamma_2^{r+j})$ agrees with that of $N'$, this triangle should contain arcs $\gamma_1^i,\gamma_2^i$ from the nodal zone assigned to $w_{r+1}'$, a contradiction. 
The case $m-j+1\le r+1 \le m$ is analogous, by changing $w_{r+j}$ to $w_{r+j-m+1}$.
\end{proof}

\begin{Cor}\label{Cor:Circular Class}
	Let $X$ be a circular snake and $N$ a nodal zone of $X$. If $W' \in [[W_N]]$, then there is a nodal zone $N'$ of $X$ such that the orientations of $N$ and $N'$ agree and $[[W_N]] = [[W_{N'}]]$.
\end{Cor}

\begin{remark}
    By Proposition \ref{Prop:Circular Class} and Corollary \ref{Cor:Circular Class}, we obtain that $X$ has two orientations. We assign a different sign in $\{1,-1\}$ for each such orientation.
\end{remark}

\begin{Def}
    Let $X$ be a circular snake and let $N$ be a nodal zone of $X$, which orientation induces an orientation $\varepsilon \in \{1,-1\}$ in $X$. If $\tilde N$ is a nodal zone which orientation induces an orientation $-\varepsilon$ in $X$, we define $-[[W_{N}X]]$ as the equivalence class $[[W_{\tilde N}X]]$.
\end{Def}

\begin{Def}\label{Def: circ snake name as a class}
    Let $X$ be a circular snake. For a fixed orientation on $X$, we define the circular snake name associated with $X$ to be $[[W_N(X)]]$, where $N$ is any nodal zone of $X$, whose orientation is in the equivalence class of the given orientation of $X$.
\end{Def}

\begin{remark}\label{Rem: Orientation Reverse not equal to cycle}
    One can have $[[W_N(X)]]=-[[W_N(X)]]$ and an example is the circular snake $X$ represented by Figure 1 of Example \ref{Example:circular snake name}, where $[[W_N(X)]]=-[[W_N(X)]]=[[abab]]$. On the other hand, it is also possible that $-[[W_N(X)]] \ne [[W_N(X)]]$. For example, consider the circular snake $X$ whose link is represented in Figures 8a and 8b, each figure with the orientation indicated by the corresponding arrows.
    \begin{center}

\tikzset{every picture/.style={line width=0.75pt}} 


    \end{center}

Note that $$[[W_N]]=[[x_1x_2x_3x_4x_5x_6x_4x_2x_3x_5x_1x_6x_1]]; \, -[[W_{N}]]=[[x_1x_2x_1x_3x_4x_5x_6x_2x_3x_6x_4x_5x_1]].$$

The only word in $[[W_N]]$ whose first and third letters are equal is $$(W_N)^{11}=[x_1x_6x_1x_2x_3x_4x_5x_6x_4x_2x_3x_5x_1].$$ 
Then, if $-W_{N} \in [[W_N]]$, $-W_{N}$ and $(W_N)^{11}$ are equivalent. However, $-W_{ N}$ is equivalent to $[abacdefbcfdea]$ (by the identification $x_1\mapsto a$; $\dots$; $x_6\mapsto f$) and $(W_N)^{11}$ is equivalent to $[abacdefbecdfa]$ (by the identification $x_1\mapsto a$; $x_6\mapsto b$; $x_2\mapsto c$; $\dots$; $x_5\mapsto f$). Since $[abacdefbcfdea] \ne [abacdefbcfdea]$, we have a contradiction and therefore $[[W_N]]\ne -[[W_{ N}]]$.
\end{remark}



\section{Realization Theorem for circular snake names}

The goal of this section is to prove that each circular snake name can be realized as a circular snake. The last of the following three definitions describes the construction of a surface with circular link. The results presented in the sequel establishes the necessary properties for this surface to be the desired abnormal surface.

\begin{Def}\label{Def: node entry and the sets Pi's}
 \normalfont Let $W=[w_1\cdots w_m]$ be a circular snake name. We define recursively the index function $r: \{1,\dots,m\} \to \{1,\dots,m\}$ as follows: $r(1)=1$ and, for every $2\le j\le m$, if $w_j$ is a node entry, then set $r(j)=j$; otherwise $r(j)=\ell$, where $\ell<j$ is such that $w_\ell$ is a node entry and $w_{\ell}=w_j$.
\end{Def}

\begin{Def}\label{Def: linear HT}
	\normalfont Given arcs $\gamma,\gamma'\in V(\R^p)$, we define the set $\Delta(\gamma,\gamma')$ as the union of straight line segments $[\gamma(t),\gamma'(t)]$ connecting $\gamma(t)$ and $\gamma'(t)$, for any $t\ge 0$.
\end{Def}

\begin{Def}\label{Def: arcs and HT assoc with a snake name}
	\normalfont Let $W=[w_1,\ldots,w_m]$ be a circular snake name. Consider the space $\R^{2m-2}$ with the standard basis $\mathbf{e}_1,\ldots, \mathbf{e}_{2m-2}$. Let $\alpha,\beta\in \F$, with $1\le \beta < \alpha$, and let $\delta_1,\ldots, \delta_m$ and $\sigma_1,\ldots, \sigma_{m-1}$ be arcs in $V(\R^{2m-2})$ (parameterized by the first coordinate, which is equivalent to the distance to the origin) such that:
	\begin{enumerate}
		\item $\delta_1(t)=\delta_m(t)=t\mathbf{e}_1$;
		\item for $1<j< m$, if $w_j$ is a node entry then $\delta_j(t) = t\mathbf{e}_1 + t^\beta \mathbf{e}_j$. Otherwise, $\delta_j(t) = \delta_{r(j)}(t) + t^\alpha \mathbf{e}_j$;
		\item for any $j=1,\ldots,m-1$, we define $\sigma_j(t)=t\mathbf{e}_1 + t^\beta \mathbf{e}_{m-1+j}$.
	\end{enumerate}

 Consider, for $j=1,\ldots,m-1$, the $\beta$-H\"older triangles $T_j=\Delta(\delta_j,\sigma_j)\cup \Delta(\sigma_j,\delta_{j+1})$, and for $j=1,\ldots, m-2$, the $\beta$-H\"older triangles $T'_j=\Delta(\sigma_{j},\delta_{j+1})\cup \Delta(\delta_{j+1},\sigma_{j+1})$. Let $S_{W}=\bigcup_{j=1}^{m-1}T_j$. The surface $S_W$ is defined as the \textbf{$\beta$-surface associated with the circular snake name} $W$. By construction, $S_W$ has circular link, which will be assumed to be oriented from $\delta_1$ to $\delta_m$ in such a way that the arcs $\delta_j$ and $\sigma_k$ appear in $S_W$ in the order $\delta_1,\sigma_1,\delta_2,\ldots,\delta_{m-1},\sigma_{m-1},\delta_m$.
\end{Def}

\begin{remark}
Note that each H\"older triangle $\Delta(\sigma_{j},\delta_{j})$ lies in the linear space generated by $\be_1,\be_j,\be_{r(j)},\be_{m-1+j}$ and each H\"older triangle $\Delta(\delta_{j}, \sigma_{j+1})$ lies in the linear space generated by $\be_1,\be_{j+1},\be_{r(j+1)},\be_{m-1+j}$. Since $j\ne j+1$ and $r(j)\ne r(j+1)$, we have $\Delta(\sigma_j,\delta_j)\cap \Delta(\sigma_\ell,\delta_\ell) , \, \Delta(\sigma_j,\delta_j)\cap \Delta(\delta_\ell,\sigma_{\ell+1}) \ne \{0\}$ if and only if $j=\ell$.
Therefore, the surface $S_W$ does not have self-intersection.
\end{remark}

Now we proceed towards proving that $S_W$ has the structure of a circular snake. The following Lemma is a direct consequence of Definition \ref{Def: arcs and HT assoc with a snake name}. This result shows that $S_W$ has the suitable tangency order setting to realize the word $W$. 
\begin{Lem}\label{Lem: tords of deltas and sigmas}
	The arcs $\delta_1,\ldots, \delta_m, \sigma_1,\ldots, \sigma_{m-1}$ of Definition \ref{Def: arcs and HT assoc with a snake name} satisfy the following:
	\begin{enumerate}
		\item[(i)] for all $i\ne j$, $\tord(\delta_i,\delta_{j})=\left\{
		\begin{array}{cl}
		\alpha & \text{if} \;  w_i=w_{j} \\
		\beta & otherwise
		\end{array}\right.$;
		\item[(ii)] for all $i$ and $j$, $\tord(\sigma_i,\delta_{j})=\beta$;
		\item[(iii)] for all $i$ and $j$ with $i\ne j$, $\tord(\sigma_i,\sigma_{j})=\beta$ .
	\end{enumerate}
\end{Lem}

The following two Lemmas are analogous to Lemmas 6.18 and 6.19 in \cite{GabrielovSouza}, and the proof presented there also works here analogously.

\begin{Lem}\label{Lem: Tj is NE}
	Each $T_j$ and $T_j'$ in Definition \ref{Def: arcs and HT assoc with a snake name} is a LNE $\beta$-H\"older triangle.
\end{Lem}

\begin{Cor}\label{Cor: HT assoc with a snake name is non-singular}
	Let $W$ be a circular snake name and let $S_W$ be the surface associated with $W$ as in Definition \ref{Def: arcs and HT assoc with a snake name}. Then every arc $\gamma \in V(S_W)$ is Lipschitz non-singular.
\end{Cor}

\begin{proof}
    Indeed, any $\gamma \in V(S_W)$ is an interior arc of some $T_j$ or $T_j'$, which are LNE.
\end{proof} 

The statement below is the analogue of Lemma 6.21 in \cite{GabrielovSouza} where $W$ is a snake name instead of circular snake name. Since the assertions deals with a primitive subword of $W$, the proof there also works here by choosing a suitable power $W^j$ of the circular snake name and considering it as a snake name.

\begin{Lem}\label{Lem: HT assoc with a primitive subsnake name is NE}
	Let $W=[w_1\cdots w_m]$ and $S_W$ be as in Corollary \ref{Cor: HT assoc with a snake name is non-singular}. If $W'=[w_j\cdots w_l]$ is a primitive circular subword of $W$ then $T(\delta_j,\delta_l)\subset S_W$ is a LNE $\beta$-H\"older triangle.
\end{Lem}

\begin{Cor}\label{Cor: interior deltas of T are abnormal}
	Let $W=[w_1\cdots w_m]$ be a circular snake name and let $S_W$ be the surface associated with $W$. Then $V(S_W) = \Abn(S_W)$.
\end{Cor}

\begin{proof}
    Each $\delta_k$ is an abnormal arc, for $k = 1,\ldots, m-1$. 
    Indeed, since $W$ is a circular snake name, there exists a semi-primitive subword $[w_i \ldots w_l]$ of some power $W^j$ of $W$ containing $w_k$ as interior letter. Therefore, $[w_i \ldots w_k]$ and $[w_k \ldots w_l]$ are primitive. By Lemma \ref{Lem: HT assoc with a primitive subsnake name is NE}, the H\"older triangles $T(\delta_i,\delta_k)$ and $T(\delta_k,\delta_l)$ are LNE. Moreover, $\tord(\delta_i,\delta_l) = \alpha > \beta = \itord(\delta_i,\delta_l)$ since $w_i=w_l$. Thus, $\delta_k$ is abnormal.

    Let $\gamma \in V(S_W)$. By considering a $j$-th power of $W$, we may assume that $\gamma \subset T_{k-1}$, for some $k>1$ such that $[w_i \ldots w_l]$ is a primitive subword realizing the abnormality of $\delta_k$ with $1 \leq i < k < l < k $. It follows that the subwords $[w_i \ldots w_{k-1}]$ and $[w_k \ldots w_l]$ are also primitive and hence the H\"older triangles $T(\delta_i,\gamma)$ and $T(\gamma,\delta_l)$ are LNE, since they are contained in the LNE H\"older triangles $T(\delta_i,\delta_k)$ and $T(\delta_{k-1},\delta_l)$, respectively. Thus, the arc $\gamma$ is abnormal, and the result follows.
\end{proof}

\begin{Teo}\label{Teo: snake name realization}
Let $W$ be a circular snake name. Then, there exists a circular snake $S$ such that $W = W_N(S)$, where $N$ is a nodal zone of $S$. 
\end{Teo}

\begin{proof}
    Let $W=[w_1\ldots w_m]$ be a circular snake name with $n$ distinct letters. Let $S=S_W$ be the $\beta$-surface associated with $W$ (see Definition \ref{Def: arcs and HT assoc with a snake name}) and let $N$ be a nodal zone whose orientation induces the same orientation of $S_W$. We claim that $S$ is a circular $\beta$-snake. Indeed, Corollary \ref{Cor: HT assoc with a snake name is non-singular} implies that $S$ is a non-singular $\beta$-surface, so it remains to prove that $G(S)=\Abn(S)$. The inclusion $\Abn(S)\subset G(S)$ is obvious, and the inverse inclusion is given by Corollary \ref{Cor: interior deltas of T are abnormal}.

	Finally, to prove that $W=W_N(S)$, note that the link of $S$ is oriented from $\delta_1$ to $\delta_m$ (see Definition \ref{Def: arcs and HT assoc with a snake name}), the $i$-th nodal zone of $T$ is $N_i=\{\gamma \in V(S) : \itord(\gamma,\delta_i)>\beta\}$, for $i=1,\ldots,m$, and the $k$-th node of $S$ is $\mathcal{N}_k=\bigcup_{i  \in I_k}N_i$ for each $k=1,\ldots,n$ (here $I_k$ is as in Definition \ref{Def: partition assoc with a snake name}).
In particular, $S$ is a circular $\beta$-snake with $m$ nodal zones and $n$ nodes, such that $W = W_{N_1}(S)$ and the result follows.
\end{proof}

\section{Weak outer bi-Lipschitz classification for circular snakes}\label{Subsection: weak equivalence}

\begin{Def}\label{Def: clusters}
Let $\cN$ and $\cN'$ be nodes of a circular $\beta$-snake $X$, and let $\mathcal{S}(\cN,\cN')$ be the (possibly empty) set of all segments of $X$ having adjacent nodal zones in the nodes $\cN$ and $\cN'$. Two segments $S$ and $S'$ in $\mathcal{S}(\cN,\cN')$ belong to the same \textbf{cluster}
if $\tord(S,S')>\beta$. This defines a \textbf{cluster partition} of $\mathcal{S}(\cN,\cN')$.
The size of each cluster $C$ of this partition is equal to the multiplicity of each segment $S\in C$ (see Definition \ref{Def of multiplicity}).
\end{Def}

In this section, we consider the combinatorial and geometric significance of the cluster partitions of the sets $\mathcal{S}(\cN,\cN')$ in Definition \ref{Def: clusters}. In order to do so, one must look to weak bi-Lipschitz maps.

\begin{Def}\label{Def:weak equivalence}
Let $h:X\to X'$ be a homeomorphism of two $\beta$-surfaces with circular link $X$ and $X'$, bi-Lipschitz with respect to the inner metrics of $X$ and $X'$. We say that $h$ is \textbf{weakly outer bi-Lipschitz} when, for any two arcs $\gamma$ and $\gamma'$ of $V(X)$, we have
$$
\tord(h(\gamma),h(\gamma'))>\beta \iff \tord(\gamma,\gamma')>\beta.
$$
If such a homeomorphism exists, we say that $X$ and $X'$ are \textbf{weakly outer Lipschitz equivalent}.
\end{Def}

We now describe when two circular $\beta$-snakes are weakly outer Lipschitz equivalent. When they have nodal zones, the following criterion is an equivalent version of Theorem 6.28 of \cite{GabrielovSouza}.

\begin{Teo}\label{Teo:weak equivalence}
Two circular $\beta$-snakes $X$ and $X'$ with nodal zones are weakly outer Lipschitz equivalent if, and only if, they can be oriented so that
\begin{enumerate}
\item[(i)] They have the same circular snake names (see Definition \ref{Def: circ snake name as a class}), the nodes $\cN_1,\ldots,\cN_n$ of $X$ are in one-to-one correspondence with the nodes
$\cN'_1,\ldots,\cN'_n$ of $X'$ (with $\cN_i$ corresponding to $\cN'_i$, for each $i$), and the nodal zones $N_1,\ldots,N_m$ of $X$ are in one-to-one correspondence with the nodal zones $N'_1,\ldots,N'_m$ of $X'$ (with $N_i$ corresponding to $N'_i$, for each $i$);
\item[(ii)] For any two nodes $\cN_j$ and $\cN_k$ of $X$, and the corresponding nodes $\cN'_j$ and $\cN'_k$ of $X'$,
each cluster of the cluster partition of the set $\mathcal{S}(\cN'_j,\cN'_k)$ (see Definition \ref{Def: clusters}) consists
of the segments of $X'$ corresponding to the segments of $X$ contained in a cluster of the cluster partition of the set $\mathcal{S}(\cN_j,\cN_k)$. This correspondence is determined by the orientation of $X$ and $X'$ in a continuous way, and so each segment containing $\cN_j$ and/or $\cN_k$ correspond to a segment $\cN'_j$ and/or $\cN'_k$.
\end{enumerate}
\end{Teo}

\begin{proof}
    The tangency order condition in Definition \ref{Def:weak equivalence} implies that a weakly outer bi-Lipschitz homeomorphism $h:X\to X'$ defines equivalence of the circular snake names $W=W(X)$ and $W'=W(X')$, and identifies cluster partitions of the sets $\mathcal{S}(N_j,N_k)$ and $\mathcal{S}(N'_j,N'_k)$ for any $j$ and $k$. Thus, the direct implication holds. In the following, we prove that conditions $(i)$ and $(ii)$ of Theorem \ref{Teo:weak equivalence} imply weak outer Lipschitz equivalence of the circular snakes $X$ and $X'$. Here we are following the given orientations for $X$ and $X'$.

By Theorem \ref{Teo: segments are LNE}, the nodal zones adjacent to any segment of $X$ and $X'$ lie in two distinct nodes.
Since the snake names $W$ and $W'$ are equivalent, each nodal zone $N_j$ of $X$ corresponds to the $j$-th entry $w_j$ of $W$ and each nodal zone $N'_j$ of $X'$ corresponds to the $j$-th entry $w'_j$ of $W'$. Moreover, nodal zones $N_j$ and $N_k$ of $X$ (resp., $N'_j$ and $N'_k$ of $X'$) belong to the same node if, and only if, $w_j=w_k$ (resp., $w'_j=w'_k$).

According to Corollary \ref{Cor: intrinsic pancake decomp of a CS}, if we choose an arc $\gamma_j$ in each nodal zone $N_j$ of $X$ and an arc $\gamma'_j$ in each nodal zone $N'_j$ of $X'$ we obtain pancake decompositions of $X$ and $X'$, such that each pancake $X_j=T(\gamma_j,\gamma_{j+1})$ of $X$ (resp., pancake $X'_j=T(\gamma'_j,\gamma'_{j+1})$ of $X'$) is a $\beta$-H\"older triangle corresponding to a segment of $X$ with adjacent nodal zones $N_j$ and $N_{j+1}$ (resp., to a segment of $X'$ with adjacent nodal zones $N'_j$ and $N'_{j+1}$). 

We construct a weakly outer bi-Lipschitz homeomorphism $h:X\to X'$ as follows:

\begin{enumerate}
    \item First, we define $h$ on each arc $\gamma_j$ as the map $\gamma_j\to\gamma'_j$ can be consistent with the parameterizations of both arcs by the distance to the origin (see Corollary 0.2 of \cite{valette2007link}).
    \item Next, for each pair of nodes $\mathcal{N}$ and $\mathcal{N}'$ of $X$, if the set $\mathcal{S}=\mathcal{S}(\mathcal{N},\mathcal{N}')$ is not empty, we choose a pancake $X_j=T(\gamma_j,\gamma_{j+1})$ corresponding to a segment from each cluster of the cluster partition of $\mathcal{S}$. Define a bi-Lipschitz homeomorphism $h_j:X_j\to X'_j$ consistent with the previously defined mappings for the arcs $\gamma_j$ and $\gamma_{j+1}$ (this is always possible, because $X_j$ and $ X'_j$ are LNE $\beta$-H\"older triangles), and define $h$ as $h_j$ in each such $X_j$.
    
    \item Finally, for any cluster of $\mathcal{S}$ of size higher than one (see Definition \ref{Def: clusters}) containing a segment for which we set $h$ to agree with $h_j$ on the corresponding pancake $X_j$, we define $h$ on the pancake $X_k$ corresponding to that segment as follows.
Since pancakes $X_j$ and $X_k$ correspond to segments in the same cluster, pancakes $X'_j$ and $X'_k$ also correspond to segments in the same cluster.
It follows from Proposition 2.20 of \cite{GabrielovSouza} that there is a bi-Lipschitz homeomorphism $h_{kj}:X_k\to X_j$ such that $\tord(\gamma,h_{kj}(\gamma))>\beta$ for each arc $\gamma\subset X_k$,
and a bi-Lipschitz homeomorphism $h'_{jk}:X'_j\to X'_k$ such that $\tord(\gamma',h'_{jk}(\gamma'))>\beta$ for each arc $\gamma'\subset X'_j$.
Then we define $h$ to map $X_k$ to $X'_k$ as the composition $h'_{jk} \circ h_j \circ h_{kj}$.
\end{enumerate}

It is straightforward that $h$ is a well-defined weakly outer bi-Lipschitz homeomorphism by construction.  
\end{proof}

When two abnormal surfaces with circular link do not have nodal zones, we have a surprisingly simple criterion for weakly outer bi-Lipschitz equivalence, depending only on their multiplicity (see Corollary \ref{multiplicidade-sem-nodos}).

\begin{Teo}\label{Teo:Multiplicity without Nodes}
Let $\beta \in \mathbb{F}_{\ge 1}$ and let $X$ and $X'$ be two circular $\beta$-snakes without nodal zones. Then, $X$ and $X'$ are weakly outer Lipschitz equivalent if, and only if, $X$ and $X'$ have the same multiplicity.
\end{Teo}

\begin{proof}
Let $m$, $m'$ be the respective multiplicities of $X$ and $X'$. Consider an arc $\gamma_0 \in V(X)$ and take $\gamma_1 \in V(X)$ such that, when $X$ is oriented from $\gamma_0$ to $\gamma_1$, $T(\gamma_0,\gamma_1)$ is a LNE $\beta$-H\"older triangle. Notice that $\gamma_1$ always exists, because $X$ has no nodal zones and then $\mathcal{H}_{\eta,\beta}(\gamma_0)$ consists of $m$ LNE $\beta$-H\"older triangles whose intersection is $\{0\}$, for $\eta>0$ small enough. Thus, we can choose $\gamma_1$ in the same triangle containing $\gamma_0$.
    
Consider $\delta>0$ small enough such that, for $\alpha=\beta+\delta \in \mathbb{F}_{>\beta}$, we have that $\mathcal{H}_{1,\alpha}(\gamma_0)$ consists of $m$ LNE $\alpha$-H\"older triangles $T_2, T_4, \dots, T_{2m}$ whose intersection is $\{0\}$, and $\mathcal{H}_{1,\alpha}(\gamma_1)$ consists of $m$ LNE $\alpha$-H\"older triangles $T_{1}, T_3 \dots,  T_{2m-1}$ whose intersection is $\{0\}$, with $\gamma_0 \in V(T_{2m})$ and $\gamma_1 \in V(\tilde T_{1})$. By re-indexing, if necessary, assume that when we go through the link of $X$, from $\gamma_1$ to $\gamma_0$ and according to its orientation, then we pass in this order to $T_1$, $T_2$, $\dots$, $T_{2m}$. Now take arbitrary $\gamma_{i} \in V(T_i)$, for $i=2,\dots,2m-1$ and set $\gamma_{2m}=\gamma_0$. For $i=1,2,\dots,2m$, define $X_i=T(\gamma_{i-1}, \gamma_i)$ (following the orientation of $X$ - see Figure 9).

\begin{center}
\tikzset{every picture/.style={line width=0.75pt}} 

\begin{tikzpicture}[x=0.75pt,y=0.75pt,yscale=-0.9,xscale=0.9]

\draw    (131,158.3) .. controls (130,217.3) and (498,217.3) .. (498,157.3) ;
\draw  [dash pattern={on 0.84pt off 2.51pt}]  (498,157.3) .. controls (496,100.3) and (129,79.3) .. (129,132.3) ;
\draw    (129,132.3) .. controls (128,191.3) and (496,191.3) .. (496,131.3) ;
\draw  [dash pattern={on 0.84pt off 2.51pt}]  (496,131.3) .. controls (494,74.3) and (127,53.3) .. (127,106.3) ;
\draw    (127,106.3) .. controls (126,165.3) and (494,165.3) .. (494,105.3) ;
\draw  [dash pattern={on 0.84pt off 2.51pt}]  (494,105.3) .. controls (492,48.3) and (125,27.3) .. (125,80.3) ;
\draw    (125,80.3) .. controls (124,139.3) and (492,139.3) .. (492,79.3) ;
\draw  [dash pattern={on 0.84pt off 2.51pt}]  (492,79.3) .. controls (491,47.3) and (418.4,48.15) .. (406.4,57.15) ;
\draw  [draw opacity=0][fill={rgb, 255:red, 0; green, 0; blue, 0 }  ,fill opacity=0.28 ] (109.96,123.3) .. controls (109.96,84.64) and (124.7,53.3) .. (142.88,53.3) .. controls (161.06,53.3) and (175.8,84.64) .. (175.8,123.3) .. controls (175.8,161.96) and (161.06,193.3) .. (142.88,193.3) .. controls (124.7,193.3) and (109.96,161.96) .. (109.96,123.3) -- cycle ;
\draw  [dash pattern={on 0.84pt off 2.51pt}]  (301.4,108.65) .. controls (248.4,137.65) and (130,127.3) .. (131,158.3) ;
\draw  [dash pattern={on 0.84pt off 2.51pt}]  (339.4,85.65) .. controls (328.4,96.15) and (322.4,89.15) .. (313.4,98.65) ;
\draw  [dash pattern={on 0.84pt off 2.51pt}]  (397.4,67.15) .. controls (386.4,77.65) and (359.9,69.15) .. (350.9,78.65) ;
\draw [line width=2.25]    (134.91,152.56) .. controls (134.21,154.28) and (133.9,155.95) .. (133.9,157.55) .. controls (133.9,164.13) and (138.94,169.67) .. (143.83,173.49) .. controls (147.5,176.35) and (151.16,178.27) .. (152.87,178.87)(132.12,151.44) .. controls (131.27,153.54) and (130.9,155.58) .. (130.9,157.55) .. controls (130.9,165.01) and (136.37,171.47) .. (141.98,175.85) .. controls (145.99,178.98) and (150.01,181.04) .. (151.87,181.7) ;
\draw [line width=2.25]    (133.19,127.99) .. controls (132.5,129.71) and (132.19,131.37) .. (132.19,132.98) .. controls (132.19,139.56) and (137.23,145.1) .. (142.11,148.92) .. controls (145.78,151.78) and (149.45,153.7) .. (151.16,154.3)(130.41,126.87) .. controls (129.56,128.97) and (129.19,131.01) .. (129.19,132.98) .. controls (129.19,140.44) and (134.65,146.9) .. (140.27,151.28) .. controls (144.28,154.41) and (148.29,156.47) .. (150.16,157.13) ;
\draw [line width=2.25]    (130.62,100.56) .. controls (129.93,102.28) and (129.62,103.95) .. (129.62,105.55) .. controls (129.62,112.13) and (134.65,117.67) .. (139.54,121.49) .. controls (143.21,124.35) and (146.88,126.27) .. (148.58,126.87)(127.84,99.44) .. controls (126.99,101.54) and (126.62,103.58) .. (126.62,105.55) .. controls (126.62,113.01) and (132.08,119.47) .. (137.7,123.85) .. controls (141.71,126.98) and (145.72,129.04) .. (147.59,129.7) ;
\draw [line width=2.25]    (129.76,74.59) .. controls (129.07,76.31) and (128.76,77.97) .. (128.76,79.58) .. controls (128.76,86.16) and (133.8,91.7) .. (138.69,95.52) .. controls (142.36,98.38) and (146.02,100.3) .. (147.73,100.9)(126.98,73.47) .. controls (126.13,75.57) and (125.76,77.61) .. (125.76,79.58) .. controls (125.76,87.04) and (131.22,93.5) .. (136.84,97.88) .. controls (140.85,101.01) and (144.86,103.07) .. (146.73,103.73) ;
\draw  [draw opacity=0][fill={rgb, 255:red, 0; green, 0; blue, 0 }  ,fill opacity=0.28 ] (449.16,123.97) .. controls (449.16,85.31) and (463.9,53.97) .. (482.08,53.97) .. controls (500.26,53.97) and (515,85.31) .. (515,123.97) .. controls (515,162.63) and (500.26,193.97) .. (482.08,193.97) .. controls (463.9,193.97) and (449.16,162.63) .. (449.16,123.97) -- cycle ;
\draw [line width=2.25]    (498.14,149.27) .. controls (498.98,151.37) and (499.36,153.41) .. (499.36,155.38) .. controls (499.36,162.84) and (493.89,169.3) .. (488.28,173.68) .. controls (484.27,176.81) and (480.25,178.87) .. (478.39,179.53)(495.35,150.39) .. controls (496.05,152.11) and (496.36,153.77) .. (496.36,155.38) .. controls (496.36,161.96) and (491.32,167.5) .. (486.43,171.32) .. controls (482.76,174.18) and (479.1,176.1) .. (477.39,176.7) ;
\draw [line width=2.25]    (495.34,124.87) .. controls (496.18,126.97) and (496.56,129.01) .. (496.56,130.98) .. controls (496.56,138.44) and (491.09,144.9) .. (485.48,149.28) .. controls (481.47,152.41) and (477.45,154.47) .. (475.59,155.13)(492.55,125.99) .. controls (493.25,127.71) and (493.56,129.37) .. (493.56,130.98) .. controls (493.56,137.56) and (488.52,143.1) .. (483.63,146.92) .. controls (479.96,149.78) and (476.3,151.7) .. (474.59,152.3) ;
\draw [line width=2.25]    (493.34,98.47) .. controls (494.18,100.57) and (494.56,102.61) .. (494.56,104.58) .. controls (494.56,112.04) and (489.09,118.5) .. (483.48,122.88) .. controls (479.47,126.01) and (475.45,128.07) .. (473.59,128.73)(490.55,99.59) .. controls (491.25,101.31) and (491.56,102.97) .. (491.56,104.58) .. controls (491.56,111.16) and (486.52,116.7) .. (481.63,120.52) .. controls (477.96,123.38) and (474.3,125.3) .. (472.59,125.9) ;
\draw [line width=2.25]    (491.74,73.27) .. controls (492.58,75.37) and (492.96,77.41) .. (492.96,79.38) .. controls (492.96,86.84) and (487.49,93.3) .. (481.88,97.68) .. controls (477.87,100.81) and (473.85,102.87) .. (471.99,103.53)(488.95,74.39) .. controls (489.65,76.11) and (489.96,77.77) .. (489.96,79.38) .. controls (489.96,85.96) and (484.92,91.5) .. (480.03,95.32) .. controls (476.36,98.18) and (472.7,100.1) .. (470.99,100.7) ;
\draw    (104.33,177.97) -- (129.45,166.47) ;
\draw [shift={(131.27,165.63)}, rotate = 155.4] [color={rgb, 255:red, 0; green, 0; blue, 0 }  ][line width=0.75]    (10.93,-3.29) .. controls (6.95,-1.4) and (3.31,-0.3) .. (0,0) .. controls (3.31,0.3) and (6.95,1.4) .. (10.93,3.29)   ;
\draw    (518.53,71.97) -- (490.97,86.37) ;
\draw [shift={(489.2,87.3)}, rotate = 332.4] [color={rgb, 255:red, 0; green, 0; blue, 0 }  ][line width=0.75]    (10.93,-3.29) .. controls (6.95,-1.4) and (3.31,-0.3) .. (0,0) .. controls (3.31,0.3) and (6.95,1.4) .. (10.93,3.29)   ;
\draw    (521.87,124.3) -- (494.31,138.71) ;
\draw [shift={(492.53,139.63)}, rotate = 332.4] [color={rgb, 255:red, 0; green, 0; blue, 0 }  ][line width=0.75]    (10.93,-3.29) .. controls (6.95,-1.4) and (3.31,-0.3) .. (0,0) .. controls (3.31,0.3) and (6.95,1.4) .. (10.93,3.29)   ;
\draw    (523.2,151.3) -- (495.64,165.71) ;
\draw [shift={(493.87,166.63)}, rotate = 332.4] [color={rgb, 255:red, 0; green, 0; blue, 0 }  ][line width=0.75]    (10.93,-3.29) .. controls (6.95,-1.4) and (3.31,-0.3) .. (0,0) .. controls (3.31,0.3) and (6.95,1.4) .. (10.93,3.29)   ;
\draw    (103.33,150.97) -- (128.45,139.47) ;
\draw [shift={(130.27,138.63)}, rotate = 155.4] [color={rgb, 255:red, 0; green, 0; blue, 0 }  ][line width=0.75]    (10.93,-3.29) .. controls (6.95,-1.4) and (3.31,-0.3) .. (0,0) .. controls (3.31,0.3) and (6.95,1.4) .. (10.93,3.29)   ;
\draw    (99.67,97.63) -- (124.78,86.13) ;
\draw [shift={(126.6,85.3)}, rotate = 155.4] [color={rgb, 255:red, 0; green, 0; blue, 0 }  ][line width=0.75]    (10.93,-3.29) .. controls (6.95,-1.4) and (3.31,-0.3) .. (0,0) .. controls (3.31,0.3) and (6.95,1.4) .. (10.93,3.29)   ;

\draw (145.29,160.43) node [anchor=north west][inner sep=0.75pt]    {$T_{1}$};
\draw (145,133.57) node [anchor=north west][inner sep=0.75pt]    {$T_{3}$};
\draw (144.14,107.14) node [anchor=north west][inner sep=0.75pt]    {$\vdots $};
\draw (133.86,73) node [anchor=north west][inner sep=0.75pt]    {$T_{2m-1}$};
\draw (460.22,157.9) node [anchor=north west][inner sep=0.75pt]    {$T_{2}$};
\draw (458.49,132.16) node [anchor=north west][inner sep=0.75pt]    {$T_{4}$};
\draw (458.14,107.14) node [anchor=north west][inner sep=0.75pt]    {$\vdots $};
\draw (457.86,73) node [anchor=north west][inner sep=0.75pt]    {$T_{2m}$};
\draw (111,28.73) node [anchor=north west][inner sep=0.75pt]    {$\mathcal{H}_{1,\alpha}( \gamma _{1})$};
\draw (455,30.07) node [anchor=north west][inner sep=0.75pt]    {$\mathcal{H}_{1,\alpha}( \gamma _{0})$};
\draw (93.67,213) node [anchor=north west][inner sep=0.75pt]   [align=left] {Figure 9: The construction of $\displaystyle T_{1}$, $\displaystyle T_{2}$,$\displaystyle \dotsc $, $\displaystyle T_{2m}$ in the proof of Theorem \ref{Teo:Multiplicity without Nodes}.};
\draw (519,66.73) node [anchor=north west][inner sep=0.75pt]    {$\gamma _{2m} =\gamma _{0}$};
\draw (86.33,173.4) node [anchor=north west][inner sep=0.75pt]    {$\gamma _{1}$};
\draw (524.8,118.53) node [anchor=north west][inner sep=0.75pt]    {$\gamma _{4}$};
\draw (525.13,143.53) node [anchor=north west][inner sep=0.75pt]    {$\gamma _{2}$};
\draw (85.67,146.27) node [anchor=north west][inner sep=0.75pt]    {$\gamma _{3}$};
\draw (81.33,95.6) node [anchor=north west][inner sep=0.75pt]    {$\gamma _{2m-1}$};
\draw (293.67,21.07) node [anchor=north west][inner sep=0.75pt]  [font=\Large]  {$X$};
\end{tikzpicture}
\end{center}

Note that $X_1$ is LNE by construction. For $2\le i\le 2m$, $X_i$ is also LNE, because $\gamma_{i-1} \in T_{i-1}$, $\gamma_{i} \in T_{i}$, $tord(\gamma_{i-1},\gamma_i)=\beta$ by construction, and $\mathcal{H}(X_i)_{\eta,\beta}(\gamma)$ consists of only one LNE H\"older triangle, for each $\gamma \in V(X_i)$ and $\eta>0$ small enough (this is due to $X_i \subseteq X$ and $X$ having no nodal zones). Then, in the same fashion to Corollary \ref{Cor: intrinsic pancake decomp of a CS}, $\{X_1,\dots,X_{2m}\}$ is a pancake decomposition of $X$. Therefore, if $X$ and $X'$ are weakly outer Lipschitz equivalent and if $h: X \to X'$ is such a homeomorphism, then $\{h(X_1),\dots,h(X_{2m})\}$ is also a pancake decomposition for $h(X)= X'$. Since, for all $1\le i< j \le 2m$, we have $tord(\gamma_{i},\gamma_{j})\ge\beta$, and equality occurs if, and only if, $i$ and $j$ have different parities, there are $\eta, \delta' >0$ small enough such that $\mathcal{H}_{\eta,\beta+ \delta'}(h(\gamma_0))$ consists of $m'$ LNE $(\beta+ \delta')$-H\"older triangles, each one containing at most one arc $h(\gamma_{2i})$, because $h$ is a weakly outer bi-Lipschitz map. This implies $m' \ge m$. Analogously, we obtain  $m \ge m'$ and hence $m=m'$.

Suppose now that $m=m'$, let $\gamma_1, \dots, \gamma_{2m-1}, \gamma_{2m}=\gamma_0 \in V(X)$ and $X_i=T(\gamma_{i-1}, \gamma_i) \subset X$ be as before, and define in a similar way $\gamma_1', \dots, \gamma_{2m-1}', \gamma_{2m}'= \gamma_0' \in V(X')$ and $ X_i'=T(\gamma_{i-1}', \gamma_i') \subset  X'$. We have

\begin{itemize}
	\item  $tord(\gamma_{2j-1},\gamma_{2i})=tord(\gamma_{2j-1}',\gamma_{2i}')=\beta$, for all $1\le i < j \le m$;
	\item $tord(\gamma_{2j-\epsilon},\gamma_{2i-\epsilon}), \, tord(\gamma_{2j-\epsilon}', \gamma_{2i-\epsilon}')>\beta$, for all $1\le i, j \le m$ and $\epsilon \in \{0,1\};$
	\item $\{X_1,\dots,X_{2m}\}$ is a pancake decomposition for $X$; 
	\item $\{ X_1',\dots,X_{2m}'\}$ is a pancake decomposition for $X'$. 
\end{itemize}

For all $1\le i, j \le m$, with $i\ne j$, and for all $\epsilon \in\{0,-1\}$, we also have $tord(\gamma, X_{2i+\epsilon})>\beta$, for each arc $\gamma \in V(X_{2j+\varepsilon})$, because otherwise $X_{2j+\varepsilon}$ and $X_{2i+\varepsilon}$ would determine a nodal zone in $X$, which is a contradiction. Similarly, $tord(\gamma', X_{2i+\epsilon})>\beta$, for all $\gamma \in V(X_{2j+\varepsilon}')$. Therefore, by Proposition 2.20 of \cite{GabrielovSouza}, for all $i=1,\dots,n$, there are $\beta$-weak homeomorphisms 
$$h_{2i-1}: X_{2i-1}\to X_1 \, ; \, h_{2i}: X_{2i}\to X_{2m},$$
$$h_{2i-1}': X_{2i-1}'\to  X_1' \, ; \, h_{2i}': X_{2i}'\to X_{2m}'$$
such that $$h_{2i-1}(\gamma_{2i-2})=\gamma_0 \, , \, h_{2i-1}(\gamma_{2i-1})=\gamma_1 \, , \, h_{2i}(\gamma_{2i-1})=\gamma_{2m-1}\, , \, h_{2i}(\gamma_{2i})=\gamma_{0};$$
$$h_{2i-1}'(\gamma_{2i-2}')=\gamma_0' \, , \, h_{2i-1}(\gamma_{2i-1}')= \gamma_1' \, , \, h_{2i}(\gamma_{2i-1}')= \gamma_{2m-1}'\, , \, h_{2i}(\gamma_{2i}')= \gamma_{0}'.$$

Therefore, the maps $h: X \to X_{2m}\cup X_1$ and $h': X' \to X_{2m}'\cup X_1'$ defined as $h(x)=h_i(x)$, if $x\in X_i$, and $h'(x)= h'_i(x)$, if $x\in X_i'$ shows us that, in order to prove that $X$ and $X'$ are $\beta$-weakly bi-Lipschitz equivalent, it is enough to show that $X_{2m} \cup X_1$ and $X_{2m}' \cup X_1'$ are $\beta$-weakly bi-Lipschitz equivalent. However, this is clearly true, since for $m>1$, $X_{2m} \cup X_1$ and $X'_{2m} \cup X'_1$ are two bubble $\beta$-snakes and, for $m=1$, $X_{2m} \cup X_1$ and $X'_{2m} \cup X'_1$ are two LNE $\beta$-horns. The result then follows.
\end{proof}

\begin{Cor}
A $\beta$-sake is LNE if, and only if, it has no nodal zones and its multiplicity is 1.
\end{Cor}

As an interesting consequence of Theorem \ref{Teo:Multiplicity without Nodes} and the techniques in its proof, we give a reinterpreted proof of Theorem 2.2 in \cite{Fernandes2003}. We prove more directly that the multiplicities of outer bi-Lipschitz equivalent plane complex curve germs are equal, and obtain equality on the Puiseux characteristic pairs by counting the numbers of H\"older triangles inside each corresponding $\mathcal{H}_{1,\alpha}(\gamma)$. This is done in an analogous way as the original proof of \cite{Fernandes2003} did with the so-called test arcs. The following proposition shows that such irreducible complex curves can be seen as circular snakes.

\begin{Prop}
    Every irreducible analytic curve in $\mathbb{C}^2$ is a circular snake with no nodal zones. Moreover, the multiplicity, in the sense of circular snakes, is equal to the multiplicity, in the sense of complex curves, of $C$.
\end{Prop}
\begin{proof}
    Let us denote by $C$ an irreducible analytic curve in $\mathbb{C}^2$ with multiplicity $m$. First, notice that $C$ is homeomorphic to $\mathbb{C}^1$, and thus its link is homeomorphic to $\mathbb{S}^1$ (see Proposition 5 in Section 8.3 of \cite{Brieskorn}). Moreover, $C$ has an isolated singularity at $0$. Hence, any arc in $V(C)$ is Lipschitz non-singular. 
    
    We will prove that $C$ is a snake without nodal zones. If $C$ is LNE ($m=1$), there is nothing to prove. Suppose now that $C$ is not LNE ($m\ge 2$). Taking an analytic change of coordinates, if necessary, we can assume that $C$ is in the following normal form:
    $$C=\{(t^m,y(t)) \mid t\in \C \, ; \, y(t)=t^n+o(t^n) \in \C\{t\} \, ; \, n> m\}.$$
    Let $\gamma \in V(C)$ be an arc and suppose that $\gamma(0)=0$ and $\gamma(r)=(re^{i\theta(r)},y(r^{\frac{1}{m}}e^{\frac{i\theta(r)}{m}}))$, for some continuous function $\theta: [0,\varepsilon)\to \R$ and every $r\in(0,\varepsilon)$ (here, $\varepsilon>0$ is small enough). Consider the arcs $\lambda, \lambda'$ given as $\lambda(r)=(re^{i(\theta(r)+\pi)},y(r^{\frac{1}{m}}e^{\frac{i(\theta(r)+\pi)}{m}}))$ and $\lambda'(r)=(re^{i(\theta(r)-\pi)},y(r^{\frac{1}{m}}e^{\frac{i(\theta(r)-\pi)}{m}}))$. Let also $T=T(\lambda,\gamma)$, $T'=T(\lambda',\gamma)$ be the only $1$-H\"older triangles in $X$ such that $T\cap T'=\gamma$. Notice that
    $$T=\{(re^{i\theta},y(r^{\frac{1}{m}}e^{\frac{i\theta}{m}})) \mid r\ge 0 \, ; \, \theta(r)\le \theta \le \theta(r)+\pi\},$$
    and since $\frac{\partial}{\partial(t^m)}(y(t))=\frac{n}{m}t^{n-m}+o(t^{n-m})$ is bounded, the projection on the first coordinate $\pi: T \to \C$ is a bi-Lipchitz homomorphism into its image. The set 
    $$\pi(T)=\{re^{i\theta} \mid r\ge 0 \, ; \, \theta(r) \le \theta \le \theta(r)+\pi\}$$ is clearly a LNE $1$-H\"older triangle, and so $T$ is also a LNE $1$-H\"older triangle. Similarly, $T'$ is a LNE $1$-H\"older triangle. Now, as $$\|\lambda(r)-\lambda'(r)\|=\left\|y(r^{\frac{1}{m}}e^{\frac{i(\theta(r)+\pi)}{m}})-y(r^{\frac{1}{m}}e^{\frac{i(\theta(r)-\pi)}{m}})\right\|=r^{\frac{n}{m}}+o(r^{\frac{n}{m}}),$$
    we have $tord(\lambda,\lambda')>1$ and hence $\gamma$ is an abnormal arc. Therefore, $C$ is a circular snake.

    Finally, for $k=0,1,\dots,m-1$, consider the disjoint $1$-H\"older triangles
    $$T_k=\left\{(re^{i\theta},y(r^{\frac{1}{m}}e^{\frac{i\theta}{m}})) \mid r\ge 0 \, ; \, \theta(r)-\frac{\pi}{2}\le \theta-2k\pi \le \theta(r)+\frac{\pi}{2}\right\},$$
    and the arcs $\gamma_k=\gamma(r)=(re^{i(\theta(r)+2k\pi)},y(r^{\frac{1}{m}}e^{\frac{i(\theta(r)+2k\pi)}{m}}))$ (note that $\gamma_k \in T_k$ and $\gamma_0=\gamma$).
    Similarly to what was done in the previous paragraph, each $T_k$ is LNE and $tord(\gamma,\gamma_k)>1$. This implies that $X\setminus\{\cup_{0\le k\le m-1}T_k\}$ is a union of $m$ $1$-H\"older triangles, whose tangency order with $\gamma$ is equal to 1. Therefore, for every $a>0$ small enough, $\mathcal{H}C_{a,1}(\gamma)$ consists of $m$ $1$-H\"older triangles $T'_k \subset T_k$, each one of then containing $\gamma_k$ and hence $m_C(\gamma)=m$. Since $\gamma\in V(C)$ is arbitrary, we conclude that every zone is a constant zone (with the multiplicity equal to $m$), and thus $C$ has no nodal zones. 
\end{proof}

\begin{Teo}
Let $C, C'$ be two germs of analytic, irreducible curve germs, at $0$, in $\mathbb{C}^2$. If $C$ and $C'$ are outer bi-Lipschitz equivalent, then their corresponding Puiseux characteristic pairs are equal.
\end{Teo}

\begin{proof}
Let $\varphi: C \to C'$ be an outer $L$-bi-Lipschitz map, $\gamma \in V(C)$ and arc and let $(m_1,n_1),\dots,(m_g,n_g)$ be the Puiseux characteristic pairs of $C$ and  $(p_1,q_1),\dots,(p_{g'},q_{g'})$ be the Puiseux characteristic pairs of $C'$. By Theorem \ref{Teo:Multiplicity without Nodes}, the multiplicity $m=n_1\dots n_g$ of $C$ and the multiplicity $m'=q_1\dots q_{g'}$ of $C'$ are equal.
	
	If $m_g > p_{g'}$, then let $\alpha \in (\frac{p_{g'}}{q_1\dots q_{g'}},\frac{m_g}{n_1\dots n_{g}})$ such that $\alpha>\frac{m_{g-1}}{n_1\dots n_{g-1}}$. Since $\frac{m_{g-1}}{n_1\dots n_{g-1}}<\alpha<\frac{m_g}{n_1\dots n_g}$, we have that $\mathcal{H}C_{1,\alpha}(\gamma)$ is a union of $n_g$ $\alpha$-H\"older triangles whose intersection is $\{0\}$, and thus $\mathcal{H}C_{1,\alpha}(\gamma)$ is also a union of $n_g$ $\alpha$-H\"older triangles whose intersection is $\{0\}$. On the other hand, $\varphi(\mathcal{H}C_{1,\alpha}(\gamma)) \subseteq \mathcal{H} C'_{L,\alpha}(\varphi(\gamma))$, and since $\alpha>\frac{p_{g'}}{q_1 \dots q_{g'}}$, we have that $\mathcal{H} C'_{L,\alpha}(\varphi(\gamma))$ is only an $\alpha$-H\"older triangle, a contradiction. Similarly, $m_g < p_{g'}$ is impossible and hence $m_g=p_{g'}$. 
	
	Now, for $\alpha = \frac{m_g}{n_1\dots n_g}=\frac{p_{g'}}{q_1\dots q_{g'}}$, and for $\eta>0$ large enough, we have that $\mathcal{H} C_{\eta,\alpha}(\gamma)$ is the union of $n_g$ $\alpha$-H\"older triangles whose intersection is $\{0\}$, and $\mathcal{H} C'_{L\eta,\alpha}(\varphi(\gamma))$ is the union of  $q_{g'}$ $\alpha$-H\"older triangles whose intersection is $\{0\}$. Since $\varphi$ is a bi-Lipschitz map and $\mathcal{H} C_{\eta,\alpha}(\gamma) \subseteq \mathcal{H} C'_{L\eta,\alpha}(\varphi(\gamma))$, we conclude that $n_g \le q_{g'}$. Working with $\varphi^{-1}$, we have $q_{g'} \le n_g$ and thus $n_g=q_{g'}$.
	
	If $m_{g-1} > p_{g'-1}$, then let $\alpha \in (\frac{p_{g'-1}}{q_1\dots q_{g'-1}},\frac{m_{g-1}}{n_1\dots n_{g-1}})$ such that $\alpha>\frac{m_{g-2}}{n_1\dots n_{g-2}}$. Since $\frac{m_{g-2}}{n_1\dots n_{g-2}}<\alpha<\frac{m_{g-1}}{n_1\dots n_{g-1}}$, we have that $\mathcal{H} C_{1,\alpha}(\gamma)$ is a union of $n_gn_{g-1}$ $\alpha$-H\"older triangles whose intersection is $\{0\}$, and thus $\varphi(\mathcal{H} C_{1,\alpha}(\gamma))$ is also a union of $n_gn_{g-1}$ $\alpha$-H\"older triangles whose intersection is $\{0\}$. On the other hand, $\varphi(\mathcal{H}C_{1,\alpha}(\gamma)) \subseteq \mathcal{H} C'_{L,\alpha}(\varphi(\gamma))$, and since $\alpha>\frac{p_{g'-1}}{q_1\dots q_{g'-1}}$, we have that $\mathcal{H} C'_{L,\alpha}(\varphi(\gamma))$ is only an $\alpha$-H\'older triangle, a contradiction. Similarly, $m_{g-1} < p_{g'-1}$ is impossible and hence $m_{g-1}=p_{g'-1}$. 
	
	Now, for $\alpha = \frac{m_{g-1}}{n_1\dots n_{g-1}}=\frac{p_{g'-1}}{q_1\dots q_{g'-1}}$, and for $\eta>0$ large enough, we have that $\mathcal{H} C_{\eta,\alpha}(\gamma)$ is the union of $n_gn_{g-1}$ $\alpha$-H\"older triangles whose intersection is $\{0\}$, and $\mathcal{H} C'_{L\eta,\alpha}(\varphi(\gamma))$ is the union of  $q_{g'}q_{g'-1}$ $\alpha$-H\"older triangles whose intersection is $\{0\}$. Since $\varphi$ is a bi-Lipschitz map and $\varphi(\mathcal{H} C_{\eta,\alpha}(\gamma)) \subseteq \mathcal{H} C'_{L\eta,\alpha}(\varphi(\gamma))$, we conclude that $n_gn_{g-1} \le q_{g'}q_{g'-1}$. Working with $\varphi^{-1}$, we have $q_{g'}q_{g'-1} \le n_gn_{g-1}$ and thus $n_gn_{g-1}=q_{g'}q_{g'-1}$.
	
	Doing this procedure successively, we obtain $m_{g-k}=p_{g-k}$ and $n_gn_{g-1}\dots n_{g-k}=q_{g'}q_{g'-1}\dots q_{g'-k}$, for every $0\le k \le {\rm min}\{g,g'\}$. This implies immediately $g=g'$ and $(m_k,n_k)=(p_k,q_k)$, for every $1\le k \le g$. The theorem follows.
\end{proof}

\section{Binary circular snakes and their names}

In this last section, we consider binary circular snakes. Since any snake name can be reduced to a binary one, they are crucial in the combinatorial classification of circular snakes.

\begin{Def}
    A \textbf{binary circular snake name} is a circular snake name $W=[w_1 \dots w_{m-1}w_m]$ such that each letter of $W$ appears exactly twice in $[w_1,\dots,w_{m-1}]$. The equivalence class $[[W]]$ is \textbf{binary} if $W^j$ is a binary circular snake name, for some $1\le j\le m-1$. 
\end{Def}

\begin{remark}
    It is easy to see that if $W$ is a binary circular snake name, then $W^j$ is also a binary circular snake name, for every $1\le j\le m-1$. Therefore, the binary property of $[[W]]$ does not depend on $j$.
\end{remark}

\begin{Def}
    A circular snake $X$ is a \textbf{binary circular snake} if $[[W_N]]$ is binary.  Alternatively, a circular snake $X$ is binary if each one of its nodes contains exactly two nodal zones.
\end{Def}

\begin{Def}\label{operation}
    Let $W=[w_1\dots w_m]$ be a circular name and let $x$ be a letter of $W$. We define the \textbf{binary reduction} of $W$ with respect to $x$ as the circular word $W_x$ defined as follows.
    \begin{enumerate}
    \item If $x\ne w_1$ and $x$ appears $p>2$ times in $W$, and if $W=[w_1X_0xX_1x\cdots xX_{p-1}xX_{p}w_m]$, then we consider $p-1$ distinct new letters $x_1,...,x_{p-1}$ that didn't appear on $W$, and define $W_x$ as $$W_x =[w_1X_0x_1X_1x_1x_2X_2x_2x_3\cdots x_{p-2}x_{p-1}X_{p-1}x_{p-1}X_{p}w_m].$$
    \item If $x = w_1$ and $x$ appears $p>3$ times in $W$, and if $W=[xX_0xX_1x\cdots xX_{p-3}xX_{p-2}x]$, then we consider $p-2$ distinct new letters $x_1,...,x_{p-2}$ that didn't appear on $W$, and define $W_x$ as 
    $$W_x =[x_1X_0x_1x_2X_1x_2x_3X_2x_2x_3\cdots x_{p-3}x_{p-2}X_{p-3}x_{p-2}X_{p-2}x_1].$$
    \end{enumerate}
\end{Def}

\begin{Prop}
If $W$ is a circular snake name, then the circular word $W_x$ as in Definition \ref{operation} is also a circular snake name.
\end{Prop}

\begin{proof}
    Note that $W_x$ is a circular word because each of the new letters $x_i$ in $W_x$ appears exactly two times, except possibly for $x_1$, which appears exactly three times if $x=w_1$. It remains to prove that $W_x$ satisfies condition $(ii)$ of Definition \ref{Def: rules for snakes names}, but such proof is the same as in Proposition 6.32 of \cite{GabrielovSouza}.
\end{proof}

\begin{remark}
    The binary reduction in circular snakes could be geometrically interpreted as splitting a node with more than two nodal zones, in a similar way to the binary reduction in snakes (see Figure 10).
\end{remark}

\tikzset{every picture/.style={line width=0.75pt}} 



\begin{remark}
    Any non-binary circular snake name $W$ could be reduced to a binary circular snake name by applying binary reduction to each letter that appears in $W$ more than two times or to the initial letter if it occurs at least three times. Moreover, if there are several such letters, the resulting binary snake name does not depend on the order of the letters to which binary reduction is applied.
\end{remark}

A natural question regarding binary snakes is ``How many weakly bi-Lipschitz non-equivalent $\beta$-circular snakes with exactly $n+1$ nodes do we have?''. If $a(n)$ is such quantity, then by Theorem \ref{Teo:weak equivalence} we have that $a(n)$ is the number of non-equivalent binary circular snake names of length $2n+3$, up to the circular orientation equivalence, that is, a equivalence relation $\sim$ in the set of binary circular snake names such that, for every such word $W$, we have $W \sim W^j$, for each $1\le j\le m-1$ ($m$ is the length of $W$) and $W \sim -W$.

As we saw before, the definitions and propositions regarding binary circular snakes are very similar to their analogous for binary snakes. So, it is expected that $a_n$ should be the $n^{th}$ Catalan Number, as in Corollary 6.63 of \cite{GabrielovSouza}. Surprisingly, that's true for $n=1,2,3$, but that's false for $n\ge 4$. Indeed, we have
\begin{itemize}
    \item $a(1)=1$, since the only nonequivalent circular snake name with 5 letters is $[ababa]$.
    \item $a(2)=2$, since the only nonequivalent circular snake names with 7 letters are $[abcabca], \, [abcbaca]$.
    \item $a(3)=5$, since the only nonequivalent circular snake names with 9 letters are 
    $$[abcdabcda], \, [abcdadcba], \,[abcdacdba], \,[abcdacbda], \,[abcadcbda].$$
    \item $a(4)=16$, since the only nonequivalent circular snake names with 11 letters are 
    $$[abcdeabcdea], \,[abcdaebcdea], \,[abcdaebdcea], \,[abcadecbdea],$$
    $$[abcadecbeda], \,[abcadebcdea], \,[abcadebceda], \,[abcadcedbea],$$
    $$[abcadbecdea], \,[abcadbedcea], \,[abcabdecdea], \,[abacdebceda],$$
    $$[abacdebecda], \,[abacdebcdea], \,[abacdebdeca], \,[abacdebedca].$$
\end{itemize}

We believe that $a(5)=64$ (in fact, we found 64 nonequivalent circular snake names with 13 letters, but we're unsure if those are the only ones) and, in general, $a(n)=\sum_{k=0}^{n-1}k!(n-1-k)!$, but we don't have a proof. So, we propose the following question.

\begin{prob}
    Find a general formula for $a(n)$, the number of nonequivalent binary circular snake names associated with a circular snake with exactly $n+1$ nodal zones, up to the orientation equivalence.
\end{prob}

\end{document}